\documentclass{amsart}
\usepackage[usenames,dvipsnames]{color}
\usepackage[all]{xy}
\usepackage[OT2,T1]{fontenc}
\DeclareSymbolFont{cyrletters}{OT2}{wncyr}{m}{n}
\DeclareMathSymbol{\Sha}{\mathalpha}{cyrletters}{"58}
\usepackage{setspace}
\usepackage{amsfonts}
\usepackage{amsmath}
\usepackage{amsthm}
\usepackage{amssymb}
\usepackage{rotating, tabularx}
\usepackage{appendix}
\usepackage{tikz}
\usetikzlibrary{matrix,arrows}
\usepackage[colorlinks]{hyperref}

\theoremstyle{plain}
\newtheorem{Theorem}{Theorem}[section]
\newtheorem{Lemma}[Theorem]{Lemma}
\newtheorem*{Lemma*}{Lemma}
\newtheorem{Proposition}[Theorem]{Proposition}

\newtheorem{algorithm}[Theorem]{Algorithm}

\newtheorem{Corollary}[Theorem]{Corollary}

\theoremstyle{definition}
\newtheorem{Definition}[Theorem]{Definition}

\theoremstyle{remark}
\newtheorem{Remark}[Theorem]{Remark}

\DeclareMathOperator{\st}{st}
\DeclareMathOperator{\rep}{Rep}

\DeclareMathOperator{\Pic}{Pic}
\DeclareMathOperator{\Sym}{Sym}
\DeclareMathOperator{\Spec}{Spec}

\DeclareMathOperator{\Sel}{Sel}
\DeclareMathOperator{\cry}{cr}
\DeclareMathOperator{\im}{im}
\DeclareMathOperator{\Hom}{Hom}
\DeclareMathOperator{\ext}{Ext}

\DeclareMathOperator{\loc}{loc}

\DeclareMathOperator{\Res}{Res}

\DeclareMathOperator{\Gal}{Gal}

\DeclareMathOperator{\aut}{Aut}

\DeclareMathOperator{\rk}{rk}

\DeclareMathOperator{\Aut}{Aut}

\DeclareMathOperator{\cl}{cl}
\DeclareMathOperator{\Ker}{Ker}
\DeclareMathOperator{\Coker}{Coker}
\DeclareMathOperator{\Div}{Div}
\DeclareMathOperator{\ab}{ab}
\newcommand{\Q}{\mathbb{Q}}
\newcommand{\Z}{\mathbb{Z}}
\renewcommand{\div}{\textrm{div}}
\newcommand{\dR}{\textrm{dR}}
\newcommand{\F}{\mathbb{F}}

\newcommand{\bom}{\bar{\omega}}
\newcommand{\cO}{\mathcal{O}}
\newcommand{\fp}{\mathfrak{p}}
\DeclareMathOperator{\NS}{NS}

\begin{document}

\begin{abstract}
We give the first explicit examples beyond the  Chabauty-Coleman method where Kim's nonabelian Chabauty program determines the set of rational points 
of a curve defined over $\Q$ or a quadratic number field.  We accomplish this by studying the role of $p$-adic heights in explicit nonabelian Chabauty. 

\end{abstract}

\title[Quadratic Chabauty and rational points I]{Quadratic Chabauty and rational points I:\\ $p$-adic heights}
\author{Jennifer S. Balakrishnan}
\address{Jennifer S. Balakrishnan, Department of Mathematics and Statistics, Boston University, 111 Cummington Mall, Boston, MA 02215, USA}
\email{jbala@bu.edu}
\author{Netan Dogra}
\address{Netan Dogra, Department of Mathematics, Imperial College London, London SW7 2AZ, UK }
\email{n.dogra@imperial.ac.uk}
\date{\today}

\subjclass[2010]{Primary 	14G05, 11G50; Secondary 14G40}

\maketitle

\vspace{-.2in}

\tableofcontents

\vspace{-.15in}

\section{Introduction}\label{sec:intro}
Let $X$ be a smooth projective curve of genus $g>1$ defined over a number field $K$.  By Faltings' celebrated work on the Mordell conjecture, the set of $K$-rational points on $X$, denoted $X(K)$, is known to be finite \cite{faltings-finiteness}. However, the method of proof is not constructive and does not produce the set $X(K)$.  Nevertheless, in certain cases, it is possible to compute $X(K)$; perhaps the most widely applicable technique is the $p$-adic method of Chabauty and Coleman. 

The Chabauty-Coleman method imposes linear conditions on the Jacobian of $X$, and in an essential way, requires that the Mordell-Weil rank of the Jacobian is less than $g$.  Kim has proposed that one can lift this restriction on the rank by replacing the Jacobian of $X$ with a larger object, the \emph{Selmer variety}, which captures more refined information about the \'etale topology of $X$. In this paper, we discuss new techniques for studying Selmer varieties, which we translate into methods for determining the set $X(K)$ in a number of new cases. In particular, we study curves whose Jacobians have Mordell-Weil rank equal to  $g$ and  give the first examples beyond the  Chabauty-Coleman method where Kim's nonabelian Chabauty program can be used to precisely determine the set of rational points of a curve defined over $\Q$ or a quadratic number field.
 
To give some context for our results, let us begin by recalling the Chabauty-Coleman method.  Let $p$ be a prime of good reduction for $X$, let $\mathfrak{p}$ be a prime above $p$, and let $J$ denote the Jacobian of $X$. Let 
$$
\log_J :J(K_{\mathfrak{p}})\to H^0 (X_{K_\mathfrak{p}},\Omega ^1 )^*
$$
be the $\mathfrak{p}$-adic logarithm map for the abelian variety $J$, where $X_{K_\mathfrak{p}}$ denotes the base change of $X$ to ${K_\mathfrak{p}}$.
Suppose that $X(K) \neq \emptyset$, and for convenience, that we know one point $b$ in $X(K)$. If the 
Mordell-Weil rank $r = \rk J(K)$ is less than $g$,  the method of Chabauty \cite{chabauty} produces a finite set of $\mathfrak{p}$-adic points on $X$, which we shall denote $X(K_{\mathfrak{p}} )_1$, and we have $$X(K_{\mathfrak{p}}) \supset X(K_{\mathfrak{p}} )_1 \supset X(K).$$
Following Coleman \cite{coleman:chabauty}, the set $X(K_{\mathfrak{p}} )_1 $ may be interpreted as the zeros of a $p$-adic path integral
$$
X(K_{\mathfrak{p}} )_1 =\left\{z \in X(K_{\mathfrak{p}}) :  \int ^z _b \omega =0\right\}
$$
for some differential $\omega $ in $H^0 (X_{K_{\mathfrak{p}} },\Omega ^1)$. By further interpreting this $p$-adic path integral as a $p$-adic power series and solving for its zeros, one can often effectively compute $X(K_{\mathfrak{p}} )_1$ (subject to the usual issues with inexact computation and $p$-adic precision) and in practice, one can often recover $X(K)$.  This is known as the \emph{Chabauty-Coleman method}.

The Chabauty-Coleman method requires that the Mordell-Weil rank of the Jacobian be less than the genus of the curve, which is somewhat restrictive.  As such, one would like to have a
refinement of the Jacobian which remembers more information about the set $X(K)$. The insight of Kim \cite{kim:siegel} is that, rather than trying to generalise the Jacobian of $X$, it is easier to generalise its Galois cohomological avatar: the Selmer group. In \cite{kim:chabauty}, Kim defined  a family of \textit{Selmer varieties} $\Sel (U_n )$ giving a decreasing sequence of subsets  \cite{balakrishnan2012non}
$$
X(K_{\mathfrak{p}} )_1 \supset X(K_{\mathfrak{p}} )_2 \supset \ldots
$$
of $X(K_{\mathfrak{p}} )_1$, which can be computed in terms of \emph{iterated} $p$-adic path integrals. The sets $X(K_{\mathfrak{p}} )_n$ contain $X(K)$, so by proving finiteness of $X(K_{\mathfrak{p}} )_n$ and explicitly computing it, one can hope to recover $X(K)$.  We refer to this as \emph{nonabelian Chabauty} or the \emph{Chabauty-Kim method}. Note that when $K=\mathbb{Q}$, conjectures of Bloch and Kato imply that $X(\mathbb{Q}_p )_n $ is finite for $n$ sufficiently large \cite{kim:chabauty}.  

However, at present, there are few examples of curves $X$  where $X(K_{\mathfrak{p}} )_n$ has been used to give more information than $X(K_{\mathfrak{p}} )_1$.  Coates and Kim \cite{kim:coates} proved that when $X/\Q$ is a curve whose Jacobian is isogenous to a product of CM abelian varieties, for $n$ sufficiently large, $X(\mathbb{Q}_p )_n $ is finite. Recently, Ellenberg and Hast \cite{ellenberg2017rational} used this to give a new proof of finiteness of $X(\Q)$ of any solvable Galois cover $X$ of $\mathbb{P}^1$ (which, for instance, includes the class of superelliptic curves). Even in these cases, it is not clear how to actually compute $X(\mathbb{Q}_p )_n $.

In this paper, we give techniques to compute rational points on curves in some cases beyond the scope of Chabauty-Coleman, by computing finite sets containing $X(K_{\mathfrak{p}} )_2$. The methods used are a generalisation of those employed to study integral points on hyperelliptic curves using $p$-adic heights \cite{balakrishnan2013p}, combined with new methods for relating unipotent path torsors to $p$-adic heights \cite{dogra:thesis}.

In \cite{balakrishnan2013p}, one works with a hyperelliptic curve $X/\Q$ of genus $g$ with a model
\begin{equation}\label{eqn1}y^2 =f(x) = x^{2g+1}+a_{2g}x^{2g} + \cdots + a_0, \qquad a_i \in \Z.\end{equation}
Let $T_0$ denote the set of primes of bad reduction for this model and let $p$ be a prime of good reduction.
Let $Y = \Spec(\Z[x,y]/(y^2 - f(x)))$, so that $Y(\Z)$ denotes the set of integral solutions to \eqref{eqn1}, and let $\infty $ denote the point at infinity. Using $p$-adic heights, one can compute a finite set of points containing $Y(\Z)$:
\begin{Theorem}[Quadratic Chabauty for integral points \cite{balakrishnan2013p}]\label{thm:oldqc}
Let $X/\Q$ be a genus $g$ hyperelliptic curve as in \eqref{eqn1}. Let $\Omega \subset \mathbb{Q}_p $ be the explicitly computable, finite set of values taken by the sum of the Coleman-Gross local heights
$$
-\sum _{v\in T_0 }h_v (z_v -\infty ),$$
for $(z_v )$ in $\prod _{v\in T_0 }Y(\Z_v )$.   Suppose that $r =g$. Then there is an explicitly computable symmetric bilinear map$$
B:H^0 (X_{\Q_p},\Omega ^1 )^* \times H^0 (X_{\Q_p},\Omega ^1 )^* \to \mathbb{Q}_p 
$$
 such that $Y(\mathbb{Z})\subset Y(\mathbb{Z}_p )$ is contained inside the 
finite set of solutions to 
$$
h_p(z-\infty )+B(\log _J (z-\infty ),\log _J (z-\infty ))\in \Omega.$$ 
\end{Theorem}

In the present work, we give a generalisation of this theorem which allows us to study \emph{rational} points on curves in some cases where the Mordell-Weil rank is not less than the genus.  

To state our results more precisely, we fix some notation. Let $K$ be $\Q$ or an imaginary quadratic field, and let $X/K$ be a smooth projective curve of genus $g>1$ with a $K$-rational point $b$.  Let $T_0$ be the set of primes of bad reduction for $X$, let $p$ be a prime of $\mathbb{Q}$ such that $\{ v|p \} \cap T_0 $ is empty, and let $T=T_0 \cup \{ v|p \}$. Let $\rho (J)=\rk \NS (J)$ denote the Picard number of $J$ (over $K$, not necessarily its geometric Picard number). The starting point for generalising Theorem \ref{thm:oldqc} is the following lemma, which may be of independent interest:
\begin{Lemma*}[Lemma \ref{nondensityiseasy}]
If $r<g+ \rho(J)-1$, then $X(K _\frak{p} )_2 $ is finite.
\end{Lemma*} 

To explain the proof of this key lemma, we first recall the set-up of Kim's nonabelian Chabauty method in Section \ref{sec:chabautykim}. Once this foundational material is recalled, the proof (in Section \ref{sec:nond}) is entirely elementary and essentially just uses the crystalline version of the Kummer isomorphism.

We further describe cases where we can describe $X(K _\frak{p})_2$ more explicitly. The main example we consider is the situation when the rank of $J$ is $g$  and $\rho(J)$ is greater than 1.  For a result applying when the rank is greater than the genus, see Proposition \ref{prop2}. 

To state our next results, we set a bit more notation. Let $\overline{X}:=X\times _{K}\overline{K}$, and let $i _{\Delta }:X\hookrightarrow X\times X$ denote the diagonal morphism, with image $\Delta =\Delta _X :=i_{\Delta }(X)$. For a codimension $d$ cycle $Z$ in a variety $W$ we denote by $\cl _Z$ the cycle class map $\Q _p (-d)\to H^{2d}(\overline{W},\Q _p )$, and denote the support of $Z$ by $|Z|$. By our assumptions on the Picard number, there is a codimension 1 cycle $Z$ in 
$X\times X$ such that the composite map 
$$
\mathbb{Q}_p (-1)\to H^2 _{\acute{e}t}(\overline{X}\times \overline{X},\mathbb{Q}_p )\to H^1  _{\acute{e}t}(\overline{X},\mathbb{Q}_p )\otimes H^1  _{\acute{e}t}(\overline{X},\mathbb{Q}_p )\to \wedge ^2 H^1  _{\acute{e}t}(\overline{X},\mathbb{Q}_p )$$  is nonzero (where the maps are, from left to right, the cycle class map $\cl _Z$, the K\"unneth projector, and the antisymmetric projection), and such that the intersection number of $Z$ with $\Delta -X\times P_1 -P_2 \times X $ is zero, where $P_1 $ and $P_2 $ are any points on $X$. For distinct points $b$ and $z$ in $X$ not contained in $i_{\Delta } ^{-1}(|Z|)$, we associate a cycle 
$D(b,z)\in \Div^0(X)$ to the triple $(b,z,Z)$ (see Definition \ref{defnofdbz}). The theorem below is inspired by a theorem of Darmon, Rotger, and Sols relating the class of $D(b,z)$ in $J(\mathbb{C})$ to iterated integrals \cite[Theorem 1]{darmon2012iterated} (see \S 6.4).
\begin{Theorem}[Quadratic Chabauty for rational points]\label{biellipticformula} 
Let $X/K$ be a smooth projective curve of genus $g>1$. Let $b \in X(K)$ be a fixed basepoint and let $z, Z,$ and  $D(b,z)$ be as above. Let $X' :=X-i_{\Delta} ^{-1}(|Z|)$. \\
(i): For each $v$ prime to $p$, the local height function
$h_{v}(z-b,D(b,z))$ takes only finitely many values for $z$ in $X' (K_v )$. If $v$ is  a prime of potential good reduction, then $h_{v}(z-b,D(b,z))$ is identically zero. \\
(ii): Suppose $r= g$, $\rho(J) > 1$, and the $p$-adic closure $\overline{J(K)}$ has finite index in $J(K_{\mathfrak{p}})$.  Let $\Omega \subset {K}_{\mathfrak{p}} $ be the finite set of values taken by the sum of local heights
$$
-\sum _{v\nmid p}h_{v}(z_v -b,D(b,z_v ))
$$
for $(z_v )$ in $\prod _{v\nmid p}X' (K_v )$. Then
there is a symmetric bilinear map
$$
B:H^0 (X_{K_{\mathfrak{p}}},\Omega ^1 )^* \times H^0 (X_{K_{\mathfrak{p}}},\Omega ^1 )^* \to \mathbb{Q}_p 
$$ such that the set of $z$ in $X'(K_{\mathfrak{p}} )$ for which
$$
h_{\mathfrak{p}}(z-b,D(b,z))-B(\log _J (z-b),\log _J (D(b,z))) \in \Omega
$$
is finite and contains $X(K_{\mathfrak{p}})_2 \cap X' (K_{\mathfrak{p}})$.
\end{Theorem}
\begin{Remark}
If $A$ is a simple abelian variety, it is a conjecture of Waldschmidt that the condition that $\overline{A(K)}$ has finite index in $A(K_{\mathfrak{p}})$ will be satisfied whenever the rank is equal to the dimension \cite[Conjecture 1]{waldschmidt2011p}.
\end{Remark}

Note that, although Theorem \ref{thm:oldqc} and \ref{biellipticformula} are both statements about relations between $p$-adic heights and single integrals which are only valid away from a finite set of points, Theorem \ref{biellipticformula} produces a polynomial in $p$-adic heights and single integrals that takes only finitely many values on $X(K)$ (away from this finite set), whereas in Theorem \ref{thm:oldqc}, we obtain a polynomial in $p$-adic heights and single integrals that takes only finitely many values when restricted to integral points.
The key difference which allows one to prove things about rational rather than integral points is that the local height $h_v (z-b ,D(b,z))$ takes only finitely many values. From the point of view of nonabelian Chabauty, the difference is that Theorem \ref{thm:oldqc} genuinely uses Kim's method applied to a quotient of the fundamental group of $Y_{\overline{\mathbb{Q}}}$, whereas the proof of Theorem \ref{biellipticformula} applies Kim's method to the fundamental group of $X$, and expresses the formula in terms of a height pairing via an auxiliary choice of a correspondence $Z$. Note that, by the Moving lemma \cite[$\S$ 11.4]{fulton2013intersection}, given any $z\in X(\overline{\Q})$, and any $Z$ as above, we can choose a rationally equivalent cycle $Z'$ with the property that $Z'$ intersects $\Delta  +b \times X +X\times z$ properly, and does not contain the points $(b,b)$ or $(z,z)$, and hence with the property that $b$ and $z$ are points of $X-i_{\Delta } ^{-1}(|Z'|)$.

The proof of Theorem \ref{biellipticformula} may be used to prove an analogue for integral points on an affine curve (see Remark \ref{affine_version}). The only differences are that there is no condition on $\rho (J_{\Q })$, and the intersection number $Z.(\Delta -X\times P_1 -P_2 \times X)$ is no longer required to be zero. In Lemma \ref{asintheproof} we see that this recovers Theorem \ref{thm:oldqc}.

Before we give an overview of the proof of Theorem \ref{biellipticformula}, we briefly sketch Nekov\'a\v r's approach to $p$-adic height pairings, which plays a crucial role in the proof. The construction has two steps: first, one constructs, for all $v$, a local height function $h_v$ on a certain set of equivalence classes of $G_v$-representations, which we refer to  in this paper as \emph{mixed extensions with graded pieces $\Q _p ,H^1 _{\acute{e}t}(\overline{X},\Q _p (1))$ and $\Q _p (1)$}. This construction is explained in detail in Section \ref{sec:mixedandnekovar}, and an interpretation is given in terms of nonabelian cohomology. Second, for any pair of divisors with disjoint support $D_1 ,D_2 $ in $\Div ^0 (X_{\Q _v})$, one associates such a mixed extension, denoted $H_X (D_1 ,D_2 )$; the representation is a subquotient of $H^1 _{\acute{e}t}(\overline{X}-|D_1 |;|D_2 |)$.

The proof of Theorem \ref{biellipticformula} proceeds in two stages. First, we construct a map from the Selmer variety of $X$ to a variety parametrising equivalence classes of mixed extensions as above. The idea is to find a mixed extension, which we denote $A(b)$, on which the pro-unipotent fundamental group acts in a Galois-equivariant way, and then map a torsor $P$ in the Selmer variety to the twist $A(b)^{(P)}$ of $A(b)$ by $P$. This construction is described in detail in Section \ref{sec:pathsandlinear}. As explained in Proposition \ref{prop:vital}, this construction is already enough to provide a nontrivial equation for $X(K _\mathfrak{p} )_2 $, in terms of single integrals and $p$-adic heights of twists of $A(b)$.

To get from Proposition \ref{prop:vital} to Theorem \ref{biellipticformula}, we relate the mixed extensions $A_Z (b,z)$ and $H_X (z-b,D(b,z))$, where $A_Z (b,z)$ denotes the twist of $A(b)$ by the element of the Selmer variety corresponding to $z$. As noted above, the latter is constructed as a subquotient of $H^1 _{\acute{e}t}(\overline{X}-\{z,b\};D(b,z))$. In Section \ref{subsection:Beilinson}, we show that by a theorem of Beilinson, $A_Z (b,z)$ similarly has a cohomological interpretation relating it to the second \'etale cohomology group of $X\times X$ relative to $b\times X \cup \Delta _X \cup X\times z$. Hence the heart of the proof is a rather elaborate diagram chase relating these two \'etale cohomology groups, details of which are in Section \ref{subsection:heart}.

The remainder of the paper is devoted to turning Theorem \ref{biellipticformula}---in certain special cases---into something explicit and computable.  To produce an algorithm using Theorem \ref{biellipticformula} to find a finite set containing $X(K_{\mathfrak{p}})_2$, one needs to compute the cycle $Z$ 
and the local heights $h_{v}$. In this paper we focus on the simplest such example, which we  describe below.

Let $X/K$ be a genus 2 bielliptic curve with affine equation 
\begin{equation}\label{eqn2}y^2 =x^6 +a_4x^4 +a_2x^2 +a_0,\end{equation}
with $a_i \in K$. Flynn and Wetherell \cite{flynn1999finding} previously considered the problem of determining the rational points of $X$.  Let $E_1 $ and $E_2 $ be the elliptic curves over $K$ defined by the equations 
$$E_1: y^2 =x^3 +a_4x^2 +a_2x+a_0 \qquad \qquad E_2: y^2 =x^3 +a_2x^2 +a_4a_0x+a_0^2$$
and let $f_i$ denote the map $X\to E_i $, ($i=1,2$), sending $(x,y)$ to $(x^2 y)$ and $(a_0 x^{-2},a_0 yx^{-3})$ respectively.

Let $h_{E_1 } $ and $h_{E_2 } $ denote the height pairings on $E_1 $ and $E_2 $ corresponding to an idele class character
$
\chi :G_K ^{\ab }\to \mathbb{Q}_p 
$
and an isotropic splitting of the Hodge filtration. In the case when  $K=\mathbb{Q}$, we take $\mathfrak{p}=(p)$ to be a prime of good reduction. In the case when $K$ is an imaginary quadratic extension, we take $p$ to be a prime of $\mathbb{Q}$ which splits as $\mathfrak{p}\overline{\mathfrak{p}}$ in $K$, where $\mathfrak{p}$ and $\overline{\mathfrak{p}}$ are both primes of good reduction, and take $\chi $ to be a character which is trivial on $\mathcal{O}_{\overline{\mathfrak{p}}}^{\times}$.
\begin{Theorem}\label{biellexample} Let $X/K$ be the genus 2 bielliptic curve \eqref{eqn2}.\\
(i): For all $v$ not above $p$, 
$$
h_{E_1 ,v}(f_1 (z))-h_{E_2 ,v}(f_2 (z))-2\chi _v (x(z))
$$ 
takes only finitely many values on $X(K_v)$, and for almost all $v$ it is identically zero. \\
(ii): Let $\Omega $ denote the explicitly computable, finite set of values taken by 
$$
-\sum _{v\nmid p}(h_{E_1 ,v}(f_1 (z_v ))-h_{E_2 ,v}(f_2 (z_v ))-2\chi _v (x(z_v )))
$$
for $(z_v )$ in $\prod _{v\nmid p}X(K_v )$.
Suppose $E_1 $ and $E_2 $ each have Mordell-Weil rank 1 over $K$, and let $P_i \in E_i (K)$ be points of infinite order. Let $\alpha_i = \frac{h_{E_i}(P_i)}{[K:\Q]\log _{E_i}(P_i)^2}$.
Then $X(K)$ is contained in the finite set of $z$ in $X(K_{\mathfrak{p}} )$ satisfying
$$h_{E_1 ,\mathfrak{p}}(f_1 (z ))-h_{E_2 ,\mathfrak{p}}(f_2 (z ))-2\chi _{\mathfrak{p}} (x(z )) -\alpha_1\log _{E_1}(f_1 (z))^2  +\alpha_2\log _{E_2}(f_2 (z))^2  \in \Omega.$$\end{Theorem}

We further show how Theorem \ref{biellexample} can be used in conjunction with other techniques to determine the set $X(K)$. In Section \ref{howtocomputethings}, we give an algorithm to compute the quantities in Theorem \ref{biellexample} and present two examples using the algorithm. Appendix \ref{naive} explains how one may give an elementary proof of part (i) of Theorem \ref{biellexample}. Appendix \ref{steffenappendix}, by J. Steffen M\"{u}ller, discusses how the Mordell-Weil sieve can be used with quadratic Chabauty to find rational points and describes the sieving carried out to recover $X_0(37)(\Q(i))$ after applying the algorithm for a suitably chosen collection of primes.

In the sequel to the present work, a slightly more general framework is developed \cite{balakrishnan2017quadratic}, which has some practical advantages for computing rational points on curves with everywhere potential good reduction. In recent work with M\"{u}ller, Tuitman, and Vonk \cite{BDMTV}, we use the methods described in these papers to determine the rational points on $X_{s }^+ (13)$, the split Cartan modular curve of level 13, the last remaining case of Serre's uniformity problem for normalisers of split Cartan subgroups, after the work of Bilu, Parent, and Rebolledo \cite{bilu-parent, bilu-parent-rebolledo}.

\section{The Chabauty-Kim method}\label{sec:chabautykim}
We begin by recasting the Chabauty-Coleman method in a motivic framework and then use this to describe Kim's generalisation. Nothing in the section is new, although as far as we are aware, the statement of Lemma \ref{notinttheliterature} is not in the literature.  In this section, $X$ is a smooth projective curve of genus $g$ over a number field $K$.  (By a \emph{curve} over a field $K$ we shall always mean a separated, geometrically integral scheme over $K$ of dimension 1.)  Let $T_0$ denote the set of primes of bad reduction for $X$, let $p$ be a prime of $\mathbb{Q}$ which splits completely in $K$ and is coprime to $T_0$. Let $T=T_0 \cup \{ v|p \}$, and fix a prime $\mathfrak{p}$ lying above $p$. Let $G_T$ denote the maximal quotient of the Galois group of $K$ unramified outside $T$. Unless otherwise indicated, when we write $G$ we will mean either $G_T $ or $G_v $ for $v$ a prime of $K$. 
\subsection{The Chabauty-Coleman method}
We begin with the classical description of the Chabauty-Coleman method. Fix a basepoint $b \in X(K)$ and let $\iota$ denote the Abel-Jacobi map 
$$ \iota:X \hookrightarrow J; \quad
z \mapsto [(z)-(b)]. $$
Let $\log_J :J(K_{\mathfrak{p}})\to H^0 (J_{K_\mathfrak{p}},\Omega ^1 )^*
$ denote the $\mathfrak{p}$-adic logarithm map for the abelian variety $J$. Consider the following diagram:

$$\begin{tikzpicture}
\matrix (m) [matrix of math nodes, row sep=3em,
column sep=3em, text height=1.5ex, text depth=0.25ex]
{X(K) & J(K)  \\
 X(K_{\mathfrak{p}} ) & J(K_{\mathfrak{p}}) & H^0(J_{K_{\mathfrak{p}}}, \Omega^1)^* &  H^0(X_{K_{\mathfrak{p}}}, \Omega^1)^*\\ };
\path[->]
(m-1-1) edge[auto] node[auto] {$\iota  $} (m-1-2)
edge[auto] node[auto] {} (m-2-1)
(m-1-2) edge[auto] node[auto] {} (m-2-2)
edge[auto] node[right] { } (m-2-4)
(m-2-1) edge[auto] node[auto] {$ \iota_{\mathfrak{p}}$ } (m-2-2)
(m-2-2) edge[auto] node[auto] {$\log_J$} (m-2-3)
(m-2-3) edge[auto] node[auto] {$\simeq $}(m-2-4);
\end{tikzpicture} $$

The image of $z$ under the composite map 
$
h: X(K_{\mathfrak{p}})\to  H^0(X_{K_{\mathfrak{p}}}, \Omega^1)^*
$
may be described as the functional sending a global differential $\eta$ to the Coleman integral $\int ^z _b \eta $.  Since the Mordell-Weil rank of $J$ is less than $g$, there is a nonzero differential $\omega $ in $H^0 (X_{K_\mathfrak{p}},\Omega ^1 )^* $ that annihilates the image of $J(K)\otimes \mathbb{Q}_p $. Hence $X(K) \subset X(K_{\mathfrak{p}})$ lies in the set of points for which $\int ^z _b \omega =0$. 

A description of the Chabauty-Coleman method more amenable to nonabelian generalisation is in terms of some standard facts from Galois cohomology and $p$-adic Hodge theory (see e.g., \cite[\S 1]{Nek93}, \cite[\S 3]{bloch-kato} or \cite[\S I.3]{fontaine1994perrin}).
We begin by letting $V:=H^1 _{\acute{e}t}(\overline{X},\mathbb{Q}_p (1))$ and define 
$H^1 _f (G_T ,V)$ to be the subspace of the space of continuous cohomology classes in $H^1 (G_T ,V)$  which are 
crystalline at all primes above $p$. Let $\kappa$ be the \emph{\'etale Abel-Jacobi} map
$$
\kappa :\Div^0 (X) \otimes \mathbb{Q}_p \to H^1 (G_T ,V)
$$
sending a divisor $\sum \mu _i z_i $ to the Kummer class of $[\sum \mu _i z_i ]\in J(K)\otimes \mathbb{Q}$ 
in $$H^1 (G_T ,T_p J)\otimes \mathbb{Q}_p =H^1 (G_T ,V).$$ 

This may be related to $p$-adic Hodge theory as follows.  We first briefly recall the Fontaine functors $D_{\cry}$ and $D_{\dR}$, which send $p$-adic Galois representations to various enriched vector spaces. Associated to $V$ there is a vector space $D_{\cry} (V):=H^0 (G_{\Q _p },V\otimes B_{\cry })$, where $B_{\cry }$ is Fontaine's ring of crystalline periods. The filtration $F^i B_{\cry }$ and Frobenius action on $B_{\cry }$ induces a filtration $F^i $ and Frobenius action on $D_{\cry }(V)$. As explained in \cite[\S 3.11]{bloch-kato}, the exact sequence 
\[
0 \to \Q _p \to B_{\cry }^{\phi =1}\to B_{\dR}/F^0 \to 0
\]
induces an isomorphism $H^1 _e (G_p ,V) \simeq D_{\dR}(V)/F^0 $ (the Bloch-Kato logarithm) where $D_{\dR}(V)$ is the filtered vector space $H^0 (G_{\Q _p },V\otimes B_{\dR})$ with filtration induced by the filtered ring $B_{\dR}$, and $H^1 _e (G_{\Q _p },V):=\Ker (H^1 (G_{\Q _p },V)\to H^1 (G_{\Q _p },V\otimes B_{\cry }^{\phi =1}))$. Moreover, in this case we have $H^1 _e (G_{\Q _p },V)=H^1 _f (G_{\Q _p },V)$. Returning to the Abel-Jacobi map,
 $\kappa $ lands in the subspace $H^1 _f (G_T ,V)$, and there
is a commutative diagram
$$
\begin{tikzpicture}
\matrix (m) [matrix of math nodes, row sep=3em,
column sep=3em, text height=1.5ex, text depth=0.25ex]
{X(K) & H^1 _{f}(G_T ,V)  \\
 X(K_{\mathfrak{p}} ) & H^1 _f (G_{\mathfrak{p}} ,V) & D_{\dR}(V)/F^0 \\ };
\path[->]
(m-1-1) edge[auto] node[auto] {$\kappa  $} (m-1-2)
edge[auto] node[auto] {} (m-2-1)
(m-1-2) edge[auto] node[auto] {$\loc _{\mathfrak{p}} $} (m-2-2)
edge[auto] node[auto] { } (m-2-3)
(m-2-1) edge[auto] node[auto] {$ \kappa _{\mathfrak{p}}$ } (m-2-2)
(m-2-2) edge[auto] node[auto] {$\simeq $} (m-2-3);
\end{tikzpicture} $$
where the top map sends $z$ to $\kappa (z-b)$, and  the bottom right isomorphism is via $p$-adic Hodge theory. Moreover, the Bloch-Kato logarithm is compatible with the usual $p$-adic logarithm: i.e., the composite map 
$
j: X(K_{\mathfrak{p}})\to D_{\dR }(V)/F^0 
$
may be described (see \cite[3.11.1]{bloch-kato}), via the isomorphism 
$$
D_{\dR }(V)/F^0 \simeq H^1 _{\dR }(X)^* /F^0 \simeq H^0 (X,\Omega ^1 )^* ,
$$ 
as the map sending $z$ to the functional sending a global differential $\eta $ to the Coleman integral $\int ^z _b \eta $. Now as before, we have that $X(K) \subset X(K_{\mathfrak{p}})$ lies in the set of points for which $\int ^z _b \omega =0$. 
\subsubsection{Refinements over number fields}\label{samir}
In \cite{Siksek:ECNF}, Siksek explains a refinement of the classical Chabauty-Coleman method over number fields. As explained in loc. cit., heuristically one might expect that if $X$ is a curve of genus $g$ defined over a number field $K$ of degree $d$ over $\mathbb{Q}$, then the Chabauty-Coleman method works whenever the rank of $J(K)$ is less than or equal to $d(g-1)$ (as the Weil restriction of $X$ is now a $g$-dimensional subscheme of the Weil restriction of its Jacobian). In \cite[Theorem 2]{Siksek:ECNF}  a precise technical condition on linear independence of $p$-adic integrals is given which is sufficient to ensure that the Chabauty-Coleman method produces a finite set of points in $\prod _{\mathfrak{p}|p}X(K_{\mathfrak{p}})$.
\subsection{The Chabauty-Kim method}
We now explain how this motivic approach generalises.
Given $b \in X(K)$, let $\pi _1 ^{\acute{e}t,\mathbb{Q}_p}(\overline{X},b)$ denote the unipotent $\mathbb{Q}_p $-\'etale fundamental group of $\overline{X}$ with basepoint $b$ \cite{deligne1989groupe}. Recall that this is equal to the $\mathbb{Q}_p $-Maltsev completion of the usual \'etale fundamental group. In particular, as a pro-algebraic group (i.e. forgetting about the Galois action) it is isomorphic to the quotient of 
a free pro-unipotent group on $2g$ generators by one relation. Let $U^{(0)}:=\pi _1 ^{\acute{e}t,\mathbb{Q}_p }(\overline{X},b)$, and for $i>0$ define $U^{(n)}:=[U^{(0)},U^{(n-1)}]$. Define
$$
U_n := U_n (b)=\pi _1 ^{\acute{e}t,\mathbb{Q}_p }(\overline{X},b)/U^{(n)},
$$
and define
$$
U[n]:=\Ker (U_n \to U_{n-1}).
$$
We will mostly be interested in the case when $n=2$. In this case, using the standard presentation of the topological fundamental group of a surface of genus $g$, we deduce that the sequence of Galois representations 
\begin{equation}\label{fiddly}
0\to H^2 _{\acute{e}t}(\overline{X},\mathbb{Q}_p )^* \stackrel{\cup ^* }{\longrightarrow }\wedge ^2 V \to U[2]\to 0.
\end{equation}
is exact.
Define
$$
P_n (b,z):=\pi _1 ^{\acute{e}t}(\overline{X};b,z)\times _{\pi _1 ^{\acute{e}t}(\overline{X},b)}U_n (b).
$$
Then the assignment $z\mapsto [P_n (b,z)]$ defines a map
$$
X(K)\to H^1 (G_T ,U_n (b)).
$$
One of the fundamental insights of the theory of Selmer varieties is that the cohomology spaces $H^1 (G,U(b))$ carry a much 
richer structure than merely that of a pointed set, and that this extra structure has Diophantine applications.
\begin{Theorem}[Kim \cite{kim:siegel}]
Let $U$ be a finite-dimensional unipotent group over $\mathbb{Q}_p $, admitting a continuous action of $G$. Let $U=U^{(0)} \supset U^{(1)} \supset \ldots $ be the central series filtration. Suppose 
$H^0 (G,U^{(i)} /U^{(i+1)})(\mathbb{Q}_p )=0$ for all $i$. Then the functor
\begin{equation}\nonumber
R\mapsto H^1 (G,U(R))
\end{equation}
from $\Q _p $-algebras to sets
is represented by an affine scheme of finite type over $\mathbb{Q}_p $, such that the six-term exact sequence in nonabelian 
cohomology is a diagram of schemes over $\mathbb{Q}_p $. 
\end{Theorem}
In this paper we will never distinguish between such a cohomology scheme and its $\mathbb{Q}_p $-points.
We now take $U=U(b)$ to be a finite-dimensional $G_T$-stable quotient of $U_n (b)$ whose abelianisation equals $V$.
Note that since the abelianisation of $U(\mathbb{Q}_p )$ has weight $-1$, it satisfies the hypotheses of the theorem, and hence 
$H^1 (G ,U)$ has the structure of the $\mathbb{Q}_p $-points of an algebraic variety over $\mathbb{Q}$. For $z$ a point of $X$, we denote by $P(z)=P(b,z)$ the push-out of $P_n (b,z)$ by $U_n \to U$.
\subsection{Local conditions}
To go from the cohomology varieties $H^1 (G_T ,U)$ to Selmer varieties, one must add local conditions.
For each $v\nmid p$, there is a 
\textit{local unipotent Kummer map} 
$$
j_v :X(K_v ) \to H^1 (G_v ,U); \quad
z \mapsto [P(z)]
$$
which is trivial when $v$ is a prime of potential good reduction and has finite image in general \cite{kim2008component}.
For $\mathfrak{p} \mid p$, by the work of Olsson \cite{olsson2011towards}, the assignment $x\mapsto [P(x)]$ lands inside the subspace of \textit{crystalline} torsors 
$H^1 _f (G_{\mathfrak{p}}  ,U)$. We define  
\begin{equation}\nonumber
j_{\mathfrak{p}} : X(K_{\mathfrak{p}} )\to H^1 _f (G_{\mathfrak{p}}  ,U).
\end{equation}

There is then a commutative diagram
\begin{equation} \label{themaindiagram}
\begin{tikzpicture}
\matrix (m) [matrix of math nodes, row sep=2em,
column sep=2em, text height=1.5ex, text depth=0.25ex]
{X(K) & H^1 (G_T ,U) \\
\prod _{v\in T} X(K_v ) & \prod _{v\in T} H^1 (G_v ,U). \\};
\path[->]
(m-1-1) edge[auto] node[auto]{} (m-2-1)
edge[auto] node[auto] {  } (m-1-2)
(m-2-1) edge[auto] node[auto] { } (m-2-2)
(m-1-2) edge[auto] node[auto] {$\prod \loc _v $} (m-2-2);
\end{tikzpicture} \end{equation}
It is also shown in \cite{kim:siegel} that the
localisation morphisms are morphisms of varieties, and the set of crystalline cohomology classes has the structure of the $\mathbb{Q}_p$-points of a variety. 
We would like to understand the following subscheme of $H^1 (G_T ,U)$:

\begin{Definition}The \textit{Selmer variety} of $U$, denoted $\Sel(U)$, is the reduced scheme associated to the subscheme of 
$H^1 (G_T ,U)$  consisting of cohomology classes $c$ satisfying the following conditions:
\begin{enumerate}
 \item {$\loc _v (c)$ comes from an element of $X(K_v )$ for all $v$ prime to $p$},
 \item {$\loc _v (c)$ is crystalline for all $v$ above $p$,}
 \item {the projection of $c$ to $H^1 (G_T ,V)$ lies in the image of $J(K)\otimes \mathbb{Q}_p $.}
\end{enumerate}
\end{Definition}

\begin{Remark} 
As this definition is slightly non-standard, we briefly recall other definitions of Selmer varieties and Selmer schemes which appear in the literature. In \cite{kim:siegel}, it is proved that $H^1 (G_T ,U)$, $H^1 (G_v ,U)$ and the corresponding cohomology groups with local conditions are represented by affine schemes of finite type over $\Q _p $. However, as explained in \cite{kim2012tangential}, in general these cohomology varieties need not be reduced. The definition given above is most similar to the definition of the Selmer scheme given in \cite{balakrishnan2012non}. There, the authors define the Selmer scheme of $U$ to be the intersection over all $v\neq p$ (equivalently over all $v \in T_0 $) of $\loc _v ^{-1}(j_v (X(\Q _v )))$. If we denote this scheme by $\Sel'(U)$, then $\Sel (U)$ is simply the reduced scheme associated to the fibre product $\Sel'(U)\times _{H^1 _f (G_T ,V)}J(K)\otimes \Q _p $. The reason we adopt this more utilitarian definition is to avoid any assumptions on the finiteness of the Shafarevich-Tate group of the Jacobian of $X$ in the statement of our results.
\end{Remark}

\subsection{Applications to Diophantine geometry}
Let $\mathfrak{p}$ be a prime above $p$.
We have a refinement of the commutative diagram (\ref{themaindiagram}):
\begin{equation*}
\begin{tikzpicture}
\matrix (m) [matrix of math nodes, row sep=2em,
column sep=2em, text height=1.5ex, text depth=0.25ex]
{X(K) & \Sel (U(b)) \\
X(K_{\mathfrak{p}} ) & H^1 _f (G_{\mathfrak{p}} ,U(b)). \\};
\path[->]
(m-1-1) edge[auto] node[auto]{} (m-2-1)
edge[auto] node[auto] {$j$  } (m-1-2)
(m-2-1) edge[auto] node[auto] {$j_{\mathfrak{p}}$ } (m-2-2)
(m-1-2) edge[auto] node[auto] {$\loc _{\mathfrak{p}} $} (m-2-2);
\end{tikzpicture} \end{equation*}

The map $j_{\mathfrak{p}}$ is not algebraic, but is locally analytic, i.e., on each residue disk in $X(K_{\mathfrak{p}} )$, we have that $j_{\mathfrak{p}}$ is given by a $p$-adic 
power series. Furthermore by \cite{kim:chabauty}, $j_{\mathfrak{p}}$ has Zariski dense image. Hence if $\loc _{\mathfrak{p}}$ is not dominant, then 
the set $j_{\mathfrak{p}}^{-1}(\loc _{\mathfrak{p}}(\Sel (U)))$ is finite.
Note that the case $n =1$ now recovers the Chabauty-Coleman method.
\begin{Definition}
Define the set $X(K_{\mathfrak{p}})_U \subset X(K_{\mathfrak{p}})$ to be $j_{\mathfrak{p}}^{-1}(\loc _{\mathfrak{p}}(\Sel (U)))$. When $U=U_n $, we write $X(K_{\mathfrak{p}})_{U_n }$ as $X(K_{\mathfrak{p}})_n $.
\end{Definition}
\begin{Remark}The sets $X(K_{\mathfrak{p}})_n$ are contained in the set of points which are \textit{weakly global of level $n$}, defined in \cite{balakrishnan2012non}. If the $p$-primary part of the Shafarevich-Tate group of the Jacobian of $X$ is finite, then the two sets are equal.
\end{Remark}
\subsection{Properties of $\Sel (U )$}\label{twistingproperties}
In this subsection we recall some properties of the varieties $\Sel (U)$.  We make repeated use of the twisting construction in nonabelian cohomology, as in \cite[I.5.3]{serre:gc}. For topological groups $U$ and $W$, equipped with a continuous homomorphism $U\to \Aut (W)$, and a continuous $U$-torsor 
$P$, we shall denote by $W^{(P)}$ the group obtained by twisting $W$ by the $U$-torsor $P$:
$$
W^{(P)}:=W\times _U P .
$$
Given a group $U$ with an action of $G$ and a continuous $G$-equivariant 
$U$-torsor $P$, we may form a group $U^{(P)}$ which is the twist of $U$ by the $U$-torsor $P$, where $U$ acts on 
itself by conjugation. There is a bijection
$$
H^1 (G,U)\to H^1 (G,U^{(P)})
$$
which sends $G$-equivariant $U$-torsors to $G$-equivariant $U^{(P)}$-torsors. We will make use of the following properties of 
the twisting constructions:
\begin{itemize}
\item The $U$-torsor $P$ is sent to the trivial $U^{(P)}$-torsor.
\item If $H$ is a subgroup of $G$, $U$ is a $G$-group and $P$ is a $G$-equivariant $U$-torsor, then the following diagram commutes:
$$
\begin{tikzpicture}
\matrix (m) [matrix of math nodes, row sep=2em,
column sep=2em, text height=1.5ex, text depth=0.25ex]
{H^1 (G,U) & H^1 (G,U^{(P)}) \\
H^1 (H,U) & H^1 (H,U^{(P)}). \\ };
\path[->]
(m-1-1) edge[auto] node[auto]{} (m-2-1)
edge[auto] node[auto] {} (m-1-2)
(m-1-2) edge[auto] node[auto]{} (m-2-2)
(m-2-1) edge[auto] node[auto] {} (m-2-2);
\end{tikzpicture} $$
\item If $U\to W $ is a homomorphism of $G$-groups, then the diagram
$$
\begin{tikzpicture}
\matrix (m) [matrix of math nodes, row sep=2em,
column sep=2em, text height=1.5ex, text depth=0.25ex]
{H^1 (G,U) & H^1 (G,U^{(P)}) \\
H^1 (G,W) & H^1 (G,W^{(Q)}) \\ };
\path[->]
(m-1-1) edge[auto] node[auto]{} (m-2-1)
edge[auto] node[auto] {} (m-1-2)
(m-1-2) edge[auto] node[auto]{} (m-2-2)
(m-2-1) edge[auto] node[auto] {} (m-2-2);
\end{tikzpicture} $$
commutes, where $P$ is a $G$-equivariant $U$-torsor and $Q$ is the $W$-torsor $P\times _U W$.
\end{itemize}
Since the twisting construction is functorial, if $H^1 (G,U)$ and $H^1 (G,U^{(P)})$ are representable, then the twisting isomorphism is an isomorphism of schemes. This implies the following lemma, which is used in the next section. To state the lemma, let $\alpha _1 ,\ldots ,\alpha _N \in \Sel (U)$ be a set of representatives for the image of $\Sel (U)$ in $\prod _{v\in T_0 }j_v (X(\Q _v ))$.

\begin{Lemma}\label{notinttheliterature}
$\Sel (U)$ is isomorphic to $\sqcup _{i=1}^N H^1 _{\mathcal{O}_K }(G_T ,U^{(\alpha _i )})$, where $H^1 _{\mathcal{O}_K}(G_T ,U^{(\alpha _i )})$ is defined to be the scheme representing $U^{(\alpha _i )}$ cohomology classes which are crystalline at $p$ and trivial at all other primes. \end{Lemma}
\begin{proof}
Let $\alpha \in \prod _{v\in T_0 }j_v (X(K_v ))$ be in the image of $\Sel (U)$ under the map $\prod _{v\in T_0 }\loc _v$, and let $\Sel (U)_{\alpha }$ denote the fibre of $\alpha $ in $\Sel (U)$. We show that $\Sel (U)_\alpha $ is isomorphic to $H^1 _{\mathcal{O}_K}(G_T ,U)$. The first two bullet points above imply that  the twisting morphism
$$
H^1 (G_T ,U)\to H^1 (G_T ,U^{(\alpha )})
$$
sends $\loc _v ^{-1}(\loc _v (\alpha ))$ to $\loc _v ^{-1}(1)$. The first and third bullet points imply that the twisting morphism sends the pre-image of $J(K)\otimes \Q _p $ to itself.
Finally, using all three bullet points, we see that the twisting morphism sends crystalline $U$-torsors to crystalline $U$-torsors.
\end{proof}

\section{Non-density of the localisation map}\label{sec:nond}

For the rest of this paper, we take $K$ to be $\mathbb{Q}$ or an imaginary quadratic extension of $\mathbb{Q}$.
Unless otherwise stated, we will henceforth take $U$ to be a quotient of $U_2 $ surjecting onto $V$.
From the standard presentation of the topological fundamental group of a smooth surface of genus $g$ in terms of $2g$ generators and 1 relation between commutators, the natural map
$
\wedge ^2 V\to U[2]
$
gives an exact sequence
\begin{equation}\label{cup:exact}
0\to H^2 _{\acute{e}t}(\overline{X})^* \stackrel{\cup ^* }{\longrightarrow }\wedge ^2 V \to U[2] \to 0.
\end{equation}
Hence the quotients $U$ intermediate between $U_2 $ and $V$ correspond to Galois subrepresentations of $\wedge ^2 V/H^2 _{\acute{e}t} (\overline{X})^* $.
Note that for any such choice of $U$, there is an inclusion $X(K_{\mathfrak{p}})_2 \subset X(K_{\mathfrak{p}})_U$.
In this paper we restrict attention to the case where $[U,U]$ is isomorphic to $\mathbb{Q}_p (1)^n$ for some $n\geq 1$.
\subsection{Finiteness results}
The reason for considering quotients of the fundamental group which are extensions of $V$ by $\mathbb{Q}_p (1)^n $ is that 
\begin{equation}\label{eqn:coh_fin2}
H^1 _f (G_{\mathfrak{p}},\mathbb{Q}_p (1))\simeq \mathcal{O}_{\mathfrak{p}}^\times \otimes \mathbb{Q}_p \simeq \mathbb{Q}_p,
\end{equation}
 (the first isomorphism may be found in \cite[3.9]{bloch-kato}, and the second comes from the fact that we assume that $p$ splits in $K$, so $K_{\mathfrak{p}}\simeq \Q _p$), and hence by Kummer theory
\begin{equation}\label{eqn:coh_fin1} H^1 _f (G_T ,\mathbb{Q}_p (1))\simeq \mathcal{O}_K ^\times \otimes \mathbb{Q}_p =0.
\end{equation}
This means $\dim H^1 _f (G_T ,\mathbb{Q}_p (1))=0$ and $\dim H^1 _f (G_{\mathfrak{p}},\mathbb{Q}_p (1))=1$ (this is the only place where our restrictions on $K$ are essential). In many situations, 
the Galois cohomology computation above is enough to prove non-density of the localisation map for $\Sel (U)$.
\begin{Lemma}\label{nondensitylemma}
Let $U$ be a quotient of $U_2 $ which is an extension of $V$ by $\mathbb{Q}_p (1)^n $. Let $p$ be a prime of $\mathbb{Q}$ such that $X$ has good reduction at all primes above $p$, and let $\mathfrak{p}$ be a prime above $p$. \\
(i) The dimension of $\Sel (U)$ is bounded above by $\rk J(K)$. \\
(ii) The dimension of $H^1 _f (G_{\mathfrak{p}} ,U)$ is equal to $g+n$. \\
\end{Lemma}
\begin{proof}
(i) By Lemma \ref{notinttheliterature}, it is enough to prove the dimension of $H^1 _{\mathcal{O}_K }(G_T ,U ^{(\alpha )})$ is bounded by $\rk J(K)$ for each $\alpha $ in a set of representatives for the image of $\Sel (U)$ in $\prod _{v \in T_0 }H^1 (G_v ,U)$. The action of $U$ on itself by conjugation induces a trivial action on $V$ and $[U,U]$, giving a Galois-equivariant short exact sequence
\[
1 \to [U,U] \to U^{(\alpha )} \to V \to 1,
\]
which induces an exact sequence of pointed varieties
$$
H^1 _f (G_T ,[U,U])\to H^1 _f (G_T ,U ^{(\alpha )})\to H^1 _f (G_T ,V).
$$
Since $[U,U]\simeq \Q _p (1)^n$, we may apply \eqref{eqn:coh_fin1} to deduce an inequality
$$
\dim \Sel (U) \leq \dim H^1 _f (G_T ,[U,U]) +\dim J(K)\otimes \Q _p =\rk J(K).
$$

(ii)
The computation of the dimension of $H^1 _f (G_{\mathfrak{p}},U)$ follows \cite[$\S 2$]{kim:chabauty}.
By $p$-adic Hodge theory, we have an isomorphism
$$
H^1 _f (G_{\mathfrak{p}},U)\simeq D_{\dR }(U)/F^0 ,
$$
and this gives a short exact sequence
$$
1\to D_{\dR }([U,U])/F^0 \to H^1 _f (G_{\mathfrak{p}},U)\to D_{\dR }(V)/F^0 \to 1.
$$
Since $[U,U]\simeq \mathbb{Q}_p (1)^n $, the dimension of $H^1 _f (G_{\mathfrak{p}},U)$ is $g+n$ by \eqref{eqn:coh_fin2}.
\end{proof}
We  deduce the following:
\begin{Lemma}\label{nondensityiseasy}
Suppose $X$ is a curve of genus $g$, such that $\rk J(K)<g+ \rho (J)-1$. Then $X(K_{\mathfrak{p}})_2$ is finite.
\end{Lemma}
\begin{proof}
By Lemma \ref{nondensitylemma}, the problem of finiteness reduces to finding an appropriate quotient $U$ of $U_2 $.
Note that, since $H^2 _{\acute{e}t}(\overline{J},\Q _p )\simeq \wedge ^2 H^1 _{\acute{e}t}(\overline{X},\Q _p )$, by dualising we have $
\Hom _{G_T} (\wedge ^2 V ,\mathbb{Q}_p (1))\simeq \Hom _{G_T }(\mathbb{Q}_p ,H^2 _{\acute{e}t}(\overline{J},\mathbb{Q}_p (1)))$
and hence the rank of this vector space is at least $\rho(J)$.
Furthermore this is an equality, since $H^2 $ of an abelian variety satisfies the Tate conjecture \cite{faltings-finiteness}.  On the other hand by $\S 2$, the representation
$
U[2]$ is isomorphic to the cokernel of 
$\mathbb{Q}_p (1)\stackrel{\cup ^* }{\longrightarrow }\wedge ^2 V  $.
\end{proof}
\begin{Remark}\label{rmk:integral_points}
If $x$ is a rational point of $X$, $Y:=X-x$ and $b$ is an integral point of $\mathcal{Y}$, a minimal regular model of $Y$, then the same argument as in Lemma \ref{nondensityiseasy} shows that $\mathcal{Y}(\Z _p )_2 $ is finite whenever $\rk J(K)<g+\rho (J)$.
\end{Remark}

\section{Mixed extensions and Nekov\'a\v r's $p$-adic height function}\label{sec:mixedandnekovar}

In this section we introduce some notation for mixed extensions in an abelian category, discuss the relationship between mixed extensions and cohomology with values in unipotent groups, and then review Nekov\'a\v r's $p$-adic height function on mixed extensions.
\subsection{Mixed extensions}\label{subsection:mixed_ext}
Let $\mathcal{A}$ be an abelian category. Let $W_0 ,\ldots ,W_n $ be objects of $\mathcal{A}$, such that for all $i<j$, 
$
\Hom _{\mathcal{A}}(W_i ,W_j )=0.
$
\begin{Definition}
We define a \textit{mixed extension with graded pieces} $W_0 ,\ldots ,W_n$ to be a tuple 
$(M,(M_i ,\alpha _i ))$, where $M$ is an object of $\mathcal{A}$,
\begin{equation}\nonumber
M=M_0 \hookleftarrow M_1 \hookleftarrow M_2 \hookleftarrow \ldots \hookleftarrow M_{n+1} =0
\end{equation}
is a filtration in $\mathcal{A}$ and $\alpha _0 ,\ldots ,\alpha _n $ are
isomorphisms
$$
\alpha _i :M_{i} /M_{i+1} \simeq W_i .
$$
\end{Definition}
A mixed extension $(M,(M_i ,\alpha _i ))$ as above will sometimes be denoted simply by $M$.
\begin{Definition}
Let $(M,(M_i ,\alpha _i))$ and $(N,(N_i ,\beta _i ))$ be mixed extensions with graded pieces $W_0 ,\ldots ,W_n $. A 
\textit{morphism of mixed extensions} is a sequence of compatible isomorphisms
$$
r_i :M_i \stackrel{\simeq }{\longrightarrow }N_i
$$
such that if $r_i $ denotes the induced morphism $M_{i-1} /M_{i}\to N_{i-1} /N_{i}$, then for all $i$, $\beta _i \circ r_i =\alpha _i $.
\end{Definition}
We denote by $\mathcal{C}(\mathcal{A};W_0 ,\ldots ,W_n )$ the category of mixed extensions with graded pieces $W_0 ,\ldots ,W_n $, 
and by $C(\mathcal{A};W_0, \ldots ,W_n )$ the set of isomorphism classes.
Note that our assumption on $\Hom _{\mathcal{A}}(W_i ,W_j )$ implies that an object of $\mathcal{C}(\mathcal{A};W_0 ,\ldots ,W_n )$ 
has no nontrivial automorphisms.
For any $0\leq i<j\leq n$ we have a tautological functor
$$
\varphi _{i,j} :\mathcal{C}(\mathcal{A};W_0 ,\ldots ,W_n )\to \mathcal{C}(\mathcal{A};W_i ,\ldots ,W_j )
$$
which induces a map
$$
\varphi _{i,j}:C(\mathcal{A};W_0 ,\ldots ,W_n )\to C(\mathcal{A};W_i ,\ldots ,W_j ).
$$
\begin{Remark}
The reason for the term ``mixed extension'' is as follows: if $n=2$ and $M$ is an object in $\mathcal{C}(\mathcal{A};W_0 ,W_1 ,W_2 )$, then in the notation of \cite[IX.9.3]{grothendieck288groupes}, $M$ is a \textit{mixed extension} of $\varphi _{0,1}(M)$ and $\varphi _{1,2}(M)$.
\end{Remark}
In the case $n=1$, we have an isomorphism
$$
C(\mathcal{A};W_0 ,W_1 )\simeq \ext ^1 (W_0 ,W_1 ),
$$
and in particular we can add mixed extensions with two graded pieces.
For general $n$, if $M$ and $N$ are objects in $\mathcal{C}(\mathcal{A};W_0 ,\ldots ,W_n )$ such that 
$\varphi _{1,n-1}(M)\simeq \varphi _{1,n-1}(N)$, then the Baer sum of $M$ and $N$, denoted $M+_{1,n-1}N$, will again be an object in 
$\mathcal{C}(\mathcal{A};W_0 ,\ldots ,W_n )$. Similarly, if $\varphi _{2,n}(M)\simeq \varphi _{2,n}(N)$, then we can form $M+_{2,n}N$.
\begin{Definition}
Let $A$ be an abelian group. A function
$$
\alpha :C(\mathcal{A};W_0 ,\ldots ,W_n )\to A
$$
is said to be \textit{bi-additive} if, whenever $\varphi _{1,n-1}(M)=\varphi _{1,n-1}(N)$, we have 
$$
\alpha (M+_{1,n-1}N)=\alpha (M)+\alpha (N),
$$
and whenever $\varphi _{2,n}(M)=\varphi _{2,n}(N)$, we have 
$$
\alpha (M+_{2,n}N)=\alpha (M)+\alpha (N).
$$
\end{Definition}
\subsection{Relation to nonabelian cohomology}
Now suppose that $\mathcal{A}=\rep _{\mathbb{Q}_p }(G)$ is the category of continuous $p$-adic representations of a profinite group $G$. Let $W_0 ,\ldots ,W_n $
be objects in $\rep _{\mathbb{Q}_p }(G)$ with the property that for all $i<j$, \\ $\Hom _{G}(W_i ,W_j )=0$. 
\begin{Definition}
Define $U(W_0 ,\ldots ,W_n )$ to be the unipotent subgroup of $\Aut (\oplus _{0\leq i \leq n}W_i) $ consisting of homomorphisms whose $\Hom (W_i ,W_j )$-component is zero if $i>j$ and the identity endomorphism if $i=j$. 
\end{Definition}
Note that here $\Aut (\oplus _{0\leq i\leq n}W_i )$ refers to the group of automorphisms of vector spaces (i.e. not necessarily $G$-equivariant). The group $\Aut (\oplus _{0\leq i \leq n}W_i )$ has a continuous $G$-action (the restriction of the $G$-action on $\Hom  
(\oplus _{0\leq i \leq n}W_i )$. In this way, $U(W_0 ,\ldots ,W_n )$ inherits a continuous $G$-action.

\begin{Definition}
Let $(M,(M_i ,\alpha _i ))$ be an object in $\mathcal{C}(\rep _{\mathbb{Q}_p }(G);W_0 ,\ldots ,W_n )$. Define $\Phi (M)$ to be the set 
of isomorphisms of vector spaces 
$$
\rho :M\stackrel{\simeq }{\longrightarrow }W_0 \oplus \cdots \oplus W_n 
$$
such that $\rho (M_i )=W_i \oplus \cdots \oplus W_n $ and the induced quotient homomorphism
$$
\rho _i :M_i /M_{i+1}\to W_i
$$
is equal to $\alpha _i $.
\end{Definition}
$\Phi (M)$ has the structure of a $G$-equivariant $U(W_0 ,\ldots ,W_n )$ torsor, and this induces a map
$$
\Phi :C(\rep _{\mathbb{Q}_p }(G);W_0 ,\ldots ,W_n )\to H^1 (G,U(W_0 ,\ldots ,W_n )).
$$
\begin{Lemma}
$\Phi $ is a bijection.
\end{Lemma}
\begin{proof}
To construct an inverse to $\Phi $, define $\Phi '$ to be the functor from the category of equivalence classes of $G$-equivariant $U$-torsors
to $\mathcal{C}(\rep _{\mathbb{Q}_p }(G);W_0 ,\ldots ,W_n )$ sending a torsor $P$ to the twist of $W_0 \oplus \ldots \oplus W_n $ by $P$.
\end{proof}

Under the correspondence, when $G=G_{\mathfrak{p}}$, the subcategory of crystalline $G_{\mathfrak{p}}$-representations is sent to $H^1 _f (G_{\mathfrak{p}},U(W_0 ,\ldots ,W_n ))$, and similarly for semistable representations.
Define 
$$
H^1 _{\st }(G_T ,U(W_0 ,\ldots ,W_n ))\subset H^1 (G_T ,U(W_0 ,\ldots ,W_n ))
$$ 
to be the subvariety of $U$-torsors which are semistable at all primes above $p$ (with no conditions at the primes in $T_0 $). 
We will henceforth use $C(\rep _{\mathbb{Q}_p } (G);W_0 ,\ldots ,W_n )$ and $H^1 (G,U(W_0 ,\ldots ,W_n ))$ interchangeably.
Note that, by our assumption that $\Hom _G (W_i ,W_j )=0$ for all $i<j$, we have $H^0 (G,U(W_0 ,\ldots ,W_n ))=0$, and hence $H^1 (G,U(W_0 ,\ldots ,W_n ))$ is represented by an affine scheme of finite type over $\Q _p $ by \cite[Proposition 2]{kim:siegel}. In particular, we use this to view $C(\rep _{\mathbb{Q}_p }(G);W_0 ,\ldots ,W_n )$, and its various decorated versions, as the $\Q _p $-points of an algebraic variety.

\subsection{Nekov\'a\v r's $p$-adic height pairing on mixed extensions}
In this section we recall the construction of Nekov\'a\v r's $p$-adic height pairing \cite{Nek93}. We will only work in the context of a smooth projective curve over $K$ having good reduction at all primes above $p$.  Our categories will be $G$-representations (for $G=G_{T}$ or $G_v$), and our objects will be $W_0 =\mathbb{Q}_p $,$W_1 =V$,$W_2 =\mathbb{Q}_p(1)$. The group $U(\Q _p ,V,\Q _p (1))$ is a central extension
\[
1 \to \Q_p (1) \to U(\Q _p ,V,\Q _p (1)) \to V \oplus \Hom (V,\Q_p (1)) \to 1.
\]
This induces an action of $H^1 _{\st }(G_T ,\Q _p (1))$ on $H^1 _{\st }(G_T ,U(\Q _p ,V,\Q _p (1)))$, and an exact sequence 
\[
1 \to H^1 _{\st }(G_T ,\Q _p (1))\to H^1 _{\st }(G_T ,U(\Q _p ,V,\Q_p (1)))\to H^1 _{\st }(G_T ,V\oplus \Hom (V,\Q _p (1)) \to 1.
\]
In particular, this gives an isomorphism
\[
H^1 _{\st }(G_T ,U(\Q _p ,V,\Q _p (1)) /H^1 _{\st }(G_T ,\Q _p (1)) \stackrel{\simeq }{\longrightarrow } H^1 _{\st }(G_T ,V) \oplus H^1 _{\st }(G_T , \Hom (V,\Q _p (1))).
\]

The variety 
$C(\rep _{\mathbb{Q}_p }(G) ;\mathbb{Q}_p ,V,\mathbb{Q}_p (1))$ has a natural involution defined by 
$$
M\mapsto M^* (1).
$$
We say a function
$$
\alpha :C(\rep _{\mathbb{Q}_p }(G) ;\mathbb{Q}_p ,V,\mathbb{Q}_p (1))\to \mathbb{Q}_p 
$$
is \textit{symmetric} if $\alpha (M)=\alpha (M^* (1))$. Nekov\'a\v r's $p$-adic height pairing is defined via a family of local height functions 
$$
h_v :H^1 (G_v ,U(\mathbb{Q}_p ,V,\mathbb{Q}_p (1))) \to \mathbb{Q}_p ,
$$
for $v$ prime to $p$, and 
$$
h_v :H^1 _{\st } (G_v ,U(\mathbb{Q}_p ,V,\mathbb{Q}_p (1)))\to \mathbb{Q}_p 
$$
for $v$ above $p$, which are continuous, bi-additive and symmetric.
The input for Nekov\'a\v r's construction is a class $\chi $ in $H^1 (G_{T},\mathbb{Q}_p )$ and 
a splitting 
\begin{equation}\label{split} s:H^1 _{\dR } (X_{K_v },\mathbb{Q}_p )\to F^1 H^1 _{\dR } (X_{K_v },\mathbb{Q}_p )\end{equation}
of the Hodge filtration of $H^1 _{\dR } (X_{K_v}  ,\mathbb{Q}_p )$ at every prime $v$ above $p$. We will restrict attention to splittings $s$ for which $\Ker (s)$ is an isotropic subspace with respect to the Hodge filtration. For such splittings, the local height is symmetric in the sense that $h_p (M)=h_p(M^* (1))$ (see \cite[$\S 4.11$]{Nek93}).

\subsubsection{$v$ prime to $p$}
For $v$ not above $p$, the construction of local height pairings is immediate given the weight-monodromy conjecture for curves \cite{raynaud19941}, which implies that 
$$
H^0 (G_v ,V)=H^1 (G_v ,V)=0,
$$
and hence by the six-term exact sequence in nonabelian cohomology,
$$
H^1 (G_v ,U(\mathbb{Q}_p ,V,\mathbb{Q}_p (1)))\simeq H^1 (G_v ,\mathbb{Q}_p (1)).
$$
This gives a function
$$
.\cup \chi _v :H^1 (G_v ,U(\mathbb{Q}_p ,V,\mathbb{Q}_p (1)))\to \mathbb{Q}_p
$$
via the isomorphism $H^2 (G_v ,\mathbb{Q}_p (1))\simeq \mathbb{Q}_p $ coming from local class field theory.
\subsubsection{$v$ above $p$}\label{subsec:algebraic}
For $v$ above $p$, the construction of local height pairings uses $p$-adic Hodge theory. As we will only be interested in the crystalline case, we restrict attention to 
describing Nekov\'a\v r's functional on crystalline mixed extensions
$$
h_v :H^1 _f (G_v ,U(\mathbb{Q}_p ,V,\mathbb{Q}_p (1)))\to \mathbb{Q}_p .
$$ 
The construction is analogous to the case when $v$ was prime to $p$: given a mixed extension $M$ in the category of filtered $\phi$-modules, 
with graded pieces $\mathbb{Q}_p ,D_{\cry}(V)$ and $D_{\cry}(\Q _p (1))$, one constructs an extension $c$ of $\mathbb{Q}_p$ by $D_{\cry}(\Q _p (1))$, 
identifies this as an element $c'$ of $H^1 _f (G_p ,\mathbb{Q}_p (1))$, and then defines 
$$
h(M):=c' \cup \chi _v .
$$
We now sketch the construction of $c$. Note that (in the category of admissible filtered $\phi$-modules) $\ext ^1(\mathbb{Q}_p ,D_{\cry}(\Q _p (1)))\simeq D_{\dR }(\Q _p (1))$, so one may 
equivalently think of $c$ as an element of $D_{\dR}(\mathbb{Q}_p (1))$.
Let $(M,(M_i ,\alpha _i ))$ be a mixed extension with graded pieces $\mathbb{Q}_p ,D_{\cry }(V)$ and $D_{\cry }(\mathbb{Q}_p (1))$. 
The extension class of $M$ in $\ext ^1 (\mathbb{Q}_p ,M_1 )$ defines an element of $M_1 /F^0 $. Using the splitting $s$ specified in \eqref{split}, one 
lifts this to an element of $M_1 $.
For weight reasons there is a 
canonical $\phi $-equivariant splitting of the inclusion $M_2 \hookrightarrow M_1 $,
and hence via $\alpha _2 $ one obtains an element $c$ of $D_{\dR}(\mathbb{Q}_p (1))$, as required.

In the language of \cite{kim:chabauty} we may define the local height of a crystalline mixed extension as follows.
There is an isomorphism \cite[$\S 2$]{kim:chabauty}:
$$
H^1 _f (G_v ,U(\mathbb{Q}_p ,V,\mathbb{Q}_p (1)))\simeq D_{\dR }(U(\mathbb{Q}_p ,V,\mathbb{Q}_p (1)))/F^0 .
$$
Let $V^{\dR }=D_{\cry }(V)$ and $D_{\cry }(1):=D_{\cry }(\mathbb{Q}_p (1))$. As for $G$-representations, we define a unipotent group $U(\mathbb{Q}_p ,V^{\dR },D_{\cry }(1))$ 
with filtration and $\phi$-action. This is then isomorphic (as a group with filtration and $\phi$-action) to 
$D_{\cry }(U(\mathbb{Q}_p ,V,\mathbb{Q}_p (1)))$.
The homogeneous space $U(\mathbb{Q}_p ,V^{\dR },D_{\cry }(1))/F^0 $ parametrises 
mixed extensions with graded pieces $\mathbb{Q}_p ,V^{\dR }$ and $D_{\cry }(1)$ in the category 
of filtered $\phi $-modules. 
Arguing as above, a splitting of the Hodge filtration determines an algebraic function
$$
U(\mathbb{Q}_p ,V^{\dR },D_{\cry }(1))/F^0 \to D_{\cry }(1).
$$
In particular, we obtain the following lemma.
\begin{Lemma}
The local height function 
$$
h_v :H^1 _f (G_v ,U(\mathbb{Q}_p ,V,\mathbb{Q}_p (1)))\to \mathbb{Q}_p 
$$
is algebraic.
\end{Lemma}
\begin{Remark}\label{Rmk:different_splittings}
If $s_1 $ and $s_2 $ are two splittings of the Hodge filtration, giving associated height functions $h_{v,1}$ and $h_{v,2}$, then their difference defines a bilinear map
$$
V_{\dR} /F^0 \times V_{\dR} /F^0 \to D_{\cry }(1).
$$
\end{Remark}
\subsubsection{Global heights}\label{globalheights}
We define 
$$
h:H^1 _{\mathcal{\st }}(G_T ,U(\mathbb{Q}_p ,V,\mathbb{Q}_p (1)))\to \mathbb{Q}_p
$$
to be the composite of 
\small{$$H^1 _{\mathcal{\st }}(G_T ,U(\mathbb{Q}_p ,V,\mathbb{Q}_p (1))) \xrightarrow{\prod_{v\in T}\loc _v } \prod _{v\in T_0 }H^1 (G_v ,U(\mathbb{Q}_p ,V,\mathbb{Q}_p (1)))\times 
\prod _{v|p }H^1 _{\st }(G_v ,U(\mathbb{Q}_p ,V,\mathbb{Q}_p (1)))$$}
\normalsize
with 
$$
\prod _{v\in T_0 }H^1 (G_v ,U(\mathbb{Q}_p ,V,\mathbb{Q}_p (1)))\times 
\prod _{v|p }H^1 _{\st }(G_v ,U(\mathbb{Q}_p ,V,\mathbb{Q}_p (1))) \xrightarrow{\sum h_v} \mathbb{Q}_p.
$$
{The function $h$ is invariant under the action of $H^1 _{\st }(G_T ,\Q _p (1))$ on $H^1 _{\st }(G_T ,U(\Q _p ,V,\Q _p (1)))$: for all $c$ in $H^1 _{\st }(G_T ,U(\Q_p ,V,\Q_p (1)))$ and $d\in H^1 _{\st }(G_T ,\Q _p (1))$, and all primes $v$ in $T$, we have 
\[
h _v (c+d)=h_v (c)+\loc _v (\chi \cup d),
\]
hence $h(c)=h(c+d)$ by class field theory.  We have
$$
\varphi _{0,1}\times \varphi _{1,2}:H^1 _{\mathcal{\st }}(G_T ,U(\mathbb{Q}_p ,V,\mathbb{Q}_p (1)))\to H^1 _f (G_T ,V)\times H^1 _f (G_T ,V),
$$
using the fact that $H^1 _{\st }(G_T ,V)\simeq H^1 _f (G_T ,V)$. By additivity and continuity, it hence factors through
\begin{align*}
H^1 _{\st }(G_T ,U(\mathbb{Q}_p ,V,\mathbb{Q}_p (1)))&\to H^1 _f (G_T ,V)^{\otimes 2}\\
M\qquad\qquad &\mapsto \varphi _{0,1}(M)\otimes (\varphi _{1,2}(M)^* (1)).\end{align*}
If $s$ is chosen to be isotropic with respect to the cup product, the function $h$ is furthermore symmetric, i.e. 
$h(M)=h(M^* (1))$ \cite[$\S 4.11$]{Nek93}.

\section{Selmer varieties and mixed extensions}\label{sec:pathsandlinear}

We now return to Selmer varieties. Here $U$ will be an extension of $V$ by $\mathbb{Q}_p (1)$.
To obtain equations for $X(K_{\mathfrak{p}} )_U $, we use Nekov\'a\v r's construction to define a map
$$
\Sel (U)\to \mathbb{Q}_p .
$$
A natural analogue of Nekov\'a\v r's construction is to start with the input of a cohomology class $\chi $ in $H^1 (G_T ,\mathbb{Q}_p )$, and to define, 
at all primes $v$ in $T_0$, an algebraic function
$$
H^1 _* (G_v ,U)\to \mathbb{Q}_p 
$$
which, restricted to $H^1 (G_v ,\mathbb{Q}_p (1))$, is simply the cup product with $\chi $. 

Given a splitting of the Hodge filtration, one may define such a function, but in order to determine equations for Selmer varieties, it is better to have a construction with some kind of linearity properties analogous to those of the global height pairing. For this reason, in this section we define a way to embed
$\Sel (U)$ into $H^1 _{\mathcal{\st }}(G_T ,U(\mathbb{Q}_p ,V,\mathbb{Q}_p (1))$ via twisting. We then apply Nekov\'a\v r's construction, giving (via composition) local functions $\Sel (U)\to \mathbb{Q}_p$.
Note that if $\mathbb{Q}_p (1)$ is replaced by a different Galois representation $W$  of motivic weight $-2$ arising in $U[2]$, one may mimic Nekov\'a\v r's construction with the cohomology class $\chi $ replaced by a cohomology class in $H^1 (G_T ,W^* (1))$ which is nontrivial and noncrystalline at $\mathfrak{p}$, assuming one can prove such a class exists. This is developed in the sequel to this paper \cite{balakrishnan2017quadratic}.

\subsection{Twisting the enveloping algebra}\label{subsec:algebraic2}
To construct a mixed extension associated to an element of $H^1 (G,U)$, we define a $G$-representation with an equivariant $U$-module structure, which will be denoted $A(b)$, and then send a $U$-torsor $P$ to the twist of $A(b)$ by $P$.

$A(b)$ will be defined to be a certain finite-dimensional quotient of the universal enveloping algebra of $\pi _1 ^{\acute{e}t,\mathbb{Q}_p }(\overline{X},b)$.  By the theory of Maltsev completion, this has a very concrete description, which we now recall (see \cite[\S 2]{kim:coates}).

\begin{Definition}Let 
\[
\Z _p [\! [\pi _1 ^{\acute{e}t, (p)}(\overline{X},b)]\! ]:= \varprojlim \Z _p [\pi _1 ^{\acute{e}t}(\overline{X},b)/N]
\]
denote the inverse limit of the group algebras of quotients $\pi _1 ^{\acute{e}t}(\overline{X},b)/N$ of $p$-power order.
Let $I$ denote the kernel of the natural map
\[
\Z _p [\! [\pi _1 ^{\acute{e}t,(p)}(\overline{X},b)]\! ]\to \Z _p .
\]
Then we define $A_n (b):=\mathbb{Q}_p \otimes \Z _p [\! [\pi _1 ^{\acute{e}t,(p)}(\overline{X},b)]\! ]/I^{n+1}$.
\end{Definition}
$A_n (b)$ is equipped with the structure of a Galois-equivariant $\pi _1 ^{\acute{e}t}(\overline{X},b)$-module, via the action of $\pi _1 ^{\acute{e}t}(\overline{X},b)$ on $\Z _p [\! [\pi _1 ^{\acute{e}t,(p)}(\overline{X},b)]\! ]$. Hence for any Galois-equivariant $\pi _1 ^{\acute{e}t}(\overline{X},b)$-torsor $P$, we can twist $A_n (b)$ by $P$ to get a Galois representation $A_n (b)^{(P)}$. When $P=\pi _1 ^{\acute{e}t}(\overline{X};b,z)$, we may identify $A_n (b) ^{(P)}$ with the Galois-equivariant $A_n (b)$-module $A_n (b,z)$ obtained by tensoring $\mathbb{Q}_p [\pi _1 ^{\acute{e}t}(\overline{X};b,z)]$, thought of as a $\mathbb{Q}_p [\pi _1 ^{\acute{e}t}(\overline{X},b)]$-module, with $A_n (b)$. It follows from the theory of Maltsev completion that the action of 
$\pi _1 ^{\acute{e}t}(\overline{X},b)$ on $A_n (b)$ factors through the homomorphism 
$$
\pi _1 ^{\acute{e}t}(\overline{X},b)\to U_n (b).
$$
Furthermore, $A_n (b)$ is a quotient of the enveloping algebra of $U_n (b)$ and a faithful representation of $U_n (b)$. More generally 
we can view the $\mathbb{Q}_p$-vector space generated by the torsor of paths from $b$ to $z$, denoted
$\mathbb{Q}_p [\pi _1 ^{\acute{e}t}(\overline{X};b,z)]$, as a $G$-equivariant free rank $1$ module over 
$\mathbb{Q}_p [\pi _1 ^{\acute{e}t}(\overline{X},b)]$. Hence we may make the following definition.
\begin{Definition}
Let
$A_n (b,z)$  be the $G$-equivariant free rank 1 $A_n (b)$-module
$$
\mathbb{Q}_p [\pi _1 ^{\acute{e}t}(\overline{X};b,z)]\times _{\mathbb{Q}_p [\pi _1 ^{\acute{e}t}(\overline{X},b)]} A_n (b).
$$
\end{Definition}
Note that $A_n (b,z)$ is naturally equipped with a $G$-stable filtration
$$
A_n (b,z)\supset IA_n (b,z)\supset \ldots \supset I^{n+1}A_n (b,z)=0
$$ 
coming from the $I$-adic filtration on 
$\mathbb{Q}_p [\pi _1 ^{\acute{e}t}(\overline{X};b,z)]$, and that the action of $A_n (b)$ respects this filtration.
We define 
$$
A[k]:=I^k A_n (b)/I^{k+1}A_n (b).
$$
A second viewpoint is that $A_n (b,z)$ is the twist of $A_n (b)$ by $[\pi _1 ^{\acute{e}t}(\overline{X};b,z)]$ via the left action of 
$\pi _1 ^{\acute{e}t}(\overline{X},b)$ on $A_n (b)$.
There is also a more general construction: for all $k$, $I^k A_n (b)$ admits compatible actions of $U_n (b)$ and 
$G$. Hence for any $G$-equivariant $U_n (b)$-torsor $P$, we may construct the twist $A_n (b)^{(P)}$ of $A_n (b)$ by $P$. In the case when $P$ is $\pi _1 ^{\acute{e}t}(\overline{X};b,z)\times _{\pi _1 ^{\acute{e}t}(\overline{X},b)}U_n (b)$, we have that $A_n ^{(P)}$ is 
just $A_n (b,z)$. The action of $U_n $ on $I^k /I^{k+1}$ is trivial, hence for any such $P$ we have an isomorphism
$$
I^k A_n (b)^{(P)}/I^{k+1}A_n (b)^{(P)}\simeq I^k A_n (b)/I^{k+1}A_n (b).
$$
Thus we obtain a well-defined map
\begin{align*}
[\; . \;]:  H^1 (G,U_n ) &\to H^1 (G,U(A[0],A[1],\ldots ,A[n])) \\
  P  &\mapsto [A_n (b)^{(P)}].\end{align*} 
An equivalent definition of this map would be to define $\aut (A_n (b))$ to denote the group of unipotent automorphisms of 
$A_n (b)$ as a filtered vector space (i.e. automorphisms of $A_n (b)$ which respect the filtration and are the identity on 
the associated graded). Then there is a group homomorphism
$$
U_n (b)\to \aut (A_n (b))
$$
and an induced map on cohomology
$$
H^1 (G,U_n ) \to H^1 (G,\aut (A_n (b))).
$$
There is also an isomorphism
$$
H^1 (G,\aut (A_n (b))) \to H^1 (G,U(\mathbb{Q}_p ,A[1],\ldots ,A[n]))
$$
coming from the $G$-equivariant $(\aut (A_n (b)),U(\mathbb{Q}_p ,A[1],\ldots ,A[n]))$-bitorsor of isomorphisms of filtered 
vector spaces 
$$
A_n (b)\stackrel{\simeq }{\longrightarrow }\oplus _{k=0}^n A[k],
$$
see \cite[Proposition 35]{serre:gc}. The map $[\; . \;]$ defined above is simply the composite.

We now focus on the depth 2 case. There is a short exact sequence 
$$
0\to A[2]\to A_2 (b)\to A_1 (b)\to 0
$$
compatible with the action of $G$ and $U$. We have that $A[2]$ is canonically isomorphic to $[U_2 ,U_2 ]\oplus \Sym ^2 V$.
\begin{Definition}
Suppose that $\rho(J) > 1$. Let
$$
\xi :A[2]\to \mathbb{Q}_p (1)
$$
be a Galois-equivariant surjection whose restriction to $[U_2 ,U_2 ] \simeq \Coker(\wedge ^2 V \stackrel{\cup ^* }{\longrightarrow} \Q _p (1))$ is nonzero and factors through 
$[U_2 ,U_2 ]\to [U,U]$. Define $A(b)$ to be the mixed extension with graded pieces $\mathbb{Q}_p ,V,$ and $\mathbb{Q}_p (1)$ 
obtained by pushing out $A[2]\hookrightarrow A_2 (b)$ by $\xi :A[2]\to \mathbb{Q}_p (1)$. We define $IA(b)$ to be the kernel of 
the projection
$
A(b)\twoheadrightarrow \mathbb{Q}_p .
$
\end{Definition}
The representation $A(b)$ has a compatible $U$-action, and hence for any $U$-torsor $P$ we obtain a mixed extension $A(b)^{(P)}$ 
with graded pieces $\mathbb{Q}_p ,V$, and $\mathbb{Q}_p (1)$. 
Since the projection map $A(b)\to \mathbb{Q}_p $ and the inclusion map $\mathbb{Q}_p (1)\to A(b)$ are $U$-equivariant, for any $P$ we have exact sequences 
$$
0\to IA(b)^{(P)}\to A(b)^{(P)}\to \mathbb{Q}_p \to 0
$$
and
$$
0\to \mathbb{Q}_p (1)\to A(b)^{(P)}\to A_1 (b)^{(P)}\to 0.
$$
When $P=P(b,z)$ we denote $A(b)^{(P)}$ by $A(b,z)$ and $IA(b)^{(P)}$ by $IA(b,z)$. When we want to emphasise the dependence on $X$, we write $A(X)(b)$ and $A(X)(b,z)$.
By our assumptions on the homomorphism $A[2]\to \mathbb{Q}_p (1)$, $A(b)$ is a faithful $U$-representation. Note that since the $U$-action on $A[2]$ is trivial, we could define
$A(b)^{(P)}$ to be the pushout of $A[2]\hookrightarrow A_2 (b)^{(P)}$ by $A[2]\to \mathbb{Q}_p (1)$.
As in the above discussion of the map $[\; . \;]$, the map from $H^1 (G,U)$ to $H^1 (G,U(\mathbb{Q}_p ,V,\mathbb{Q}_p (1)))$ is algebraic.

\subsection{Description of $h(A(b,z))$}
Let $U$ be a quotient of $U_2 $ which is an extension of $V$ by $\mathbb{Q}_p (1)$. As explained in Section \ref{sec:nond}, $U$ corresponds to a Tate class
\[
Z:\Q _p \hookrightarrow \wedge ^2 H^1 _{\acute{e}t}(\overline{X},\Q _p (1))
\]
lying in the kernel of the cup product map. Let $A(b)$ be the corresponding quotient of the enveloping algebra 
of $U$. We now consider the maps
\begin{align*}
H^1 (G_v ,U)\to \mathbb{Q}_p ; & \quad P\mapsto h _v (A(b)^{(P)}) \\
H^1 (G_T ,U)\to \mathbb{Q}_p ; & \quad P \mapsto h (A(b)^{(P)}).
\end{align*}
The following lemma follows from the work of Kim and Tamagawa \cite{kim2008component}.
\begin{Lemma}
Let $v$ be a prime of $K$ that is coprime to $p$. Then the map
$$
X(K_v ) \to \mathbb{Q}_p; \quad
z\mapsto h_v (A(b,z))$$
is identically zero when $v$ is a prime of potential good reduction and 
has finite image in general. \end{Lemma}
\begin{proof}
If $v$ is a prime of potential good reduction, then there is a finite Galois extension $L|K_v $ such that for every $L$-rational point $z$, the $U$-torsor $P(z)$ admits a $G_L$-equivariant trivialisation. From \cite[$\S$I.5.8]{serre:gc},  there is a short exact sequence
$$
1\to H^1 (\Gal (L|K_v ),U^{G_L})\to H^1 (G_{K_v },U)\to H^1 (G_L ,U),
$$
and hence every $G_{K_v}$-equivariant $U$-torsor is trivial, since $U^{G_L}=1$.

For the general case, we use \cite[Corollary 0.2]{kim2008component}, which says that the map
$$j_v :X(K_v )\to H^1 (G_{K_v },U)$$
has finite image. This implies the lemma, as the map $z\mapsto h_v (A(b,z))$ factors through $j_v$.
\end{proof}
We now consider global properties of $A(b,z)$. The mixed extension $A(b,z)$ is a mixed extension of $A_1 (b,z)$ and $IA(b,z)^* (1)$. To understand the height of $A(b,z)$, we first 
need to understand the map
$$
H^1 (G,U)\to \ext ^1 (V,\mathbb{Q}_p (1))
$$
defined by sending a torsor $P$ to the twist of $IA(b)$ by $P$ (when $P=P(b,z)$, the twist of $IA(b)$ by $P$ is $IA(b,z)$).
Let $\langle \, , \,\rangle $ $:V\times V\to \mathbb{Q}_p (1)$ be the homomorphism induced from the Weil pairing and let $\tau _W :V\stackrel{\simeq }{\longrightarrow } \Hom (V,\mathbb{Q}_p (1))$ denote the homomorphism sending $v$ to $w\mapsto \langle w,v \rangle .$ Let $\tau _{W*}$ denote the induced isomorphism $H^1 (G ,V)\simeq \ext ^1 (V,\Q _p (1))$.
Let $ \tau _Z :V\to \Hom (V,\mathbb{Q}_p (1))$ denote the homomorphism sending $v$ to $w\mapsto [\widetilde{w},\widetilde{v}],$
where $\widetilde{w}$ and $\widetilde{v}$ are lifts of $w$ and $v$ to $U$ and $[\, ,\,]$ denotes the commutator in the group $U$. Let $\tau _{Z*}$ denote the induced homomorphism 
$$
H^1 (G,U)\to H^1 (G,V)\to \ext ^1 (V,\mathbb{Q}_p (1)).
$$
We will also denote by $\tau _{Z*}$ the map $H^1 (G,V)\to \ext ^1 (V,\mathbb{Q}_p (1))$ through which the above map factors.
Then by definition of the twisting construction, 
there is an equality of extensions of $\mathbb{Q}_p (1)$ by $V$: 
$$
[IA(b,z)]=[IA(b)]+\tau _{Z*} ([P(b,z)]).
$$
Let $a(Z)$ denote the linear map
$
H^1 _f (G_T ,V)\to H^1 _f (G_T ,V)
$
defined by $$a(Z)=\tau _{W*} ^{-1}\circ \tau _{Z*}.$$ By the above, $A(b,z)$ is a mixed extension of $\kappa (z-b)$ and $a(Z)(\kappa (z-b))+[IA(b)]$, where $\kappa $ is the \'etale Abel-Jacobi map.

We now explain how one obtains equations for the finite set $X(K_{\mathfrak{p}})_U $. First we make precise our choice of $p$-adic height. If $K=\mathbb{Q}$, then up to scalars, there is a unique choice of character $\chi $. Recall that in the imaginary quadratic case, we have a decomposition $p\mathcal{O}_K =\mathfrak{p}\overline{\mathfrak{p}}$. We henceforth take $\chi $ to be an idele class character which vanishes on $\mathcal{O}_{\overline{\mathfrak{p}}}^\times $. By class field theory, the space of such characters is one-dimensional, and hence $\chi $ is uniquely determined up to scalars. Since the mixed extensions $A(b,z)$ are crystalline at all primes above $p$, this means that 
$$
h(A(b,z))=h_\mathfrak{p}(A(b,z))+\sum _{v\in T_0 }h_{v}(A(b,z)).
$$
Let $\omega _0 ,\ldots ,\omega _{g-1}$ be a basis of $H^0 (X_{K_\mathfrak{p}},\Omega ^1 )$.
\begin{Proposition}\label{prop:vital} Suppose $\rk J(K)=g$, that $\rho (J)>1$, and that the map 
\begin{equation}\label{injectivelog}
J(K)\otimes _{\mathbb{Z}}\mathbb{Q}_p \to H^1 _f (G_{\mathfrak{p}},V)
\end{equation}
is an isomorphism. Let $b$ be a $K$-rational point of $X$.
Then the set 
$$
\Omega =\{ -\sum _{v\in T_0 } h_v (A(b,z_v )):(z_v )\in \prod _{v\in T_0 }X(K_v ) \}
$$
is finite, and there are constants $c_{ij},d_i$ (for $0\leq i\leq g-1$) such that $X(K_{\mathfrak{p}})_U$ is finite and contained in the set of $z$ in $X(K_{\mathfrak{p}})$ satisfying
\begin{equation}\label{propositionequation}
 h _{\mathfrak{p}}(A(b,z))+\sum _{0\leq i,j<g }c_{ij}\left(\int ^z _b \omega _i \right)\left(d_j +\sum _{0\leq k<g}a(Z)_{jk}\int ^z _b \omega _k \right)\in \Omega ,
\end{equation}
where $a(Z)_{jk}$ denotes the matrix of $a(Z)$ acting on $H^0 (X,\Omega ^1 )$ with respect to the basis $\omega _i $.
\end{Proposition}
\begin{proof}
By injectivity of (\ref{injectivelog}), for all $0\leq i\leq g-1$ there is a $\kappa _i $ in $H^1 _f (G_T ,V)$ such that $\loc _{\mathfrak{p}}(\kappa _i )=\omega _i ^* $ via the isomorphism $H^1 _f (G_{\mathfrak{p}})\simeq H^0 (X_{K_\mathfrak{p}},\Omega ^1 )^* $. Let $H_{ij}$ be a mixed extension with graded pieces $\mathbb{Q}_p, V,$ and $\mathbb{Q}_p (1)$ such that $\varphi _{0,1}(H_{ij})=\kappa _i $ and $\varphi _{1,2}(H_{ij})=\kappa _j ^* (1)$. Define 
$
c_{ij}=-h(H_{ij}).
$
Define $d_i $ by 
$$
\loc _{\mathfrak{p}}(IA(b)^* (1))=\sum _{0\leq i<g}d_i \omega _i ^* .
$$
Then since (\ref{injectivelog}) is an isomorphism, we have 
\begin{align*}
\varphi _{0,1}(A(b,z)) & =\sum _{0\leq i<g} \left(\int ^z _b \omega _i \right)\kappa _i ,\\
\varphi _{1,2}(A(b,z)) & =\sum _{0\leq j<g} \left(d_j +\sum _{0\leq k<g} a(Z)_{jk}\int ^z _b \omega _k \right)\kappa _j ^* (1). 
\end{align*}
Hence in $\Sym ^2 H^1 _f (G_T ,V)$ we have
$$
\varphi _{0,1}(A(b,z))\varphi _{1,2}(A(b,z))=\sum _{0\leq i,j<g}\left(\int ^z _b \omega _i \right)\left(d_j +\sum _{0\leq k<g}a(Z)_{jk}\int ^z _b \omega _k \right)\kappa _i \kappa _j ,
$$
giving an equality of global heights
$$
h(A(b,z))=\sum _{0\leq i,j<g}\left(\int ^z _b \omega _i \right)\left(d_j +\sum _{0\leq k<g} a(Z)_{jk}\int ^z _b \omega _k \right)h(H_{ij}).
$$
This establishes that $K$-rational points on $X$ satisfy the above equation. By $\S\ref{subsec:algebraic}$ and $\S\ref{subsec:algebraic2}$, for any $\beta $  in $\mathbb{Q}_p $, and any functional
$$
B:H^1 _f (G_{\mathfrak{p}},V)\otimes H^1 _f (G_{\mathfrak{p}},V)\to \mathbb{Q}_p ,
$$
the equation
$$
h _{\mathfrak{p}}(A(b)^{(P)})+B(A_1 (b)^{(P)},(IA(b)^{(P)})^* (1))= \beta
$$
defines a codimension one subvariety $W_\alpha $ of $H^1 _f (G_{\mathfrak{p}},U)$. For $P=A(b,z)$, the left hand side of this equation is equal to 
\begin{equation}\label{eq:LHS}
h _{\mathfrak{p}}(A(b,z))+\sum _{0\leq i,j<g}\left(\int ^z _b \omega _i \right)\left(d_j +\sum _{0\leq k<g} a(Z)_{jk}\int ^z _b \omega _k \right)B(\omega _i ^* \otimes \omega _j ^* )= \beta .
\end{equation} Then, as in \cite{kim:chabauty}, $j_{\mathfrak{p}}^{-1}(W_{\alpha })$ is finite, completing the proof of the proposition.
\end{proof}
\begin{Remark}\label{independence_of_splitting}
Note that the constants $d_i$ and $c_{ij}$ depend on the choice of splitting of the Hodge filtration. However by Remark \ref{Rmk:different_splittings}, the left hand side of \eqref{eq:LHS} is independent of the splitting.
\end{Remark}
\begin{Remark}
If $Z_1$ and $Z_2 $ are two nontrivial Tate classes in the kernel of the cup-product, and $Z_1 \neq -Z_2$, then their sum will be another such Tate class, and the associated mixed extension $A_{Z_1 +Z_2 }(b,z)$ is simply the Baer sum of the mixed extensions $A_{Z_1 }(b,z)$ and $A_{Z_2 }(b,z)$ corresponding to $Z_1$ and $Z_2 $; i.e., in the notation of Section \ref{subsection:mixed_ext},
\[
A_{Z_1 +Z_2 }(b,z)=A_{Z_1 }(b,z)+_{0,1}A_{Z_2}(b,z).
\]
Hence by additivity, $h_p (A_{Z_1 +Z_2 }(b,z))=h_p (A_{Z_1 }(b,z))+h_p (A_{Z_2 }(b,z))$, and so we get no new equations for $X(K_{\mathfrak{p}})$. On the other hand, if $Z_1 ,\ldots ,Z_d$ is a basis for $\Hom (\Q _p ,\Ker (\wedge ^2 H^1 (\overline{X},\Q _p )\to \Q _p (-1))(1)$, then the morphism
\[
H^1 _f (G_{\mathfrak{p}},U_2 ) \to H^1 _f (G_{\mathfrak{p}},V)\times \Q _p ^d
\]  
sending a torsor $P$ to 
\[
(\pi _* (P),h_{\mathfrak{p}}(A_{Z_1 }(b)^{(P)}), \ldots ,h_{\mathfrak{p}}(A_{Z_d }(b)^{(P)}))
\]
is surjective. Since $X(K_{\mathfrak{p}})$ has Zariski dense image in $H^1 _f (G _{\mathfrak{p}},U_2 )$, this implies that we obtain $d$ independent equations satisfied by $X(\Q )$.
\end{Remark}

\begin{Remark}\label{rk:sequel_equal}
In the sequel to this paper \cite[Lemma 13]{balakrishnan2017quadratic}, it is shown that $X(K_{\mathfrak{p}})_U$ is equal to the set of $z\in X(K_{\mathfrak{p}})$ satisfying \eqref{propositionequation}.
\end{Remark}

To complete the proof of Theorem \ref{biellipticformula}, we need to relate $h(A(b,z))$ to a height pairing between algebraic cycles. This identification is explained in $\S\ref{sec:quotientsheights}$.\subsection{Equations for $X(K_{\mathfrak{p}} )_U $ when the Mordell-Weil rank is larger than the genus}
We briefly consider the case where the rank is larger than the genus. Then the formula becomes more complicated, as to get 
constraints on the height of $A(b,z)$, one needs to know the class of $A_1 (b,z)$ in $H^1 _f (G_T ,V)$, and this can no 
longer be recovered directly from its image in $H^1 _f (G_p ,V)$. Instead one shows that the class of a point in $H^1 (G_T ,V)$ is `overdetermined' by the linear and quadratic relations it satisfies and produces an equation just involving functions on $X(K_{\mathfrak{p}})$ by taking an appropriate resultant.

For convenience, we fix a connected component of $\Sel (U_2 )$ corresponding to $$\alpha =(\alpha _v )\in \prod _{v\in T_0 }j_2 (X(K_v )),$$
and describe 
$$
X(K_{\mathfrak{p}})_{\alpha }:=j_{\mathfrak{p}}^{-1}\loc _{\mathfrak{p}}((\prod _{v \in T_0 }j_v )^{-1}(\alpha ))\subset X(K_{\mathfrak{p}})_U .$$

Suppose that $\rk J(K) = n=g+k$, and that $\rk \NS(J)>k$. Let
$$
(Z_0 ,\ldots ,Z_k):\mathbb{Q}_p (-1)^{k+1} \hookrightarrow \Ker (\wedge ^2 H^1 _{\acute{e}t}(\overline{X})\stackrel{\cup }{\longrightarrow }H^2 _{\acute{e}t}(\overline{X})).
$$
be an injective Galois-equivariant homomorphism, let $U_{Z_m }$ be the quotient of $U_2 $ corresponding to $Z_m$, and let $A_{Z_m } (b)$ denote the corresponding quotient of $A_2 (b)$. For $0\leq m\leq k$, define $\alpha _m $ to be minus the sum of the local heights of $A_{Z_m }(b)^{(P)}$ away from $p$:
$$
\alpha _m :=-\sum _{v\in T_0 }h_v (A_{Z_m } (b)^{(\alpha _v )}).
$$
Let $D_0 ,\ldots ,D_{n-1}$ be elements of $\Pic ^0 (X)$ generating $\Pic ^0 (X)\otimes \mathbb{Q}$.
For $0\leq m\leq k$, let $(a(Z_m )_{ij})_{0\leq i,j<n }$ denote the matrix of the endomorphism of $J(K)\otimes \mathbb{Q}$ 
induced by $Z_m $, and let the image of $IA_{Z_m }(b)$ in $H^1 (G_T ,V)$ equal $\sum c(Z_m )_i \kappa (D_i )$.
Let
$F_m$ in $\mathbb{Q}_p [S_0 ,\ldots ,S_{n-1},T_0 ,\ldots ,T_{n-1}]$ for $0\leq m\leq n$ denote the following polynomial:
$$
\begin{array}{ll} T_m -\sum _{j=0}^{n-1} S_j \int _{D_j} \omega _m, & 0\leq m\leq g-1 \\
T_{m}\! -\! \alpha _{m-g}\! - \! \sum _{0\leq i, j<n}h(D_i ,D_j )S_i (c(Z_{m-g})_j S_j +\sum _{0\leq l< n}a(Z_{m-g})_{lj }S_l), & g\leq m\leq n. \\
\end{array} 
$$

\begin{Proposition}\label{prop2}
Let $F=\Res (F_0 ,\ldots ,F_n )\in \mathbb{Q}_p [T_0 ,\ldots ,T_n ]$ be the resultant of the polynomials $F_0 ,\ldots ,F_n$ with respect to the variables $S_0, \ldots ,S_{n-1}$. Then the set of $z$ in $X(K_{\mathfrak{p}})$ such that
\begin{equation}\nonumber
 F\left(\int ^z _b \omega _0 ,\ldots ,\int ^z _b \omega _{g-1},h_{\mathfrak{p}}(A_{Z_0}(b,z)),\ldots ,h_{\mathfrak{p}}(A_{Z_k}(b,z))\right)=0 
\end{equation}
is finite and contains $X(K_{\mathfrak{p}})_\alpha $.
\end{Proposition}

\section{Chabauty-Kim theory and $p$-adic heights}\label{sec:quotientsheights}
This section is concerned with relating the mixed extensions $A(b,z)$ defined above to the mixed extensions arising 
from the theory of motivic height pairings as developed by Nekov\'a\v r  \cite{Nek93} and Scholl \cite{scholl1991height}. Such relations have been established in the case of fundamental groups 
of affine elliptic curves in work of Balakrishnan and Besser \cite{bb:crelle} and Balakrishnan, Dan-Cohen, Kim and Wewers \cite{balakrishnan2012non} 
and in the case of affine hyperelliptic curves in work of Balakrishnan, Besser and M\"uller \cite{balakrishnan2013p}. 

\subsection{Notation}
In this section, we will repeatedly consider various Ext groups of constructible $\mathbb{Q}_p $-sheaves on $\overline{X}\times \overline{X}$. As all cohomology will be \'etale, we will omit subscripts. For codimension one cycles $Z_1 ,Z_2 \subset X\times X$, we will write $H^i (\overline{X}\times \overline{X}-|Z_1 | ;|Z_2 | )$ to mean 
$$
\ext ^i (j_{1!} j_1 ^* \mathbb{Q}_p ,j_{2!}j^* _2 \mathbb{Q}_p ):=\mathbb{Q}_p \otimes \varprojlim 
\ext ^i (j_{1!} j_1 ^* \mathbb{Z}/p^n \mathbb{Z} ,j_{2!}j^* _2 \mathbb{Z}/p^n \mathbb{Z} ),
$$
where $j_1 $ and $j_2 $ are the open immersions of the complements of $Z_1 $ and $Z_2$ into $X\times X$, and the Ext groups are in the category of constructible sheaves on $\overline{X}\times \overline{X}$. Similarly if $i_1 ,i_2 $ are the closed immersions of $|Z_1 |$ and $|Z_2 |$ into $X\times X$ we write $H^i _{|Z_1|}(\overline{X}\times \overline{X};|Z_2 |)$ to mean $ \Q _p \otimes \varprojlim \ext ^i (i_{1*} i_1 ^* \mathbb{Z}/p^n \mathbb{Z} ,j_{2!}j^* _2 \mathbb{Z}/p^n \mathbb{Z} )$, and so on.
We write $D.E$ to mean the intersection number of the cycles. For a smooth variety $S$ and a 
cycle $E$ in $Z^i (S)$ we write $\widetilde{\cl } _E $ to mean the induced homomorphism
$$
\mathbb{Q}_p (-k)\to H^{2k}_E (\overline{S})
$$
and write $\cl _E $ to mean the composite map
$$
\mathbb{Q}_p (-k)\to H^{2k}_E (\overline{S})\to H^{2k}(\overline{S}).
$$
Finally, to simplify notation we will often write $H^i (X)$ to mean $H^i (\overline{X})$, etc.

\subsection{The height pairing on algebraic cycles}\label{subsec:nekovar}
To relate fundamental groups to $p$-adic heights,  we first explain what the local height functions defined above have to do with height pairings. We restrict attention to the case of the $p$-adic height pairing on the curve $X$. Given a pair $(Z,W)$ of cycles in $\Div ^0 (X)$ with disjoint support $|Z|$ and $|W|$, we construct a mixed extension $H_X(Z,W)$ with graded pieces $\mathbb{Q}_p ,V,$ and $\mathbb{Q}_p (1)$ as a subquotient of 
$H^1 (\overline{X}-|Z|;|W|)(1)$ as follows \cite[$\S 5.6$]{Nek93}.  The representation $H^1  (\overline{X}-|Z|;|W|)(1)$ is a mixed extension with graded pieces $\Ker (H^2 _{|Z|}(\overline{X})\to H^2 (\overline{X}))(1)$, $V$ and
$\Ker (H^2 _{|W|}(\overline{X})\to H^2 (\overline{X}))^*$. Pulling back by 
$$
\mathbb{Q}_p \stackrel{\widetilde{\cl }_Z }{\longrightarrow }\Ker (H^2 _{|Z|}(\overline{X})\to H^2 (\overline{X}))(1)
$$ 
and then pushing out by the dual of 
$$
\mathbb{Q}_p (-1) \stackrel{\widetilde{\cl }_W }{\longrightarrow }\Ker (H^2 _{|W|}(\overline{X})\to H^2 (\overline{X}))
$$
 gives a mixed extension with graded pieces $\mathbb{Q}_p ,V,$ and $\mathbb{Q}_p (1)$, denoted $H_X (Z,W)$. Composing with $h_v$ gives, at each prime, a functional 
$$
(Z,W)\mapsto h_v (H_X(Z,W)).
$$
By \cite[$\S 2$]{Nek93}, this is bi-additive, symmetric, and if $Z=\div (f)$ then 
$$
h_v (Z,W)=\chi _v (f(W)).
$$ 
We denote $h_v (H_X (Z,W))$ simply by $h_v (Z,W)$. Given cycles $Z$ and $W$ in $\Div ^0 (X_K )$ with disjoint support, one defines the global $p$-adic height $h (Z,W)$ associated to $\chi ,s$ to be the sum over all $v$ of $h_v (Z,W)$. The function $h$ is bilinear and factors through $\Pic ^0 (X)\times \Pic ^0 (X)$, unlike the local heights. 

\subsection{Beilinson's formula}\label{subsection:Beilinson}
The proof of the relation to $p$-adic heights starts with a motivic interpretation of $A_n (b,z)$, due to Beilinson \cite[Proposition 3.4]{deligne2005groupes} and is followed by a little diagram chasing. To state Beilinson's theorem, let $Y$ be a smooth geometrically connected variety over a field $K$ of characteristic zero. Let $b$ and $z$ be $K$-rational points of $Y$. 
As before, let 
$$
A_n (Y) (b):=\mathbb{Q}_p [\pi _1 ^{\acute{e}t}(\overline{Y},b)]/I^{n+1}
$$
and
$$
A_n (Y) (b,z):=\mathbb{Q}_p [\pi _1 ^{\acute{e}t}(\overline{Y},b,z)]\otimes _{\mathbb{Q}_p [\pi _1 ^{\acute{e}t}(\overline{Y},b)]}A_n (Y)(b).
$$
\begin{Theorem}[Beilinson {\cite[Proposition 3.4]{deligne2005groupes}}]\label{beilinson34}
 Let $Y^n $ denote the $n$-fold product of $Y$ over $K$. Let $D_0 $ denote $b\times Y^{n-1}$, 
$D_n$ denote $Y^{n-1}\times z$, and for $0<i<n$, define $D_i $ to be the codimension one subscheme of $Y^n $ on which the $i$th and $(i+1)$th coordinates are equal. Then there is a functorial isomorphism of $G_K$-representations
$$
A_n (Y)(b,z)\simeq \left\{ \begin{array}{cc} H^n (\overline{Y}^n ;\bigcup _{i=0}^n D_i )^* & b\neq z \\
H^n (\overline{Y}^n ;\bigcup _{i=0}^n D_i )^* \oplus \mathbb{Q}_p & b =z. \\ \end{array}\right.
$$
\end{Theorem}
We will be interested in applying Theorem \ref{beilinson34} in the case when $n = 2$, for the smooth projective curve $X$ and for the affine curve $Y:=X-x$ obtained by removing $x \in X(K)$. Define $S:=Y\times Y$.

Let $b$ and $z$ be distinct, both not equal to $x$. Define $X_1 :=\{b \}\times X$,$X_2 :=X\times \{z \}$, and define
\begin{equation}\nonumber
 i_1 ,i_2 ,i_{\Delta } :X \hookrightarrow X\times X
\end{equation}
to be the closed immersions with images $X_1 ,X_2,$ and $\Delta $, respectively.
For future use we also let
\begin{equation}\nonumber
\pi _1 ,\pi _2 :X\times X \to X
\end{equation}
denote the projection maps. We use the same notation for the corresponding maps with $X$ and $X\times X$ replaced by $Y$ and $Y\times Y$.

By Beilinson's theorem, the diagram

$$
\begin{tikzpicture}
\matrix (m) [matrix of math nodes, row sep=2em,
column sep=2em, text height=1.5ex, text depth=0.25ex]
{0 & V^{\otimes 2} & A_2 (Y)(b,z) & A_1 (Y)(b,z) & 0 \\
0 & \Coker({\Q _p (1)\stackrel{\cup ^* }{\longrightarrow }V^{\otimes 2}}) & A_2 (X)(b,z) & A_1 (X)(b,z) & 0 \\ };
\path[->]
(m-1-1) edge[auto] node[auto]{} (m-1-2)
(m-1-2) edge[auto] node[auto]{} (m-1-3)
edge[auto] node[auto]{} (m-2-2)
(m-1-3) edge[auto] node[auto]{} (m-1-4)
edge[auto] node[auto]{} (m-2-3)
(m-1-4) edge[auto] node[auto]{} (m-1-5)
edge[auto] node[auto]{} (m-2-4)
(m-2-1) edge[auto] node[auto]{} (m-2-2)
(m-2-2) edge[auto] node[auto]{} (m-2-3)
(m-2-3) edge[auto] node[auto]{} (m-2-4)
(m-2-4) edge[auto] node[auto]{} (m-2-5);
\end{tikzpicture} $$
is dual to
\begin{equation}\label{eq:SES}
\begin{tikzpicture}
\matrix (m) [matrix of math nodes, row sep=2em,
column sep=2em, text height=1.5ex, text depth=0.25ex]
{0 & H^1 (X;\{ b,z \}) & H^2 (X \! \times \! X;X_1 \! \cup \! X_2 \! \cup \! \Delta ) & \Ker (H^1 (X)^{\otimes 2}\stackrel{\cup }{\to }H^2 (X)) & 0 \\
0 & H^1 (Y; \{ b,z \}) & H^2 (S;X_1 \! \cup \! X_2 \! \cup \! \Delta ) & H^1 (X)\otimes H^1 (X) & 0. \\ };
\path[->]
(m-1-1) edge[auto] node[auto]{} (m-1-2)
(m-1-2) edge[auto] node[auto]{} (m-1-3)
edge[auto] node[auto]{} (m-2-2)
(m-1-3) edge[auto] node[auto]{} (m-1-4)
edge[auto] node[auto]{ $ \iota $} (m-2-3)
(m-1-4) edge[auto] node[auto]{} (m-1-5)
edge[auto] node[auto]{} (m-2-4)
(m-2-1) edge[auto] node[auto]{} (m-2-2)
(m-2-2) edge[auto] node[auto]{} (m-2-3)
(m-2-3) edge[auto] node[auto]{} (m-2-4)
(m-2-4) edge[auto] node[auto]{} (m-2-5);
\end{tikzpicture} \end{equation}
Here, the top right morphism is the map 
$$
H^2 (X \times X;X_1 \cup X_2 \cup \Delta )\to \Ker (H^2 (X\times X ;X_1 \cup X_2 )\to H^2 (X;b \cup z))
$$
 composed with the isomorphism
$$
\Ker (H^2 (X\times X ;X_1 \cup X_2 )\to H^2 (X; \{ b,z\} )) \simeq \Ker (H^1 (X)\otimes H^1 (X)\stackrel{\cup }{\longrightarrow }H^2 (X))
$$
coming from the commutative diagram
$$
\begin{tikzpicture}
\matrix (m) [matrix of math nodes, row sep=2em,
column sep=2em, text height=1.5ex, text depth=0.25ex]
{H^2 (X\times X;X_1 \cup X_2 ) & H^1 (X;b )\otimes H^1 (X;z) & H^1 (X)\otimes H^1 (X) \\
H^2 (X; \{ b,z\} ) & H^2 (X; \{ b,z\} ) & H^2 (X) \\ };
\path[->]
(m-1-1) edge[auto] node[auto]{$\simeq $} (m-1-2)
edge[auto] node[auto]{$\Delta _X ^* $} (m-2-1)
(m-1-2) edge[auto] node[auto]{$\simeq $} (m-1-3)
edge[auto] node[auto]{$\cup $} (m-2-2)
(m-1-3) edge[auto] node[auto]{$\cup $} (m-2-3)
(m-2-1) edge[auto] node[auto]{$\simeq $} (m-2-2)
(m-2-2) edge[auto] node[auto]{$ \simeq $} (m-2-3);
\end{tikzpicture} 
$$
and the bottom right map is similarly coming from an isomorphism $H^2 (S;X_1 \cup X_2 )\simeq H^1 (X)^{\otimes 2}$.

Via K\"unneth projectors, we have a cycle class map $\cl _Z :\Q _p (-1)\to H^1 (X)\otimes H^1 (X)$. Via the cycle class map, we may pull back the bottom row of $\eqref{eq:SES}$ to obtain an extension of $\Q _p (-1)$ by $H^1 (Y;\{ b,z \})$, giving a mixed extension with graded pieces $\Q _p (-1),V(-1)$ and $\Q _p $. Let $E_Z =E_Z (b,z)$ be the mixed extension with graded pieces $\Q _p ,V$ and $\Q_p (1)$ obtained by twisting this by $\Q _p (1)$.
As explained above, $E_Z$ is the Tate dual of $A_Z (Y)(b,z)$.
If the intersection number of $Z$ with $\Delta -X_1 -X_2 $ is zero, then its cycle class lies in the image of $\iota $, hence in this case we may pull back the top row of \eqref{eq:SES} by $\cl _Z$, and then $E_Z$ is the Tate dual of $A_Z (X)(b,z)$.

\subsection{$h(A(b,z))$ as a height pairing between algebraic cycles}\label{subsection:heart}
Via the cohomological characterisation of $A_Z (b,z)$, describing the local heights of $A(b,z)$ 
in terms of the height pairings on $X$ amounts to finding divisors $D_1 $, $D_2$ in $\Div ^0 (X)$, and an isomorphism
between the subquotient $H^1 (\overline{Y}-|D_1 |;|D_2 |)$ corresponding to $h(D_1 ,D_2 )$  and the subquotient of $H^2 (S;X_1 \cup X_2 \cup \Delta )$ corresponding to $A(b,z)$.

Let $Z$ be a divisor of $\overline{S}$ intersecting $X_1$, $X_2$, and $\Delta$ properly. We 
somewhat abusively denote the composite map
\begin{equation}\nonumber
\mathbb{Q}_p (-1)\stackrel{\cl _{Z}}{\longrightarrow }H^2 (\overline{S})\to 
H^1 (\overline{X})^{\otimes 2}\stackrel{\simeq }{\longrightarrow }H^2 (\overline{S};X_1 \cup X_2 )
\end{equation}
by $\cl _Z$,
where the last map is the isomorphism defined above.
\begin{Definition}\label{defnofdbz}
Define $D(b,z)\in \Div ^0 (X)$ to be the cycle $$i_{\Delta } ^* Z -i_1 ^* Z -i_2 ^*Z +(Z.X_1 +Z.X_2 -Z.\Delta )x.$$
\end{Definition}

The following theorem says that the mixed extension 
$A(b,z)$ is exactly the one built out of the degree zero divisors $z-b$ and $D(b,z)$. In \cite[Theorem 2.2]{darmon2012iterated}, Darmon, Rotger and Sols proved that the Abel-Jacobi class of $D(b,z)$ is equal to the extension of $\mathbb{Z}$-mixed Hodge structure corresponding to the motive whose \'etale realisation is $IA(b,z)$. This generalised previous work of Kaenders \cite{kaenders2001mixed}. The theorem below refines this to determine $A(b,z)$ as a mixed extension of $\kappa (z-b)$ and $IA(b,z)^* (1)$.
\begin{Theorem}\label{height}
Let $Z$ be any codimension 1 cycle in $X\times X$ whose image in $H^2 (S)$ is nonzero.
The mixed extension $E_Z$ is isomorphic to $H_X (z-b,i_{\Delta } ^* Z -i_1 ^* Z -i_2 ^* Z +mx)(-1)$, where $m$ 
is the intersection number of $Z$ with $X_1 +X_2 -\Delta$, and $H_X$ is Nekov\'a\v r's mixed extension construction defined in Section \ref{subsec:nekovar}.
\end{Theorem}
Before giving the proof of this theorem, we explain how it completes the proof of Theorem \ref{biellipticformula}.
\begin{proof}[Proof of Theorem \ref{biellipticformula}]
By Theorem \ref{height}, for all $v$,
\[
h_v (A(b,z))=h_v (E_Z )=h(z-b,D(b,z)),
\]
as $E_Z$ and $H_X (z-b,D(b,z))$ are isomorphic mixed extensions and $h(z-b,D(b,z))=h(H_X(z-b,D(b,z)))$ by definition.
Hence Theorem \ref{biellipticformula} follows from Proposition \ref{prop:vital}.
\end{proof}
\begin{Remark}\label{affine_version}
One may also use Theorem \ref{height} to turn Remark \ref{rmk:integral_points} into a formula computing a finite set of points containing $\mathcal{Y}(\mathcal{O}_K )$. More precisely, if $b$ is an integral point of $\mathcal{Y}$, and $Z$ is a cycle with nonzero image in $\wedge ^2 H^1 (\overline{X})$, then for all $v$ not dividing $\mathfrak{p}$, $h_v (z-b,(i_\Delta ^* -i_1 ^* -i_2 ^* )Z+mx)$ takes only finitely many values and is identically zero on all primes of good reduction, and one obtains a formula for a finite set containing $\mathcal{Y}(\mathcal{O}_{\mathfrak{p}})_2 \cap X' (K_{\mathfrak{p}})$ in terms of $h_{\mathfrak{p}} (z-b,(i_\Delta ^* -i_1 ^* -i_2 ^* )Z+mx)$ and $\log _J$ in an analogous manner.

\end{Remark}
\begin{proof}[Proof of Theorem \ref{height}]
For any cycle $W\subset X$ we have a commutative diagram with exact columns and rows
$$
\footnotesize
\begin{tikzpicture}
\matrix (m) [matrix of math nodes, row sep=2em,
column sep=2em, text height=1.5ex, text depth=0.25ex]
{H^1 _{|i_{\Delta } ^* W|  }(Y;\{ b , z \}) & H^2 _{|W| } (S;X_1 \cup X_2 \cup \Delta   ) & H^2 _{ |W|} (S; X_1 \cup X_2 ) \\
H^1 (Y;\{b , z \}) & H^2 (S ;X_1 \cup X_2 \cup \Delta   ) & H^2 (S;X_1 \cup X_2 )\\
H^1 (Y -|i_{\Delta } ^* W| ;\{ b,z \}) & H^2 (S -|W| ;X_1 \cup X_2 \cup \Delta  ) & H^2 (S-|W|;X_1 \cup X_2 ). \\ };
\path[->]
(m-1-1) edge[auto] node[auto]{} (m-2-1)
edge[auto] node[auto] {} (m-1-2)
(m-1-2) edge[auto] node[auto]{} (m-2-2)
edge[auto] node[auto] {} (m-1-3)
(m-1-3) edge[auto] node[auto]{} (m-2-3)
(m-2-1) edge[auto] node[auto]{} (m-3-1)
edge[auto] node[auto] {} (m-2-2)
(m-2-2) edge[auto] node[auto]{} (m-3-2)
edge[auto] node[auto] {} (m-2-3)
(m-2-3) edge[auto] node[auto]{} (m-3-3)
(m-3-1) edge[auto] node[auto]{} (m-3-2)
(m-3-2) edge[auto] node[auto]{} (m-3-3);
\end{tikzpicture} $$
To prove the theorem, we first find a cycle $W$ such that the image of $\cl_Z (\mathbb{Q}_p (-1))$ in 
$H^2 (S-|W|;X_1 \cup X_2 )$ is zero. This identifies $E_Z$ with a subspace of \\
$H^1 (Y -|i_{\Delta } ^* W| ;\{ b,z \}) $. One then determines the subspace exactly by giving a cohomological interpretation of the inclusion of the weight 2 part of $E_Z$ inside the weight 2 part of $H^1 (\overline{Y} -|i_{\Delta } ^* W| ;\{ b, z \})$.

Suppose $i_1 ^* Z =\sum n_i x_i$. Then $\pi _2 ^* i_1 ^* Z =\sum n_i x_i \times \overline{X}$.  Similarly, define $\pi _1 ^* i_2 ^* Z $. 
Define 
\begin{equation}\nonumber
W:=Z -\pi _2 ^* i_1 ^* Z-\pi _1 ^* i_2 ^* Z .
\end{equation}

\begin{Lemma}
The image of $\cl_Z (\mathbb{Q}_p (-1))$ in $H^2 (S-|W|;X_1 \cup X_2 )$ is zero.
\end{Lemma}
\begin{proof}
Let $D:=X\times X-S$.
It is enough to show that $\cl_Z (\mathbb{Q}_p (-1))$ is in the image of 
$$
H^2 _{|W| \cup D}(X\times X;X_1 \cup X_2 )\to H^2 (S;X_1 \cup X_2 ).
$$
Let $W_1 :=|i_1 ^* W|\cup i_1 ^{-1}D$ and $W_2 :=|i_2 ^* W|\cup i_2 ^{-1}D$.
There is a commutative diagram with exact rows
$$
\footnotesize
\begin{tikzpicture}
\matrix (m) [matrix of math nodes, row sep=2em,
column sep=2em, text height=1.5ex, text depth=0.25ex]
{0 & H^2 _{|W|\cup D}(\overline{X}\times \overline{X};X_1 \cup X_2 ) & H^2 _{|W|\cup D}(\overline{X}\times \overline{X}) & 
H^2 _{W_1} (X) \oplus H^2 _{W_1}(X )  \\
& H^2 (\overline{X}\times \overline{X};X_1 \cup X_2 ) & H^2 (\overline{X}\times \overline{X}) & H^2  (X_1 )  \oplus H^2 (X_2 ). \\};
\path[->]
(m-1-1) edge[auto] node[auto]{} (m-1-2)
(m-1-2) edge[auto] node[auto]{} (m-2-2)
edge[auto] node[auto] {} (m-1-3)
(m-1-3) edge[auto] node[auto]{} (m-2-3)
edge[auto] node[auto] {} (m-1-4)
(m-1-4) edge[auto] node[auto]{} (m-2-4)
(m-2-2) edge[auto] node[auto]{}  (m-2-3)
(m-2-3) edge[auto] node[auto]{} (m-2-4);
\end{tikzpicture} $$
The class of $Z$ in $H^2 (\overline{X}\times \overline{X})$ lifts to an element of $H^2 _{W\cup D}(\overline{X}\times \overline{X})$ by construction.
Hence to show $\cl_Z (\mathbb{Q}_p (-1))$ lifts to an element of $H^2 _{W\cup D}(\overline{X}\times \overline{X};X_1 \cup X_2 )$, 
it is enough to show that it lies in the kernel of 
$$H^2 _{W\cup D}(\overline{X}\times \overline{X}) \stackrel{i_1 ^* \oplus i_2 ^*}{\longrightarrow} H^2 _{W_1} (X) \oplus H^2 _{W_1}(X ). $$
This is the case since, in $H^2 _{W_1} (X)$, we have
$i_1 ^* \pi _2 ^* i_1 ^* Z=i_1 ^* Z$ and $i_1 ^* \pi _1 ^* i_2 ^* Z=0$, and similarly for $H^2 _{W_2 }(X)$.
\end{proof}
Hence we deduce that $E_Z$ is a subobject of $H^1 (Y-|i_{\Delta } ^* W|; \{b,z \})$, and all that remains is to determine the homomorphism
$$
\mathbb{Q}_p (-1)\to H^2 _{|W|\cup x} (X)
$$
induced by this identification.
Let 
 $\delta :\Ker (\gamma )\to \Coker (\alpha )$ 
denote the connecting homomorphism associated to 
$$
\small
\begin{tikzpicture}
\matrix (m) [matrix of math nodes, row sep=2em,
column sep=1em, text height=1.5ex, text depth=0.25ex]
{ & H^1 (Y;  \{b,z\}  ) & H^2 ( S ; X_1 \! \cup \! X_2 \! \cup \! \Delta ) & H^2 ( S ; X_1 \! \cup \! X_2 ) & 0 \\ 
0 & H^1 (Y-|i_{\Delta } ^* W| ; \{b,z \} ) & H^2 (S -|W| ;X_1 \! \cup \! X_2 \! \cup \! \Delta ) & H^2 (S -|W| ;X_1 \! \cup \! X_2 ). & \\};
\path[->]
(m-1-2) edge[auto] node[auto]{ $\alpha  $ } (m-2-2)
edge[auto] node[auto] {} (m-1-3)
(m-1-3) edge[auto] node[auto]{$\beta  $} (m-2-3)
edge[auto] node[auto] {} (m-1-4)
(m-1-4) edge[auto] node[auto]{$\gamma  $} (m-2-4)
edge[auto] node[auto] {} (m-1-5)
(m-2-1) edge[auto] node[auto]{} (m-2-2)
(m-2-2) edge[auto] node[auto]{} (m-2-3)
(m-2-3) edge[auto] node[auto]{} (m-2-4);
\end{tikzpicture} $$
Then by construction, $E_{Z}$ is isomorphic to the pullback of 
$H^1 (\overline{Y}-|i_{\Delta } ^* W|; \{ b,z \})$ by the homomorphism
$$
\mathbb{Q}_p (-1)\to \Ker (\gamma )\stackrel{\delta }{\longrightarrow }\Coker (\alpha )\to H^2 _{|i_{\Delta } ^* W|}(Y; \{ b,z \} ).
$$
We claim that the diagram 
$$ \small
\begin{tikzpicture}
\matrix (m) [matrix of math nodes, row sep=2em,
column sep=2em, text height=1.5ex, text depth=0.25ex]
{  \Ker (\gamma ) & \Coker (\alpha ) \\
H^2 _{|W|}(S;X_1 \cup X_2 ) & 
H^2 _{|i_{\Delta } ^* W|} (Y; \{ b,z \} ) \\};
\path[->]
(m-2-1) edge[auto] node[auto] {} (m-1-1)
edge[auto] node[auto] {$i_{\Delta } ^* $} (m-2-2)
(m-1-1) edge[auto] node[auto] {$\delta $} (m-1-2)
(m-1-2) edge[auto] node[auto] {} (m-2-2);
\end{tikzpicture} $$
commutes. This follows from the definition of the long exact sequence in cohomology associated to a short exact sequence of sheaves: for example, it is implied by the following lemma, whose proof we sketch.
\begin{Lemma}
For $1\leq i,j\leq 3$, let $I^\bullet _{i,j}$ be complexes of abelian groups, and let
$$ \small
\begin{tikzpicture}
\matrix (m) [matrix of math nodes, row sep=2em,
column sep=2em, text height=1.5ex, text depth=0.25ex]
{& 0 & 0 & 0 & \\
0 & I^\bullet _{1,1} & I^\bullet _{1,2} & I^\bullet _{1,3} & 0 \\
0 & I^\bullet _{2,1} & I^\bullet _{2,2} & I^\bullet _{2,3} & 0  \\
0 & I^\bullet _{3,1} & I^\bullet _{3,2} & I^\bullet _{3,3} & 0  \\
& 0 & 0 & 0 & \\};
\path[->]
(m-1-2) edge[auto] node[auto] { } (m-2-2)
(m-1-3) edge[auto] node[auto] { } (m-2-3)
(m-1-4) edge[auto] node[auto] { } (m-2-4)
(m-2-1) edge[auto] node[auto] {} (m-2-2)
(m-2-2) edge[auto] node[auto] { } (m-2-3)
edge[auto] node[auto] { } (m-3-2)
(m-2-3) edge[auto] node[auto] { } (m-2-4)
edge[auto] node[auto] { } (m-3-3)
(m-2-4) edge[auto] node[auto] { } (m-2-5)
edge[auto] node[auto] { } (m-3-4)
(m-2-4) edge[auto] node[auto] { } (m-2-5)
(m-3-1) edge[auto] node[auto] {} (m-3-2)
(m-3-2) edge[auto] node[auto] { } (m-3-3)
edge[auto] node[auto] { } (m-4-2)
(m-3-3) edge[auto] node[auto] { } (m-3-4)
edge[auto] node[auto] { } (m-4-3)
(m-3-4) edge[auto] node[auto] { } (m-3-5)
edge[auto] node[auto] { } (m-4-4)
(m-3-4) edge[auto] node[auto] { } (m-3-5)
(m-4-1) edge[auto] node[auto] {} (m-4-2)
(m-4-2) edge[auto] node[auto] { } (m-4-3)
edge[auto] node[auto] { } (m-5-2)
(m-4-3) edge[auto] node[auto] { } (m-4-4)
edge[auto] node[auto] { } (m-5-3)
(m-4-4) edge[auto] node[auto] { } (m-4-5)
edge[auto] node[auto] { } (m-5-4)
(m-4-4) edge[auto] node[auto] { } (m-4-5);
\end{tikzpicture} $$
be a commutative diagram of abelian groups with exact columns and rows.
Define
\begin{align*}
J_1 & :=\Ker (H^i (I_{2,3}^{\bullet })\to H^{i+1} (I_{2,1}^{\bullet })), \\
J_2 & :=\Coker (H^{i-1} (I_{3,3}^{\bullet })\to H^{i} (I_{3,1}^{\bullet })), \\
K_1 & :=\Ker (H^i (I_{1,3}^{\bullet })\to H^{i+1} (I_{2,1}^{\bullet })), \\
K_2 & :=\Coker (H^{i-1} (I_{3,3}^{\bullet })\to H^{i+1} (I_{1,1}^{\bullet })).\end{align*}
Let 
$$
\delta :\Ker(J_1 \to H^i (I_{3,3}^{\bullet })) \to \Coker(H^i (I_{2,1}^{\bullet })\to J_2 )
$$
be the connecting homomorphism associated to $$
\begin{tikzpicture}
\matrix (m) [matrix of math nodes, row sep=2em,
column sep=2em, text height=1.5ex, text depth=0.25ex]
{& H^i (I_{2,1}) & H^i (I_{2,2}) & J_1 & 0 \\
0 & J_2 & H^i (I_{3,2}) & H^i (I_{3,3}). \\};
\path[->]
(m-2-1) edge[auto] node[auto] {} (m-2-2)
(m-1-2) edge[auto] node[auto] { } (m-1-3)
edge[auto] node[auto] { } (m-2-2)
(m-1-3) edge[auto] node[auto] { } (m-1-4)
edge[auto] node[auto] { } (m-2-3)
(m-1-4) edge[auto] node[auto] {} (m-1-5)
edge[auto] node[auto] { } (m-2-4)
(m-2-2) edge[auto] node[auto] { } (m-2-3)
(m-2-3) edge[auto] node[auto] { } (m-2-4);
\end{tikzpicture} $$
Then the diagram
$$ \small
\begin{tikzpicture} 
\matrix (m) [matrix of math nodes, row sep=1em,
column sep=2em, text height=1.5ex, text depth=0.25ex]
{ & \Ker(J_1 \to H^i (I_{3,3})) & \Coker(H^i (I_{2,1})\to J_2 ) & \\
K_1 & & & K_2 \\
& H^i (I_{1,3})  &  H^{i+1} (I_{1,1}) & \\};
\path[->]
(m-2-1) edge[auto] node[auto] {} (m-1-2)
edge[auto] node[auto] {} (m-3-2)
(m-1-2) edge[auto] node[auto] {$\delta $} (m-1-3)
(m-1-3) edge[auto] node[auto] {} (m-2-4)
(m-3-2) edge[auto] node[auto] {} (m-3-3)
(m-3-3) edge[auto] node[auto] {} (m-2-4);
\end{tikzpicture} $$
commutes.
\end{Lemma}
\begin{proof}
Let $d_{i,j}^k $ be the differential $I_{i,j}^k \to I_{i,j}^{k+1}$ and let $Z^k _{i,j}=\Ker (d_{i,j}^k )$. 
Consider the following function from $K_1$ and $K_2$: start with $v_1 $ in $K_1 $, lift to $v_2 $ in  $Z^i _{1,3}$, lift that to get $v_3 $ in $I^i _{2,2}$, take differentials to get $ v_4$  in $Z^{i+1}_{2,2}$, check that this can be lifted to $v_5$ in $Z^{i+1}_{1,1}$, take its image in $K_2 $. 
We claim the top and bottom maps from $K_1 $ to $K_2$ are both instances of this construction. In the top map, one starts with an element in $Z^i _{1,3}$, maps it to an element of $Z^i _{2,3}$, lifts it to an element of $Z^i _{2,2}$, maps it down to $Z^i _{3,2}$, lifts it to an element of $Z^i _{3,1}$, lifts that to an element of $I_{2,1}^i $, maps it to an element of $Z^{i+1}_{2,1}$ and finally lifts that to an element of $Z_{1,1}^{i+1}$. In the bottom map, one starts with an element in $Z^i _{1,3}$, lifts it to an element of $I^i _{1,2}$, maps that down to an element of 
$Z^{i+1}_{1,2}$, and then lifts that to an element of $Z^{i+1}_{1,1}$. This proves the claim, since $I^\bullet _{1,2}$ and $I^\bullet _{2,1}$ are both subcomplexes of $I^\bullet _{2,2}$, and the differentials on $I^\bullet _{1,2}$ and $I^\bullet _{2,1}$ are just the restriction of the differential on $I_{2,2}^\bullet $.
\end{proof}
By commutativity of the diagram
 $$
\begin{tikzpicture}
\matrix (m) [matrix of math nodes, row sep=2em,
column sep=2em, text height=1.5ex, text depth=0.25ex]
{  H^2 _{|W|}(S;X_1 \cup X_2 ) & H^2 _{|W|}(S) \\
H^2 _{|i_{\Delta } ^* W|}(Y; \{b,z \}) & H^2 _{|i_{\Delta } ^* W|}(Y) \\};
\path[->]
(m-1-1) edge[auto] node[auto] {} (m-1-2)
edge[auto] node[auto] {$i_{\Delta } ^* $} (m-2-1)
(m-2-1) edge[auto] node[auto] {$\simeq $} (m-2-2)
(m-1-2) edge[auto] node[auto] {$i_{\Delta } ^* $} (m-2-2);
\end{tikzpicture} $$
we deduce that $E_Z$ is isomorphic to the pullback of $H^1 (Y-|i_{\Delta } ^* W|; \{ b,z \})$ by 
$$
\widetilde{\cl }_{i_{\Delta } ^* W} :\mathbb{Q}_p (-1)\to H^2 _{|i_{\Delta } ^* W|}(Y).
$$
Finally, we show that this  implies that the map 
$$
\mathbb{Q}_p (-1)\to \Ker (H^2 _{|i_{\Delta } ^* W|\cup x}(X)\to H^2 (X))
$$
is equal to 
$$\widetilde{\cl }_{i_{\Delta } ^* W-(W.\Delta )x} \to H^2 _{|i_{\Delta } ^* W|\cup x}(X).$$
Via the isomorphism 
$
H^1 (X; \{ b,z\} ) \simeq H^1 (Y; \{ b,z\} ),
$
one obtains an isomorphism
$$
H^2 _{|i_{\Delta } ^* W|}(Y)\simeq \Ker (H^2 _{|i_{\Delta } ^* W|\cup x}(X)\to H^2 (X))
$$
which sends the class of a cycle $\sum d_i (z_i )$ with support in $W\cap Y$ to $\sum d_i (z_i )-(\sum d_i )x$.
 This completes the proof of the theorem.
\end{proof}

\section{$p$-adic heights on hyperelliptic curves}\label{alltheheightcomputations}

In this section, we recall facts about $p$-adic height pairings and use them to relate the height pairing of the cycles $z-b$ and 
$D(b,z)$ to the height pairings arising in Theorems \ref{thm:oldqc} and \ref{biellexample}. We fix a choice of idele class character $\chi $ and an isotropic splitting $s$ of the Hodge filtration on $H^1 _{\dR }(X_{K_{\mathfrak{p}} })$.

By the work of  Besser \cite{Bes99a}, Nekov\'a\v r's $p$-adic height pairing is equal to the $p$-adic height pairing of Coleman and Gross defined in \cite{coleman-gross}. In \cite[$\S 2 $]{bb:crelle}, it is shown that one may extend the Coleman-Gross local height pairing to divisors with non-disjoint support, although as in the case of the real-valued height pairing, such an extension will, in general, depend on a choice of a global tangent vector at each point. As explained in \cite{balakrishnan2013p}, there is a canonical choice of such a tangent vector when $X$ is a hyperelliptic curve with a fixed odd degree model.

We write $h_v (D)$ to mean $h_v (D,D)$, and $h(D)$ to mean $\sum _v h_v (D)$.
When $X=E$ is an elliptic curve with origin $\infty $, for $z$ in $E(K_v )$ we define
$$
h_v (z):=h_v ((z)-(\infty )).
$$
\subsection{Height identities}
Let $X$ be a hyperelliptic curve, and let $w$ denote the hyperelliptic involution on $X$. In this subsection, we briefly review the theory of height pairings on hyperelliptic curves \cite{bb:crelle, bb:heights}.
\begin{Definition}
For a divisor $D$ on $X$,  define $D^+ :=D+w^* D$ and $D^- :=D-w^* D.$
\end{Definition}
\begin{Lemma}
For any divisors $D_1, D_2 \in \Div^0(X)$,
$$
h_v (D_1 ,D_2 )=\frac{1}{4}h_v (D_1 ^+ ,D_2 ^+ )+\frac{1}{4}h_v (D_1 ^- ,D_2 ^- ).
$$
\end{Lemma}
Part (i) of the next lemma is proved in \cite{balakrishnan2013p} (see (4.3) and the subsequent discussion). Part (ii) also follows straightforwardly from the proof.
\begin{Lemma}\label{anotherheightidentity}
Let $X$ be a hyperelliptic curve of genus $g$, defined by a monic odd degree model $y^2 =f(x)$.
Let $\infty $ denote the point at infinity. \\
(i) Let $z$ be a point of $X$ not equal to $\infty $, with $y(z)\neq 0$. Then 
$$
h_v (z^+ -2\infty)=2\chi _v (y(z))+2\chi _v (2).
$$
(ii) Let $z_1 ,z_2 $ be points of $X$ not equal to $\infty $. Suppose $x(z_1 )\neq x(z_2 )$. Then 
$$
h_v (z_1 ^+ -2\infty,z_2 ^+ -2\infty )=2\chi _v (x(z_1 )-x(z_2 )).
$$
\end{Lemma}
\begin{proof}
As explained in \cite[$\S 4$]{balakrishnan2013p}, one finds that normalised parameters at $z$ and $w(z)$ are given by $x-x(z)/2y(z)$, and that 
$-y/x^{g+1}$ is a normalised parameter at infinity. The lemma now follows from the definition of the Coleman-Gross pairing
on divisors of non-disjoint support.
\end{proof}
\begin{Lemma}\label{doesntmatterwhichformula}
Let $E$ be an elliptic curve
$$
y^2 =x^3 +ax^2 +bx+c.
$$
Then for any $z_1 ,z_2 $ in $E$ both not equal to $\infty $, and with $x(z_1 )\neq x(z_2 )$,
$$ 2h_v (z_1 -\infty )+2h_v (z_2 -\infty )-h_v (z_1 -z_2 )-h_v (z_1 -w(z_2 )) =  2 \chi _v (x(z_1 )-x(z_2 )).$$
\end{Lemma}
\begin{proof}
We first break the left hand side into symmetric and antisymmetric parts. The antisymmetric part equals
$$
\frac{1}{2}h_v (z_1 ^- )+\frac{1}{2}h_v (z_2 ^-)-\frac{1}{4}h_v (z^- _1  -z^- _2)-\frac{1}{4}h_v (z^- _1  +z^- _2).
$$
By expanding, this can be seen to be zero. The symmetric part equals 
$$
 \frac{1}{2}h_v (z_1 ^+ -2\infty )+\frac{1}{2}h_v (z_2 ^+ -2\infty )-\frac{1}{2}h_v (z_1 ^+ -z_2 ^+ ).
$$
Expanding, this equals
$$
\frac{1}{2}h_v (z_1 ^+ -2\infty ,z_2 ^+ -2\infty )+\frac{1}{2}h_v (z_2 ^+ -2\infty ,z_1 ^+ -2\infty ),
$$
hence the result now follows from Lemma \ref{anotherheightidentity}.
\end{proof}
\begin{Lemma}\label{height:zwz}
For any $z$ not equal to $\infty $,
$$
h_v (z-\infty ,w(z)-\infty )+h_v (z-\infty ,z-\infty )=\chi _v (2y(z)).
$$
\end{Lemma}
\begin{proof}
The antisymmetric parts of $h_v (z-\infty ,w(z)-\infty )$ and $h_v (z-\infty ,z-\infty )$ cancel out, hence the left hand side is equal to
$
\frac{1}{2}h_v (z^+ -2 \infty ),
$
which equals $\chi _v (2y(z))$ by Lemma \ref{anotherheightidentity}.
\end{proof}
\subsection{Integral points on hyperelliptic curves}
Let $X$ be a hyperelliptic curve given by an equation of the form
$
y^2 =f(x),
$
where $f(x)$ is a monic polynomial in $\mathcal{O}_K [x]$ of degree $2g+1$. Let $Y=X-\infty $. Take $Z$ to be the cycle 
$\Gamma _w =\{ (z,w(z)) \} \subset X\times X$. Let $\{ z_1 ,\ldots ,z_{2g+1} \}$ denote the set of $\overline{K}$ points of $X$ with 
$y$-coordinate zero, and let $W$ denote the divisor $\sum _i z_i $. Let $b$ and $z$ be points of $Y$ with nonzero $y$-coordinate.
Then 
$$\begin{array}{ccc}
i_1 ^* \Gamma _w = w(b), &
i_2 ^* \Gamma _w = w(z), &
i_{\Delta } ^* \Gamma _w = W+\infty,
\end{array}$$
hence $
D(b,z)=W-w(b)-w(z)-(2g-1)\infty .
$
So the class of $A(Y)(b,z)$ is dual to $H _X (z-b,W-w(b)-w(z)-(2g-1)\infty )$, by Theorem \ref{height}.
The following lemma illustrates how Theorem \ref{thm:oldqc} may be deduced from Theorem \ref{height} together with the affine version of Theorem \ref{biellipticformula}.
\begin{Lemma}\label{asintheproof}
For any prime $v$,
$$
h_v (z-b,D(b,z))=h_v (z-\infty )-h_v (b-\infty ).
$$
\end{Lemma}
\begin{proof}
First, note that additivity yields
$$h_v (z-b,D(b,z))=h_v (z-b,W-(2g+1)\infty )-h_v (z-b,2\infty -w(z)-w(b)).$$
Since $2(g+1)\infty -W=\div (y)$, the first term is equal to $\chi (y(z))-\chi (y(b))$. For the second term, since 
$z-b$ and $2\infty -w(z)-w(b)$ are disjoint, 
$$h_v (z-b,2\infty -w(z)-w(b))=\frac{1}{2}h_v (z-b,2\infty -w(z)-w(b))+\frac{1}{2}h_v (2\infty -w(z)-w(b),z-b).$$
By additivity,
\begin{align*}
h_v (z-b,2\infty -w(z)-w(b)) &= h_v (z-\infty ,\infty -w(z))+h_v (z-\infty ,\infty -w(b)) \\ 
&\quad +h_v (\infty -b,\infty -w(z))+h_v (\infty -b,\infty -w(b))
\end{align*}
and similarly for $h_v (2\infty -w(z)-w(b),z-b)$. Using the fact that $h_v (D_1 ,D_2 )=h_v (w(D_1 ),w(D_2 ))$, this gives 
$$h_v (z-b,2\infty -w(z)-w(b))=h_v (z-\infty ,\infty -w(z))+h_v (\infty -b,\infty -w(b)).$$
The result now follows from Lemma \ref{height:zwz}.
\end{proof}
\subsection{Rational points on bielliptic curves}\label{7.3}
In this subsection we return to the case where $X$ is a genus 2 curve of the form $
y^2 =x^6 +a_4x^4 +a_2x^2 +a_0,$
and explain how to deduce Theorem \ref{biellexample} from Theorem \ref{biellipticformula}.  Let $h_v $ and $h$ denote (local and global, resp.) heights on $X$, $h_{E_1 ,v}$ and $h_{E_1 }$ heights on $E_1$, and $h_{E_2 ,v}$ and $h_{E_2 }$ heights on $E_2$.  Recall from the introduction the associated elliptic curves 
$$E_1: y^2 =x^3 +a_4x^2 +a_2x+a_0 \qquad \qquad E_2: y^2 =x^3 +a_2x^2 +a_4a_0x+a_0^2$$
 and morphisms $f_i :X\to E_i $.
Define $Z_1 ,Z_2 \subset X\times X$ to be the graphs of the automorphisms $g_1 :(x,y) \mapsto (-x,y)$ and  
$g_2 :(x,y)\mapsto (-x,-y)$ respectively.
 As explained at the end of $\S 6.3$, the fact that the intersection number of $Z_1 -Z_2 $ with $\Delta -X_1 -X_2 $ is zero implies that $Z = Z_1 - Z_2$ defines a quotient of the fundamental group of $\overline{X}$, and a quotient $A(b,z)$  of $A(X)(b,z)$.
Note that 
\begin{align*}
i_1 ^* (Z_1 -Z_2 ) &= g_1 (z)-g_2 (z), \quad i_2 ^* (Z_1 -Z_2 ) = g_1 (b)-g_2 (b), \\
i_{\Delta } ^* (Z_1 -Z_2 ) &= (0,\sqrt{a_0})+(0,-\sqrt{a_0})-\infty -w(\infty ),
\end{align*}
so  $D(b,z)=(0,\sqrt{a_0})+(0,-\sqrt{a_0})-\infty -w(\infty )-g_1 (z)+g_2 (z)-g_1 (b)+g_2 (b).$
The following lemma completes the proof of Theorem \ref{biellexample}.
\begin{Lemma}\label{completingproof}
For any $b$ and $z$ with $x(b)\neq x(z)$ and both not equal to zero or infinity,
\begin{align*}
h_v (z-b,D(b,z)) &= h_{E_1 ,v} (f_1 (z)-\infty )-h_{E_1 ,v} (f_1 (b)-\infty )-h_{E_2 ,v} (f_2 (z)-\infty )\\
&\quad +h_{E_2 ,v} (f_2 (b)-\infty )+ 2\chi (x(b))-2\chi (x(z)).
\end{align*}\end{Lemma}
\begin{proof}
For $i=1,2$, let $D_i$ denote the divisor $w(f_i (z))+w(f_i (b))-2\infty $.
Then
\begin{align*}
D(b,z) &= -\infty -w(\infty ) + (0,\sqrt{a_0}) + (0,-\sqrt{a_0})-g_1 (z)+g_2 (z)-g_1 (b)+g_2 (b)\\
&=f_1 ^* (D_1 )-f_2 ^* (D_2 ),\end{align*}
hence
\begin{align*}
h_v (z-b,D(b,z)) &= h_{E_1 ,v}(f_1 (z)-f_1 (b),w(f_1 (z))+w(f_1 (b))-2\infty ) \\
&\quad - h_{E_2 ,v}(f_2 (z)-f_2 (b),w(f_2 (z))+w(f_2 (b))-2\infty ).
\end{align*}
As in the proof of Lemma \ref{asintheproof},
 \begin{align*}h_{E_1 ,v}(f_1 (z)-f_1 (b),w(f_1 (z))+w(f_1 (b))-2\infty )& =h_{E_1 ,v} (f_1 (z)-\infty )- h_{E_1 ,v} (f_1 (b)-\infty )\\
 &\quad +\chi (y(f_1 (z)))-\chi (y(f_1 (b)))\end{align*}
and similarly for $f_2 $.
Hence
\begin{align*}h_v (z-b,D(b,z)) &= h_{E_1 ,v} (f_1 (z)-\infty )-h_{E_1 ,v} (f_1 (b)-\infty )-h_{E_2 ,v} (f_2 (z)-\infty )\\
&\quad+ h_{E_2 ,v} (f_2 (b)-\infty ) +\chi (y(f_1 (z))y(f_2 (b))/y(f_1 (b))y(f_2 (z))). \end{align*}
The lemma now follows from recalling that $y(f_1 (z))/y(f_2 (z))=a_0x(z)^2 $. 
\end{proof}
The proof of Theorem \ref{biellexample} now follows from Theorem \ref{height} and Lemma \ref{completingproof}.

\section{Computing $X(K_{\mathfrak{p}})_U$ and $X(K)$}\label{howtocomputethings}

In this section, we explain how to use Theorem \ref{biellexample} in practice and describe the computation of  $X(K_{\mathfrak{p}})_U$, where $X$ is a bielliptic genus 2 curve whose Jacobian has rank 2 and $U$ is associated to the cycle $Z$ as in Section \ref{7.3}. Throughout this section, we will use the phrase ``computing $X(K_{\mathfrak{p}})_U$'' to mean ``computing a finite set containing $X(K_{\mathfrak{p}})_U$'' (though see Remark \ref{rk:sequel_equal}).
We give two numerical examples of $X(K_{\mathfrak{p}})_U$ and further discuss how one might effectively extract $X(K)$ from $X(K_{\mathfrak{p}})_U$. We assume in this section that $p$ is a prime of good reduction for $X$ and of ordinary reduction for $J$.\subsection{An alternative formula for $X(K_{\mathfrak{p}})_U $}
We record the following slight variant of Theorem \ref{biellexample}, which turns the computation into one which can be carried out over two affine patches covering $X(K)$.

\begin{Corollary}\label{cor1} Let $X/K$ be a genus 2 bielliptic curve 
$$
y^2 =x^6 +a_4x^4 +a_2x^2 +a_0
$$ over $K = \Q$ or an imaginary quadratic field, and $E_i$ an elliptic curve as above. Define $Q_i \in E_i (\overline{\Q })$ by 
$Q_1 =(0,\sqrt{a_0 }), Q_2 =(0,a_0 )$.

\noindent (i) For all $v\nmid p$, and $i=1,2$,
$$h_{E_i ,v}(f_i (z)+Q_i )+h_{E_i ,v}(f_i (z)-Q_i )-2h_{E_{3-i} ,v} (f_{3-i} (z))$$
takes only finitely many values on $X(K_v)$, and for almost all $v$ is identically zero. \\
(ii)
Suppose $\rk E_1(K) = \rk E_2(K) = 1$, and let $P_i \in E_i (K)$ be points of infinite order. Let $\alpha_i = \frac{h_{E_i }(P_i)}{[K:\Q]\log _{E_i}(P_i)^2}$.
Let $\Omega _i $ denote the finite set of values taken by 
$$
\sum _{v\nmid p}\left( h_{E_i ,v}(f_i (z)+Q_i )+h_{E_i ,v}(f_i (z)-Q_i )-2h_{E_{3-i}, v } (f_{3-i} (z))\right),
$$
for $(z_v )$ in $\prod _{v\nmid p}X(K_v )$.
Then for $i=1,2$, $X(K)$ is contained in the finite set of $z$ in $X(K_{\mathfrak{p}} )$ satisfying
\begin{align}\label{rho1eq}\rho_i (z) &:=2h_{E_{3-i} ,\mathfrak{p}}(f_{3-i} (z )) -h_{E_i ,\mathfrak{p}}(f_i (z )+Q_i )-h_{E_i ,\mathfrak{p}}(f_i (z )-Q_i )\\
&\qquad -2\alpha_{3-i}\log _{E_{3-i}}(f_{3-i} (z))^2+2\alpha_i (\log _{E_i}(f_i (z))^2  +\log _{E_i}(Q_i )^2 )   \in \Omega_i.\nonumber \end{align}
\end{Corollary}
\begin{proof}
This follows from Theorem \ref{biellexample} together with Lemma \ref{doesntmatterwhichformula}.
\end{proof} 
\subsection{Computing all points in $X(K_\mathfrak{p})_U$}
Using Corollary \ref{cor1}, we calculate $X(K_\mathfrak{p})_U$ as the union of points found in the following two computations:
$$X(K_\mathfrak{p})_U = \{z \in X(K_\mathfrak{p})_U : x(z) \notin \mathfrak{p}, \rho_1(z) \in \Omega_1  \} \cup \{z \in X(K_\mathfrak{p})_U : x(z) \in \mathfrak{p},\rho_2(z) \in \Omega_2\}.$$

We explain in Algorithm \ref{alg1} below how to compute each of the following terms:

 \begin{align*}\rho_1(z) =\underbrace{2h_{E_2 ,\mathfrak{p}}(f_2 (z))}_{\textrm{Steps 7d,e,f}}-\underbrace{h_{E_1 ,\mathfrak{p}}(f_1 (z)+(0,\sqrt{a_0}))}_{\textrm{Steps\, 7b,e,f}}-\underbrace{h_{E_1 ,\mathfrak{p}}(f_1 (z)+(0,-\sqrt{a_0}))}_{\textrm{Steps 7c,e,f }})\\
-\underbrace{2\alpha_2}_{\textrm{Step\, 3}} \underbrace{\log _{E_2}(f_2 (z))^2}_{\textrm{Step\, 7g}}+\underbrace{2\alpha_1}_{\textrm{Step\, 3}} (\underbrace{\log _{E_1}(f_1 (z))^2}_{\textrm{Step\, 7g}}  +\underbrace{\log _{E_1}((0,\sqrt{a_0}))^2}_{\textrm{Step\, 3}}  ) \end{align*} as power series over $K_{\mathfrak{p}}$, which allows us to search for the points $z\in  X(K_{\mathfrak{p}})_U$ that are solutions to the equation $\rho_1(z) = \beta$ for $\beta \in \Omega_1$. 

Essentially all of the terms of $\rho_i(z)$ can be computed in terms of single and double Coleman integrals. By a double Coleman integral we mean an iterated Coleman integral of the form $\int_{z_1}^{z_2} \eta_1\eta_2$ where $\eta_i$ are differential 1-forms. We recall an interpretation of the local height $h_\mathfrak{p}$ as a double Coleman integral, which is used in Algorithm \ref{alg1}:
\begin{Lemma} We have that $h_{E_i ,\mathfrak{p}}(z) = \int_{\infty}^z \omega_0\bar{\omega_0}$, where $\bar{\omega_0}$ is the dual to $\omega_0 =\frac{dx}{2y}$ under the cup product pairing on $H_\dR^1(E_i)$.\end{Lemma}
\begin{proof} See \cite[$\S 4 $]{balakrishnan2013p}, where the local height $h_p$ of $z - \infty$ is denoted as $\tau(z)$.\end{proof}

\begin{algorithm}[Computing the set $\{z \in X(K_\mathfrak{p})_U : x(z) \notin \mathfrak{p}, \rho_1(z) \in \Omega_1  \}$] $\;$\\
\label{alg1}\noindent\textbf{Input:} Genus 2 curve $X/K$ defined by an equation $y^2 = x^6 + a_4x^4 + a_2x^2 + a_0$ such that the corresponding $E_1(K),E_2(K)$ each have Mordell-Weil rank 1, a good ordinary prime $p$, finite set of values $\Omega_1$.\\
\textbf{Output:} The following subset of $X(K_{\mathfrak{p}})_U: \{z \in X(K_\mathfrak{p})_U : x(z) \notin \mathfrak{p}, \rho_1(z) \in \Omega_1  \}$. 
\begin{enumerate}
\item\label{step1} Compute points $P_1 \in E_{1}(K)$ and  $P_2 \in E_{2}(K)$ of infinite order.
\item\label{step2} Compute global $p$-adic heights $h_{E_1 }(P_1)$ and $h_{E_2 }(P_2)$, using minimal models for $E_1, E_2$, using the algorithm of Mazur, Stein, and Tate \cite{mazur-stein-tate}.
\item\label{step4} Compute $$\quad \qquad \log_{E_1}((0,\sqrt{a_0}))^2 = \left(\int_{\infty}^{(0,\sqrt{a_0})} \omega_0\right)^2, \quad \alpha_i = \frac{h_{E_i }(P_i)}{[K:\Q](\int_{\infty}^{P_i} \omega_0)^2}, \quad i = 1,2.$$
\item\label{barom} Compute the cup product pairing between elements in $H_\dR^1(E_1)$ and also between elements in $H_\dR^1(E_2)$; use this to compute $\bar{\omega_0}$ for $E_1$ and $\bar{\omega_0}$  for $E_2$ to write  $h_{E_i ,\mathfrak{p}} = \int \omega_0\bar{\omega_0}$.
\item\label{almost all} Enumerate the list of points $\mathcal{D} = X(\F_\mathfrak{p}) \setminus \{(0,\pm \overline{\sqrt{a_0})}\}$. 
\item Initialise an empty set $R$.
\item For each $D \in \mathcal{D}$:\begin{enumerate}
\item Compute $Q$, a lift of $D$, and a local coordinate $(x(t),y(t))$ at $Q$. 
\item\label{P1param} Compute $S_1 := f_1(Q) + (0,\sqrt{a_0})$.  Likewise compute $f_1((x(t),y(t))) + (0,\sqrt{a_0})$, which sends the local coordinate to this residue disk. 
\item\label{P2param} Compute $f_1(Q) - (0,\sqrt{a_0})$.  Likewise compute $f_1((x(t),y(t)))- (0,\sqrt{a_0})$, which gives a local coordinate in the residue disk. 
\item\label{P4param} Compute $f_2(Q)$.  We have $f_2(x(t)) = (x(t))^{-2}$ gives the $x$-coordinate of a local coordinate in the residue disk of $f_2(Q)$.
\item\label{8e} Compute the following local heights at $\mathfrak{p}$ of the points in Steps \ref{P1param} - \ref{P4param}: $h_{E_1, \fp}(f_1(Q) + (0,\sqrt{a_0})), h_{E_1, \fp}(f_1(Q) - (0,\sqrt{a_0})), h_{E_2,\fp}(f_2(Q))$.
\item\label{correction} Using Step \ref{barom}, for each of the points in Steps \ref{P1param} - \ref{P4param},  use the local coordinates computed to calculate a power series expansion  of  $h_{E_i ,\mathfrak{p}}$ in the disk of the respective point, using Step \ref{8e} to set the global constant of integration.  
\item\label{8g} Compute $\log_{E_i}(f_i(Q)(t)) = \log_{E_i}(f_i(Q)) + \int_{f_i(Q)(t)} \omega_0$.
\item Finally, let $\rho_1(t)$ be the appropriately weighted sum of contributions from Steps \ref{step4}, \ref{correction}, and \ref{8g}, as in Equation \ref{rho1eq}.
\item For each $\beta \in \Omega_1$, compute the set of roots of $\rho_1(t) = \beta$. For each root $r$, append $X(x(r),y(r)) \in X(K_{\mathfrak{p}})$ with multiplicity to the set $R$.
\end{enumerate}
\item Output $R$,  the subset $\{z \in X(K_\mathfrak{p})_U : x(z) \notin \mathfrak{p}, \rho_1(z) \in \Omega_1  \} \subset X(K_{\mathfrak{p}})_U$.
\end{enumerate}
\end{algorithm}
\begin{Remark} We clarify Step \ref{correction} above: e.g., for $S_1$, first compute a local coordinate $S_1(t)  = (x_1(t),y_1(t))$ at $S_1$ (if $S_1$ is non-Weierstrass, $x_1(t) = t + x(S_1)$) and use it to compute $ h_{E_1 ,\mathfrak{p}}(S_1(t)) = h_{E_1 ,\mathfrak{p}}(S_1) - 2\left(\int_{S_1}^{S_1(t)}\omega_0\bom_0+ \int_{S_1}^{S_1(t)} \omega_0\int_{\infty}^{S_1}\bom_0\right).$   Then use the parametrisation computed in Step~\ref{P1param} so that this power series in the disk of $S_1$ uses the correct parameter, that induced by the local coordinate at $Q$. Likewise, in Step \ref{8g} one must also be careful about local coordinates: one way is to compute a local coordinate $f_i(Q)(t) = (x_i(t),y_i(t))$ at $f_i(Q)$ to  compute $\int_{f_i(Q)(t)} \omega_0$, then correct the parametrisation so that this power series within the disk of $f_i(Q)$ uses the correct parameter, that induced by the local coordinate at $Q$, as in Step~\ref{correction}.  \end{Remark}

The computation of $\rho_2(z) \in \Omega_2$ is carried out in an analogous manner and only involves the two residue disks of $X(K_\mathfrak{p})$ not considered in Step \ref{almost all} of Algorithm \ref{alg1}. Putting this together gives an algorithm to compute $X(K_\mathfrak{p})_U$.

\begin{Remark} For a discussion of the $p$-adic precision in the computation of Coleman integrals resulting in a provably correct number of terms in the corresponding power series expansions, see \cite[$\S 3.3$]{BBM:Computing}. Applying Strassman's theorem gives an upper bound on the number of roots, which may be found explicitly using \texttt{gp}.\end{Remark}

We now give two examples illustrating the algorithm to compute $X(K_\mathfrak{p})_U$, carried out using Sage \cite{sage}.   
\subsection{Example 1: Rational points on a genus 2 bielliptic curve with rank 2 Jacobian}\label{ex1}
We compute $X(\mathbb{Q})$, where $X$ is the genus 2 curve $$X: y^2 = x^6 -2x^4 - x^2 + 1.$$
Let $E_1 $ and $E_2 $ be the corresponding elliptic curves, which each have Mordell-Weil rank 1 over $\mathbb{Q}$ 
and integral $j$-invariant. On $E_1$, the point $P_1 =(0,1)$ is of infinite order, and on $E_2$, the point $P_2 =(0,1)$ is 
of infinite order.   We fix a branch of the $p$-adic logarithm $\log_p$ and take $\chi $ to be the cyclotomic character, normalised so that $\chi _p (z)=\log _p (z)$ and for $v\neq p$, 
$\chi _v (z)=-v(z)\log _p (v)$.  Note that, with respect to this choice of character, our local height is twice the local height as defined in \cite{silverman1988computing}. Moreover, $E_1$ and $E_2$ each have good ordinary reduction at $p = 3$. We determine a finite set containing $X(\mathbb{Q}_3 )_2 $ and use this to determine $X(\mathbb{Q})$ exactly. 
We are not able to determine whether $X(\mathbb{Q}_3 )_2 =X(\mathbb{Q})$. 

\subsubsection{Local contributions away from $p$}
The curve $X$ has bad reduction at $2$, bad but potential good reduction at $7$, and good reduction at all other primes. Hence to 
determine the set $\Omega $ we need to determine the possible values of 
$$h_{E_1 ,2}(f_1 (z))-h_{E_2 ,2}(f_2 (z))-2\chi _2 (x(z)).$$
First note that $X(\mathbb{Q}_2 )$ has no $\mathbb{Q}_2 $ points whose $x$-coordinate 
has valuation zero (e.g. by checking mod 8). It will turn out that the above functions can (each) only take two possible values, corresponding 
to $v(x)>0$ and $v(x)<0$, where $v$ denotes the $2$-adic valuation.
We compute local heights on $E_1 $.  The equation given above for $E_1 $ is minimal at 2. $E_1 $ has type II reduction, which means that the singular point mod 2 does not lift to a $\mathbb{Q}_2 $ point. Hence 
$
h_{E_1 ,2} (f_1 (z))=2\max \{0,-v_2 (x(z))\}\log _p (2).
$

We compute local heights on $E_2$.  The equation given for $E_2 $ is minimal, and it has type IV reduction. The unique singular point of the special fibre is $(0,1)$. By Silverman \cite{silverman1988computing}, the local height at points $(x_0 ,y_0 )$ of bad reduction is given by
$$
h_{E_2 ,2} ((x_0 ,y_0 ))=-\frac{2}{3}(1+v(y_0 ))\log _p (2).
$$ 
Hence the possible values of $h_{E_2 ,2}(f_2 (z))$ are $2\max \{0,v(x(z))\}\log _p (2)$ when the valuation of $x(f_2(z))$ is positive, and 
$-\frac{2}{3}\log _p (2)$ when the valuation of $x(f_2(z))$ is negative.
Hence $$
h_{E_1 ,2}(f_1 (z))-h_{E_2 ,2}(f_2 (z))-2\chi _2 (x(z))=\left\{
\begin{array}{cc}
 0 & v(x(z))<0 \\
-\frac{2}{3}\log _p (2) & v(x(z))>0. 
\end{array} \right.
$$ 
Finally $h_{E_2 ,2}((0,1))=-\frac{2}{3}\log _p (2)$ and $h_{E_1 ,2}((0,1))=0$.

Hence by Lemma \ref{doesntmatterwhichformula}, $$h_{E_1 ,2} (f_1 (z)+(0,1))+h_{E_1 ,2} (f_1 (z)-(0,1))-2h_{E_2 ,2} (f_2 (z))=\left\{
\begin{array}{cc}
 0 & v(x(z))<0 \\
 \frac{4}{3}\log _p (2) & v(x(z))>0 \\
\end{array} \right.
$$ 
$$
h_{E_2 ,2} (f_2 (z)+(0,1))+h_{E_2 ,2} (f_2 (z)-(0,1))-2h_{E_1 ,2} (f_1 (z))=\left\{
\begin{array}{cc}
 -\frac{4}{3}\log _p (2) & v(x(z))<0 \\
 -\frac{8}{3}\log _p (2) & v(x(z))>0. 
\end{array} \right.
$$ 
We deduce $\Omega _1 =\{ 0,\frac{4}{3}\log _p (2) \}$ and $\Omega _2 =\{-\frac{4}{3}\log _p (3),-\frac{8}{3}\log _p (3) \} $.
\subsubsection{Local contributions at $p = 3$}
By Corollary \ref{cor1}, to determine the $X(\Q _3 )_U$, we need to carry out Algorithm~\ref{alg1} twice: for the residue disks corresponding to $\overline{\infty^{\pm}}$, we find $z$ with $\rho_1(z) \in \Omega_1$, and for the residue disks corresponding to $\overline{(0, \pm 1)}$, we find $z$ with $\rho_2(z) \in \Omega_2$. This gives $X(\Q_3)_U$:
\small
\begin{center}
 \begin{tabular}{||c |  r  |c | l ||}
    \hline
 $X(\F_3)$ & recovered $x(z)$ in residue disk & $z \in X(\Q)$ & $\rho_i(z) = \beta$ \\
  \hline
 $\overline{\infty^{\pm}}$ & $3^{-1} + 1 + 3^{3} + 2 \cdot 3^{4} +  O(3^{6})$ &  & $\rho_1(z) = 0$ \\
 & $2 \cdot 3^{-1} + 1 + 2 \cdot 3 + 2 \cdot 3^{2} + 3^{3} + 2 \cdot 3^{5} +  O(3^{6})$ &  & $\rho_1(z) = 0$ \\
 & $\infty^{\pm}$ & $\infty^{\pm}$ & $\rho_1(z) = \frac{4}{3}\log_3(2)$ \\
  $\overline{(0, \pm 1)}$ & $2 \cdot 3 + 3^{2} + 3^{3} + 3^{4} + 3^{5} +O(3^6)$ & $(\frac{3}{2}, \pm \frac{1}{8})$ & $\rho_2(z) = -\frac{8}{3}\log_3(2)$\\ 
 & $3 + 3^{2} + 3^{3} + 3^{4} + 3^{5} +  O(3^{6})$ & $(-\frac{3}{2}, \pm \frac{1}{8})$ &  $\rho_2(z) = -\frac{8}{3}\log_3(2)$ \\
 & $O(3^6)$ & $(0, \pm 1)$ &  $\rho_2(z) = -\frac{4}{3}\log_3(2)$ \\

 \hline
\end{tabular}
\end{center}
\normalsize
Code illustrating Algorithm~\ref{alg1}, producing this set of points, is available at \cite{qc1code}.

\begin{Theorem} We have $X(\Q) = \left\{(0,\pm 1), \left(\frac{3}{2}, \pm \frac{1}{8}\right), \left(-\frac{3}{2}, \pm \frac{1}{8}\right), \infty^{\pm}\right\}.$\end{Theorem}
\begin{proof}We wish to compute $X(\Q)$ from $X(\Q_3)_U$. To do this, we must do two things: prove that the points in $X(\Q_3)_U$ which do not appear to be rational actually are not rational and check the multiplicities of all recovered points, to rule out the possibility that the table collapses multiple points that are just 3-adically close to the points in the table to the indicated precision. We start with the second task.  Our computation shows that the solution $x(z) = O(3^6)$ occurs as a root of $\rho(z) = -\frac{4}{3}\log_3(2)$ with multiplicity two, which gives the known global points $(0,\pm 1)$ and  two points $3$-adically close to $(0,\pm 1)$. Likewise, solving $\rho(z) = \frac{4}{3}\log_3(2)$ yields $\infty^{\pm}$ on $X$ and  two points 3-adically close to $\infty^{\pm}$. The other points in the table, however, occur as roots with multiplicity 1. Note that $\rho(z)$ is an even function, so by considering the local expansion of $\rho$ at each of the global points $(0,1), (0,-1), \infty^+, \infty^-$, we see that its power series expansion must have a global double root at each of these points.

Now we show that the ``extra'' $\Q_3$ points recovered in the disks of $\infty^{\pm}$ cannot be rational, for the following formal group consideration. Consider $z \in X(\Q_3)$ with $v_3(x(z)) = -1$. Then the corresponding point $f_1(z)$ on $E_1$ has $v_3(x(f_1(z))) = -2$. However, note that $E_1(\F_3)$ has order 3 and $E_1(\Q)$ is generated by $P$, where $P = (0,1)$. Thus the smallest multiple of $P$ in the formal group is $3P = (-\frac{8}{81}, -\frac{757}{729})$, which implies that the $v_3(x(Q)) \leq -4 $ for any $Q \in \langle 3P \rangle$. So $f_1(z)$ cannot be rational and thus $z\not\in X(\Q)$.  Thus we conclude
$X(\Q) = \left\{(0,\pm 1), \left(\frac{3}{2}, \pm \frac{1}{8}\right), \left(-\frac{3}{2}, \pm \frac{1}{8}\right), \infty^{\pm}\right\}.$\end{proof}

\subsection{Example 2: $X_0(37)(\Q(i))$}\label{ex2}
Over $\mathbb{Q}$, the modular curve $X_0(37)$ has the model $y^2 = -x^6-9x^4-11x^2+37$. Recall that $X_0(37)$ has good reduction away from $37$.  For convenience, we make the change of variables $(x,y) \mapsto (ix,y)$ so that we take as our working model $$X:y^2 =x^6 -9x^4 +11x^2 +37.$$  Let $J$ denote the Jacobian of $X$. We have $\rk J(\Q) = \rk J_0(37)(K) = 2$.  We thank Daniels and Lozano-Robledo \cite{corrx037} for bringing this example to our attention.

In this  subsection, we construct finite sets of $\mathfrak{p}$-adic points containing $X(K_{\mathfrak{p}})_2$ for various primes $\mathfrak{p}$. Using the Mordell-Weil sieve, as carried out by J. Steffen M\"{u}ller (described in Appendix \ref{steffenappendix}), this is then used to determine $X(K)$.  We work with the following models of $E_1$ and $E_2$: \begin{equation*}
E_1: y^2 =x^3 -16x+16 \qquad \qquad E_2: y^2 =x^3 -x^2 -373x+2813, \end{equation*}
with maps $f_i $ from $X$ to $E_i $ that are given by sending $(x,y)$ to $(x^2 -3,y)$ and $(37x^{-2}+4,37x^{-3})$, respectively.

We have $\rk E_1(K) = \rk E_2(K) = 1$ and we take $P_1 = (0,4) \in E_1 (K)$ and $P_2 = (4,37) \in E_2 (K)$ as our points of infinite order.  We use primes $p$ which are good, ordinary, and, so that we work over $\Q_p$ and not a quadratic extension, split in $K$ and $\Q(\sqrt{37})$: we take $p= 41, 73,$ and $101$. For each of these primes $p$, we choose a prime $\mathfrak{p}$ lying above it in $\mathcal{O}_K$, and take $\chi$ to be a non-trivial idele class character of $K$ which is trivial on $\mathcal{O}_{\overline{\mathfrak{p}}}^\times $. We normalise $\chi$ so that $\chi _{37}(37)=-\log _p (37)$.

\subsubsection{Local calculations at $37$}

In this subsection we prove that for all $b,z\in X(\mathbb{Q}_{37})$ with $x(z)$ and $x(b)$ not equal to infinity,
\begin{align*}&h_{E_1 ,37}(f_1 (z))-h_{E_1 ,37}(f_1 (b))-h_{E_2 ,37}(f_2 (z))\\
& +h_{E_2 ,37}(f_2 (b))+2\chi _{37}(x(z))-2\chi _{37}(x(b))=0.\end{align*}
Recall that by Lemma \ref{completingproof}, this is equivalent to the statement that the inertia subgroup of $G_{\mathbb{Q}_{37}}$ acts trivially on $A(b,z)$.
In \cite{dogra:thesis} this is proved directly. As that proof involves other tools we do not want to introduce, we shall prove this by determining the local heights explicitly.
\begin{Lemma}For all $z$ in $X(\mathbb{Q}_{37})$, we have\\
(i) $h_{E_1 ,37}(f_1 (z))=2\chi _{37}(x(z))$. \\
(ii) $h_{E_2 ,37}(f_2 (z))=\frac{2}{3}\chi _{37} (37).$
\end{Lemma}
\begin{proof}
Note that there are no $\mathbb{Q}_{37}$-points of $X$ for which $x(z)$ has positive 37-adic valuation.
The Weierstrass equations given for $E_1$ and $E_2$ are both minimal at 37. The Weierstrass equation for $E_1 $ is also regular, hence all $\mathbb{Q}_{37}$-points 
are points of good reduction. This establishes part (i). 
The elliptic curve $E_2 $ has split multiplicative reduction of type I3. The singular point of $E_{2}(\mathbb{F}_{37})$ is $(4,0)$, and all points of $E_{2,\mathbb{Q}_{37}}$ in the image of $X(\mathbb{Q}_{37})$ reduce to this point. By Silverman's algorithm \cite[Theorem 5.2]{silverman1988computing}, we deduce that for all $z$ in $X(\mathbb{Q}_{37})$, we have
$h_{E_2 ,37}(f_2 (z))=\frac{2}{3}\chi _{37}(37).$
This completes
the proof of part (ii).
\end{proof}

By Lemmas \ref{doesntmatterwhichformula} and \ref{completingproof}, this gives
$\Omega _1 = \{\frac{4}{3}\log _p (37) \}$ and $\Omega _2 =\{-\frac{2}{3}\log _p (37) \}$.

Hence $X(K_{\mathfrak{p}})_U$ may be computed by determining the solutions to
\begin{align*}
\rho _i (z) &= 2h_{E_{3-i} ,\mathfrak{p}}(f_{3-i} (z))-h_{E_i ,\mathfrak{p}}(f_i (z)+Q_i )-h_{E_i ,\mathfrak{p}}(f_i (z)-Q_i ) \\
&\qquad -2\alpha _{3-i} h_{E_{3-i} }(f_{3-i} (z))+2\alpha _i (h_{E_i }(f_i (z))+\log_{E_i}(Q_i )^2)\in \Omega _i,
\end{align*}
where $Q_1 =(-3,\sqrt{37})$ and $Q_2 =(4,37)$.

We computed finite sets containing $X(\Q_{41})_U, X(\Q_{73})_U,$ and $X(\Q_{101})_U$ using the methods of the paper, using a mild adaptation of the code in \cite{qc1code}. In the tables below, for each disk corresponding to the four choices $\overline{(\pm x,\pm y)}$ we give details for the disk corresponding to $\overline{(x,y)}$ with $x,y < \frac{p}{2}$.  We fix an identification $X(K_\mathfrak{p}) \simeq X(\Q_p)$. 

Here is data for $X(\Q_{41})_U$: 
\small
\begin{center}
\begin{tabular}{||c |  r  |c  ||}
    \hline
$X(\F_{41})$  & recovered $x(z)$ in residue disk  & $z \in X(K)$   \\
  \hline
 $\overline{(1,9)}$ & $1 + 16 \cdot 41 + 23 \cdot 41^{2} + 5 \cdot 41^{3} + 23 \cdot 41^{4} + O(41^5)$  &  \\ 
  & $1 + 6 \cdot 41 + 23 \cdot 41^{2} + 30 \cdot 41^{3} + 14 \cdot 41^{4} + O(41^5)$ &  \\
  $\overline{(2,1)}$ & $2 + O(41^5)$ & $(2,1)$ \\ 
  & $2 + 19 \cdot 41 + 36 \cdot 41^{2} + 15 \cdot 41^{3} + 26 \cdot 41^{4} +O(41^5)$ &  \\
  $\overline{(4,18)}$ &  &   \\ 
  $\overline{(5,12)}$ & $5 + 25 \cdot 41 + 26 \cdot 41^{2} + 26 \cdot 41^{3} + 31 \cdot 41^{4} +O(41^5)$ &  \\ 
  & $5 + 14 \cdot 41 + 12 \cdot 41^{3} + 33 \cdot 41^{4} +O(41^5)$&   \\
  $\overline{(6,1)}$ &  $6 + 18 \cdot 41^{2} + 31 \cdot 41^{3} + 6 \cdot 41^{4} +O(41^5)$&    \\ 
  & $6 + 30 \cdot 41 + 35 \cdot 41^{2} + 11 \cdot 41^{3} + O(41^5)$ &  \\
  $\overline{(7,15)}$ &  &     \\ 
  $\overline{(9,4)}$ & $9 + 9 \cdot 41 + 34 \cdot 41^{2} + 22 \cdot 41^{3} + 24 \cdot 41^{4} +O(41^5)$ & $(i,4)$   \\ 
  & $9 + 39 \cdot 41 + 14 \cdot 41^{2} + 6 \cdot 41^{3} + 17 \cdot 41^{4} +O(41^5)$ &  \\
  $\overline{(12,5)}$ &  &   \\ 
   $\overline{(13,19)}$ & $13 + 10 \cdot 41 + 2 \cdot 41^{2} + 15 \cdot 41^{3} + 29 \cdot 41^{4} +O(41^5)$ &  \\ 
  & $13 + 7 \cdot 41 + 8 \cdot 41^{2} + 32 \cdot 41^{3} + 14 \cdot 41^{4} + O(41^5)$ & \\
   $\overline{(16,1)}$ &$16 + 13 \cdot 41 + 6 \cdot 41^{3} + 18 \cdot 41^{4} +O(41^5)$  &  \\ 
  & $16 + 12 \cdot 41 + 8 \cdot 41^{2} + 9 \cdot 41^{3} + 32 \cdot 41^{4} +O(41^5)$ &   \\
   $\overline{(17,20)}$ & $17 + 24 \cdot 41 + 37 \cdot 41^{2} + 16 \cdot 41^{3} + 28 \cdot 41^{4} +O(41^5)$  &  \\ 
  & $17 + 19 \cdot 41 + 20 \cdot 41^{2} + 7 \cdot 41^{3} + 7 \cdot 41^{4} +O(41^5)$ &  \\
   $\overline{(18,20)}$ &  $18 + 3 \cdot 41 + 7 \cdot 41^{2} + 9 \cdot 41^{3} + 38 \cdot 41^{4} +O(41^5)$ &   \\ 
  & $18 + 41 + 34 \cdot 41^{2} + 3 \cdot 41^{3} + 32 \cdot 41^{4} +O(41^5)$ &    \\
   $\overline{(19,3)}$ &  &   \\ 
  $\overline{(20,6)}$ & $20 + 7 \cdot 41 + 40 \cdot 41^{2} + 22 \cdot 41^{3} + 7 \cdot 41^{4} +O(41^5)$ &   \\ 
  & $20 + 23 \cdot 41 + 26 \cdot 41^{2} + 17 \cdot 41^{3} + 22 \cdot 41^{4} +O(41^5)$ &  \\
 $\overline{\infty^{+}}$ & $\infty^+$  & $\infty^+$   \\
   $\overline{(0,18)}$ &  $32 \cdot 41 + 13 \cdot 41^{2} + 16 \cdot 41^{3} + 8 \cdot 41^{4} +O(41^5)$& \\ 
  & $9 \cdot 41 + 27 \cdot 41^{2} + 24 \cdot 41^{3} + 32 \cdot 41^{4} +O(41^5)$&   \\
 \hline
\end{tabular}
\end{center}

Here we compute $X(\Q_{73})_U$:
\begin{center}
 \begin{tabular}{||c |  r  |c  ||}
    \hline
$X(\F_{73})$  & recovered $x(z)$ in residue disk & $z \in X(K)$ (or $X(\Q(\sqrt{3}))$) \\
  \hline
 $\overline{(2,1)}$ &$2 + 61 \cdot 73 + 50 \cdot 73^{2} + 71 \cdot 73^{3} + 56 \cdot 73^{4} + O(73^{5})$ &  \\ 
  & $2 + O(73^{5}) $& $(2,1)$  \\
  $\overline{(5,26)}$ & $5 + 63 \cdot 73 + 4 \cdot 73^{2} + 42 \cdot 73^{3} + 25 \cdot 73^{4} + O(73^{5})$  &   \\ 
  & $5 + 39 \cdot 73 + 65 \cdot 73^{2} + 33 \cdot 73^{3} + 60 \cdot 73^{4} + O(73^{5})$ &   \\
  $\overline{(7,16)}$ & $7 + 62 \cdot 73 + 31 \cdot 73^{2} + 33 \cdot 73^{3} + 44 \cdot 73^{4} + O(73^{5})$  &   \\ 
  & $7 + 29 \cdot 73 + 67 \cdot 73^{2} + 69 \cdot 73^{3} + 17 \cdot 73^{4} + O(73^{5})$ &   \\
  $\overline{(9,34)}$ &  &    \\ 
  $\overline{(10,30)}$ & $10 + 53 \cdot 73 + 35 \cdot 73^{2} + 21 \cdot 73^{3} + 67 \cdot 73^{4} + O(73^{5})$ &    \\ 
  & $10 + 39 \cdot 73 + 40 \cdot 73^{2} + 17 \cdot 73^{3} + 59 \cdot 73^{4} + O(73^{5})$ &  \\
  $\overline{(18,17)}$ &  &    \\ 
  $\overline{(19,2)}$ &  &    \\ 
  $\overline{(20,15)}$ &  &    \\ 
  $\overline{(21,4)}$ & $21 + 17 \cdot 73 + 70 \cdot 73^{2} + 42 \cdot 73^{3} + 18 \cdot 73^{4} + O(73^{5})$ &    \\ 
  & $21 + 52 \cdot 73 + 67 \cdot 73^{2} + 20 \cdot 73^{3} + 27 \cdot 73^{4} + O(73^{5})$ & $(\sqrt{3},4)$   \\
  $\overline{(23,31)}$ & $23 + 18 \cdot 73 + 59 \cdot 73^{2} + 23 \cdot 73^{3} + 2 \cdot 73^{4} + O(73^{5})$  &  \\ 
  & $23 + 70 \cdot 73 + 53 \cdot 73^{2} + 21 \cdot 73^{3} + 50 \cdot 73^{4} + O(73^{5})$ &  \\
   $\overline{(25,25)}$ &  &   \\ 
         $\overline{(27,4)}$ & $27 + 62 \cdot 73 + 28 \cdot 73^{2} + 56 \cdot 73^{3} + 58 \cdot 73^{4}+ O(73^{5})$ & $(i,4)$  \\ 
            & $27 + 24 \cdot 73 + 30 \cdot 73^{2} + 20 \cdot 73^{3} + 65 \cdot 73^{4} + O(73^{5})$ &  \\

         \hline
\end{tabular}
\end{center}

\begin{center}
 \begin{tabular}{||c |  r  |c  ||}
  \hline 
$X(\F_{73})$  & recovered $x(z)$ in residue disk  & $z \in X(K)$ \\
  \hline 
    $\overline{(29,8)}$ & $29 + 70 \cdot 73 + 21 \cdot 73^{2} + 56 \cdot 73^{3} + 5 \cdot 73^{4} + O(73^{5})$  &  \\ 
  & $29 + 34 \cdot 73 + 42 \cdot 73^{2} + 19 \cdot 73^{3} + 54 \cdot 73^{4}+ O(73^{5})$ & \\
   $\overline{(30,20)}$ &  &   \\ 
    $\overline{(36,17)}$ & $36 + 70 \cdot 73 + 19 \cdot 73^{2} + 11 \cdot 73^{3} + 54 \cdot 73^{4} + O(73^{5})$ &    \\ 
  & $36 + 32 \cdot 73 + 23 \cdot 73^{2} + 23 \cdot 73^{3} + 28 \cdot 73^{4}+ O(73^{5})$&  \\
 $\overline{\infty^{+}}$ &$\infty^{+}$  & $\infty^{+}$    \\
    $\overline{(0,16)}$ & $61 \cdot 73 + 63 \cdot 73^{2} + 51 \cdot 73^{3} + 16 \cdot 73^{4} + O(73^{5})$ &   \\ 
  & $12 \cdot 73 + 9 \cdot 73^{2} + 21 \cdot 73^{3} + 56 \cdot 73^{4} + O(73^{5})$ &   \\
 \hline
\end{tabular}
\end{center}

Here we compute $X(\Q_{101})_U$:
\small
\begin{center}
 \begin{tabular}{||c |  r  |c  ||}
    \hline
$X(\F_{101})$  & recovered $x(z)$ in residue disk  & $z \in X(K)$  \\
  \hline
 $\overline{(2,1)}$ & $2 + O(101^{7})$ & $(2,1)$  \\ 
  & $2 + 38 \cdot 101 + 11 \cdot 101^{2} + 99 \cdot 101^{3} + 26 \cdot 101^{4} + O(101^{5})$ &  \\
  $\overline{(8,36)}$ & $8 + 90 \cdot 101 + 39 \cdot 101^{2} + 80 \cdot 101^{3} + 70 \cdot 101^{4} + O(101^{5})$ &   \\ 
  & $8 + 40 \cdot 101 + 84 \cdot 101^{2} + 74 \cdot 101^{3} + 15 \cdot 101^{4} + O(101^{5})$ &    \\
  $\overline{(10,4)}$ & $10 + 5 \cdot 101 + 29 \cdot 101^{2} + 66 \cdot 101^{3} + 10 \cdot 101^{4} + O(101^{5})$  & $(i,4)$   \\ 
  & $10 + 49 \cdot 101 + 80 \cdot 101^{2} + 74 \cdot 101^{3} + 8 \cdot 101^{4} + O(101^{5})$ & \\
  $\overline{(12,7)}$ & $12 + 12 \cdot 101 + 95 \cdot 101^{2} + 55 \cdot 101^{3} + 48 \cdot 101^{4} + O(101^{5})$  &    \\ 
  & $12 + 36 \cdot 101 + 62 \cdot 101^{2} + 97 \cdot 101^{3} + 27 \cdot 101^{4} + O(101^{5})$ &   \\
  $\overline{(14,21)}$ &  $14 + 62 \cdot 101 + 62 \cdot 101^{2} + 41 \cdot 101^{3} + 51 \cdot 101^{4} + O(101^{5})$ &    \\ 
  & $14 + 80 \cdot 101 + 72 \cdot 101^{2} + 32 \cdot 101^{3} + 75 \cdot 101^{4} + O(101^{5})$ &  \\
  $\overline{(15,11)}$ &  &   \\ 
  $\overline{(17,18)}$ & $17 + 65 \cdot 101 + 37 \cdot 101^{2} + 80 \cdot 101^{3} + 45 \cdot 101^{4} + O(101^{5})$  &    \\ 
  & $17 + 50 \cdot 101 + 61 \cdot 101^{2} + 89 \cdot 101^{3} + 61 \cdot 101^{4} + O(101^{5})$ & \\
   $\overline{(18,45)}$ &  &    \\ 
  $\overline{(20,47)}$ &  &    \\ 
  $\overline{(22,3)}$ & $22 + 59 \cdot 101 + 78 \cdot 101^{2} + 43 \cdot 101^{3} + 53 \cdot 101^{4} + O(101^{5})$ &   \\ 
    & $22 + 96 \cdot 101 + 29 \cdot 101^{2} + 43 \cdot 101^{3} + 86 \cdot 101^{4} + O(101^{5})$ &  \\
   $\overline{(24,19)}$ &  &   \\ 
      $\overline{(27,39)}$ &  &  \\ 
  $\overline{(28,37)}$ & $28 + 30 \cdot 101 + 83 \cdot 101^{2} + 5 \cdot 101^{3} + 23 \cdot 101^{4} + O(101^{5})$ &   \\ 
      &$28 + 37 \cdot 101 + 24 \cdot 101^{2} + 78 \cdot 101^{3} + 35 \cdot 101^{4} + O(101^{5})$ & \\
  $\overline{(30,46)}$ &  &   \\ 
  $\overline{(31,23)}$ & $31 + 23 \cdot 101 + 11 \cdot 101^{2} + 67 \cdot 101^{3} + 39 \cdot 101^{4} + O(101^{5})$  &   \\ 
   & $31 + 29 \cdot 101 + 68 \cdot 101^{2} + 29 \cdot 101^{3} + 24 \cdot 101^{4} + O(101^{5})$ &  \\
  $\overline{(34,45)}$ & $34 + 91 \cdot 101 + 46 \cdot 101^{2} + 28 \cdot 101^{3} + 34 \cdot 101^{4} + O(101^{5})$ &   \\ 
  &$34 + 51 \cdot 101 + 73 \cdot 101^{2} + 34 \cdot 101^{3} + 14 \cdot 101^{4} + O(101^{5})$ &  \\
   $\overline{(37,22)}$ &  &   \\ 
   $\overline{(38,28)}$ &  &   \\ 
   $\overline{(39,46)}$ & $39 + 76 \cdot 101 + 86 \cdot 101^{2} + 18 \cdot 101^{3} + 64 \cdot 101^{4} + O(101^{5})$  &   \\ 
  & $39 + 31 \cdot 101 + 43 \cdot 101^{2} + 10 \cdot 101^{3} + 48 \cdot 101^{4} + O(101^{5})$ & \\
       $\overline{(46,6)}$ &  &   \\ 
     $\overline{(47,32)}$ &  &  \\ 
    $\overline{(48,27)}$ & $48 + 43 \cdot 101 + 100 \cdot 101^{2} + 47 \cdot 101^{3} + 19 \cdot 101^{4}+ O(101^{5})$ &   \\ 
  & $48 + 21 \cdot 101 + 38 \cdot 101^{2} + 80 \cdot 101^{3} + 95 \cdot 101^{4} + O(101^{5})$ &  \\
  $\overline{(50,5)}$ & $50 + 59 \cdot 101 + 19 \cdot 101^{2} + 64 \cdot 101^{3} + 36 \cdot 101^{4} + O(101^{5})$  &    \\
 & $50 + 74 \cdot 101 + 69 \cdot 101^{2} + 80 \cdot 101^{3} + 21 \cdot 101^{4} + O(101^{5})$ &   \\
   $\overline{\infty^{+}}$ & $\infty^+$ & $\infty^+$   \\
    $\overline{(0,21)}$ &  &   \\
    \hline
\end{tabular}
\end{center}
\normalsize

Using a slightly modified Mordell-Weil sieve (see Appendix \ref{steffenappendix}) on the sets $X(\Q_{41})_U, X(\Q_{73})_U,$ and $X(\Q_{101})_U$, one may determine the $K$-rational points exactly.

\begin{Theorem}\label{x037pts}We have $X_0(37)(\Q(i))= \{(\pm 2, \pm 1)
,(\pm i,  \pm 4), \infty^{\pm} \}$. \end{Theorem} 

\begin{Remark}We note that the computation of $X(\Q_{73})_U$ recovered the points $(\pm \sqrt{-3}, \pm 4) \in X_0(37)(\Q(\sqrt{-3}))$ as well.\end{Remark}

\subsection*{Acknowledgements}It is a pleasure to thank Minhyong Kim for countless enlightening conversations, Ben Moonen and Michael Stoll for helpful suggestions, Harris Daniels and \'Alvaro Lozano-Robledo for suggesting that we try the example in $\S\ref{ex2}$, and Steffen M\"{u}ller for carrying out the Mordell-Weil sieve computation, as described in the appendix. We also thank Steffen M\"{u}ller and the anonymous referees for their numerous valuable comments on earlier versions of this manuscript. Part of this paper builds on material in the thesis of the second author; he is very grateful to his examiners Guido Kings and Victor Flynn for several suggestions which have improved the present work. JSB was supported by NSF grant DMS-1702196 and the Clare Boothe Luce Professorship (Henry Luce Foundation). ND was supported by the EPSRC and by NWO/DIAMANT grant number 613.009.031.
\appendix

\section{An elementary approach to Theorem \ref{biellexample}}\label{naive}
 We give a proof of Theorem \ref{biellexample}(i) in the spirit of \cite{balakrishnan2013p}. Namely, we give an elementary proof of the following proposition:
 \begin{Proposition}\label{thm:neat}
For almost all primes $l$,
   \begin{equation}\label{verykey}
  2\lambda _l (f_1 (z))-\lambda _l (f_2 (z)-(0,1))-\lambda _l (f_2 (z)+(0,1))
 \end{equation}
is zero, and for all $l\neq p$ it can only take finitely many values.
 \end{Proposition}
This result will follow straightforwardly from standard facts about local heights (see \cite[\S VI]{silverman:aec2}).
\begin{Theorem}
Let $y^2 +a_1 xy +a_3 y=x^3 +a_2 x^2 +a_4 x +a_6$ be a Weierstrass equation with $\mathbb{Z}_l $-coefficients.\\
(i): At all points $z$ of good reduction with respect to this model, 
\begin{equation}\nonumber
 h(z)=\frac{1}{2}\min \{v_l (x(z)),0\} \log _p (l).
\end{equation}
(ii): For any $z$ not in $E[2]$, 
\begin{equation}\nonumber
 h([2]z)=4h(z)+v_l (2y(z)+a_1 x(z)+a_3 ) \log _p (l).
\end{equation}
\end{Theorem}
\begin{proof}[Proof of Proposition \ref{thm:neat}]
First suppose that $l$ is an odd prime of good reduction for the models $y^2 =x^3 +bx^2 +ax+1$ and $y^2 =x^3 +ax^2 +bx+1$. Let $z$ be a rational point of $X$, and let
$w=(x_0 ,y_0 ):=f_2 (z)$ be its image in $E_2$. We want to estimate the height of the points $w_1 =z-(0,1)$ and $w_2 =z-(0,-1)$. 
 Let $w_1 =(x_1 ,y_1 )$ and $w_2 =(x_2 ,y_2 )$. Then by an application of the addition formula, we see $x_1 =\frac{b x_0 +2 -2y_0 }{x_0 ^2 }$ and $x_2 =\frac{b x_0 +2 +2y_0 }{x_0 ^2 }$. 
 If $x_0 \in \mathbb{Z}_l ^\times $, then $\min (v_l (x_1 ),0)=0$. If $x_0$ has negative valuation, then it is easy to see $v_l (x_1 )$ and $v_l (x_2 )$ are positive. Finally, if $l|x_0 $, 
 then either $y_0 \equiv 1$ or $y_0 \equiv -1 \mod l$. In the first case, $v_l (x_2 )=-2v_l (x_0 )$. 
 For $x_1 $, note that
 \begin{align*}
y_0  &\equiv 1 +\frac{1}{2}(bx_0 +ax_0 ^2 )-\frac{1}{8}(bx_0 +ax_0 ^2 )^2 \mod x_0 ^3\\
&\equiv 1+\frac{1}{2}bx_0 -\frac{1}{8}(b^2 -4a)x_0 ^2,
\end{align*}
and hence $v_l (x_2 )=0$, since $b^2 -4a$ is in $\mathbb{Z}_l ^\times $. The second case is identical, swapping $x_1 $ and $x_2 $.
 
 Hence in all cases, we have
 \begin{equation}\nonumber
2\max \{ v_l (x_0 ),0 \} +\min \{ v_l (x_1 ),0 \} +\min \{ v_l (x_2 )\}=0.
 \end{equation}

It follows that at all primes $l$ at which $y^2 =x^3 +ax^2 +bx+1 $ defines a smooth $\mathbb{Z}_l $-model, (\ref{verykey}) is identically zero, since 
\begin{equation}\nonumber
\max \{ v_l (x_0 ),0\}=-\min \{v_l (x(f_1 (z),0\}.
\end{equation}
Now suppose that the Weierstrass equations  $y^2 =x^3 +bx^2 +ax+1$ and $y^2 =x^3 +ax^2 +bx+1$
do not define smooth $\mathbb{Z}_l$-models. After passing to a finite extension $K|\mathbb{Q}_l $ we find a new Weierstrass equation over 
$\mathcal{O}_K $ with only finitely many points of bad reduction, such that the origin is a point of good reduction. Hence using part (ii) of the theorem stated above, 
$\lambda _l (z)-\frac{1}{2}v_l  (x_F (z))\log _p (l)$ can only take finitely many values, where $x_F (z)$ denotes the $x$-coordinate of $z$ with respect to the new Weierstrass equation $F$. However, 
whenever elliptic curves $E_1 $ and $E_2 $ defined by Weierstrass equations $F_1 (x_1 ,y_1 )=0$ and $F_2 (x_2 ,y_2 )=0$ are isomorphic, there is an isomorphism of the form
\begin{align*}
x_2 &=\alpha x_1 +\beta,\\
y_2 &=\gamma y_1 +\delta x_1 +\epsilon
\end{align*}
(see \cite{liu2002algebraic}, Corollary 7.4.33), hence $\lambda _v (z)-\frac{1}{2}v_l (x_F (z))\log _p (l)$ can only take finitely many values.

\end{proof}

\section{Applying the Mordell-Weil sieve, by J. Steffen M\"{u}ller}\label{steffenappendix}
\subsection*{The Mordell-Weil sieve}
Let $K$ be a number field with ring of integers $\cO_K$ and let $X/K$ be a smooth projective 
curve of genus $g\ge 2$ with Jacobian $J/K$ of rank $r = \rk(J/K)$. Fix an embedding
$\iota:X \hookrightarrow J$ defined over $K$.
The Mordell-Weil sieve is a technique for obtaining information about $K$-rational points on
$X$ by combining information about the image of $X(k_v)$ inside $J(k_v)$ under $\iota$ for several
primes $v$ of $\cO_K$, where $k_v$ is the residue field at $v$. It was introduced by
Scharaschkin~\cite{Scharaschkin:Thesis}; further information on the case $K=\Q$ can be
found, for instance, in~\cite{Bruin-Stoll:MWSieve} and~\cite{PSS:Twists}.
Siksek~\cite{Siksek:ECNF} describes a variant of the Mordell-Weil sieve over number fields which is adapted to work well with
his explicit Chabauty method over number fields introduced in loc. cit., see
also~\S \ref{samir}.

The general idea of the Mordell-Weil sieve is as follows: Suppose for simplicity that there are no nontrivial $K$-torsion
points on $J$ (see~\cite[Remark~6.1]{BBM:Computing} on how to remove this assumption).
Also suppose that we know generators $P_1,\ldots,P_r$ of $J(K)$. Let $M>1$ be an integer and
let $C_M \subset J(K)/MJ(K)$ be a set of residue classes $c$ for which we want to show that the image
of $X(K)$ under $\iota$ does not map to $c$ under the canonical epimorphism $\pi:J(K)\to
J(K)/MJ(K)$.
Let $S$ be a finite set of primes of $\cO_K$ such that $X$ has good reduction at these primes and consider
the commutative diagram
\[
    \xymatrix{
    X(K)\ar@{->}[r]^-{\pi\circ\iota}  \ar@{->}[d]&
    J(K)/MJ(K)\ar@{->}[d]^{\alpha_S}\\
    \prod_{v\in S}X(k_v)\ar@{->}[r]_-{\beta_S}&\prod_{v\in S} J(k_v)/MJ(k_v)\,.\\
} \]
Here $\alpha_S = (\alpha_v)_{v \in S}$ and $\beta_S = (\beta_v)_{v \in S}$, where
$\alpha_v$ is induced by reduction $J(K)\to J(k_v)$ and $\beta_v =\pi_v\circ \iota_v$ is
the composition of the canonical epimorphism $\pi_v:J(k_v)\to J(k_v)/MJ(k_v)$ and the
embedding $\iota_v \colon X(k_v) \hookrightarrow J(k_v)$.
To prove that $\pi(\iota(X(K))) \cap C_M = \emptyset$ it suffices to show that
\[
    \alpha_S(C_M) \cap \im(\beta_S)=\emptyset\,.
\]
One can also include information at bad primes and ``deep'' information,
see~\cite{Bruin-Stoll:MWSieve}.

Now suppose that $P_1,\ldots,P_r\in J(K)$ only generate a subgroup $G$ of
$J(K)$ of finite index.
It is often difficult to deduce generators of $J(K)$ from $G$;
in fact, it is not known how this can be done in practice when $r>0$ and $g>3$.
Instead one typically proceeds by first saturating $G$ at small primes and then
pretending that $G=J(K)$. The final step is to show that the
orders $\#J(k_v)$ are coprime to the index $(J(K):G)$ for all $v \in
S$, which implies that $G$ and $J(K)$ have the same image in $J(k_v)$ for all $v \in S$.

Sometimes, however, it is advantageous to work directly with a subgroup $G$, which is
known to be {\em not} saturated. In this case, one can use the following strategy, suggested by Besser. 
Suppose that $v \in S$ is a prime such that 
$D := \gcd\left(\#J(k_v),(J(K):G)\right) >1$. Let $q_1,\ldots,q_s$ be
the primes dividing $D$.
For $i \in \{1,\ldots,s\}$ we let $\ell_i = v_{q_i}(\#J(k_v))$ and set $n = \prod^s_{i=1}q_i^{\ell_i}$.
Then the reduction of $nJ(K):=\{nP\,:\,P \in J(K)\}$ is contained in the reduction of $G$ modulo $v$, so
the multiple $n \iota_v(P)$ is contained in the reduction of $G$ at $v$ for every $P \in X(k_v)$.
Therefore, instead of checking whether $\beta_v(P) \in \alpha_v(C_M)$, we check whether
$n\beta_v(P) \in \alpha_v(nC_M)$, where $nC_M = \{nc\,:\, c \in C_M\}$.

\subsection*{Quadratic Chabauty and the Mordell-Weil sieve}\label{sec:qcmws}
The $p$-adic techniques described in the main part of the present text give
congruence conditions for rational points on
$X$.
More precisely, they can be used to compute, for good ordinary primes $\mathfrak{p}$ of
$\cO_K$, a finite subset $X(K_{\mathfrak{p}})_U \subset X(K_{\mathfrak{p}})$ (to finite precision) which contains
$X(K)$.
After identifying the rational points among $X(K_{\mathfrak{p}})_U$, 
one is left with the task of showing that the remaining elements do not
correspond to rational points.

It is discussed in~\cite{BBM:Computing} how to use the Mordell-Weil sieve for this
purpose:
Suppose for now that $J(K)_{\mathrm{tors}}$ is trivial and that $P_1,\ldots,P_r$ generate $J(K)$.
Using linearity of single Coleman integrals, we can compute, for every point $z \in X(K_{\mathfrak{p}})_U$, a tuple
$(\tilde{a}_1,\ldots,\tilde{a}_r) \in \left(\Z/p^N\Z\right)^r$ so that if $\iota(z) =
a_1P_1+\ldots+a_rP_r$ for integers $a_1,\ldots,a_r$, then $a_i \equiv \tilde{a}_i \pmod{p^N}$ for all $i \in
\{1,\ldots,r\}$.
We can apply quadratic Chabauty for several primes $p_1,\ldots,p_s$ to $N_1,\ldots,N_s$
respective digits of precision, and set $M =m \cdot p_1^{N_1}\cdots p_s^{N_s}$, where $m$ is an auxiliary integer.
Discarding rational points and using the Chinese Remainder Theorem, we find tuples $(\tilde{a}_1,\ldots,\tilde{a}_r) \in \left(\Z/M\Z\right)^r$ with the following property: If the set $C_M$ of residue classes in $J(K)/MJ(K)$ corresponding to these tuples does not contain the image of a $K$-rational point on $X$, then the known $K$-rational points are the only ones on $X$. The Mordell-Weil sieve can be used to prove this.

Suppose now that $G \subset J(K)$ is a subgroup of finite index that is generated by the classes of
the differences of all known $K$-rational points on $X$.
Quadratic Chabauty requires the computation of $p$-adic integrals and the current
implementation requires this to take place over $\Q_p$, as opposed to an extension field. 
Since, for the combination with the Mordell-Weil sieve, we need to do this for several
primes of good ordinary reduction, we would like to work directly with the group $G$, and not with its saturation
at small primes.
This is possible using the approach introduced at the end of the previous subsection.

See \cite[$\S\S 6-8$]{BBM:Computing} for more details about fine-tuning
the Mordell-Weil sieve when used in combination with quadratic Chabauty; after some slight
modifications the statements given there remain valid in the situation considered here.

\subsection*{Computing $X_0(37)(\Q(i))$}\label{sec:x037qi}
We use the Mordell-Weil sieve, combined with the $p$-adic methods described in the main
text, to compute the set of $K$-rational points on $X_0(37)$, where $K = \Q(i)$. 
Recall from Section~\ref{ex2} that $X:y^2 =x^6 -9x^4 +11x^2 +37$
is a model for $X_0(37)$ over $K$ and that we have $r= \rk(J/K) = 2$.  
Note that
$$ \mathcal{A} := \{(\pm 2, \pm 1),(\pm i:  \pm 4), \infty^{\pm} \}\subset X(K)\,,$$
where the sign of $Y/X$ is $\pm$ for  $\infty^{\pm}$;
we want to show that we actually have equality.
We use the point $(2,1)$ as our base point for the Abel-Jacobi map $\iota:X
\hookrightarrow J$.

The subgroup $G$ of $J(K)$ generated by the differences of points in $\mathcal{A}$
can be generated by $P$, $Q$ and $R$, where
$P= [(-2, -1) - (2, -1)]$ and  $Q = [(2, 1)- (i, -4)]$ are non-torsion points, and $R
= [(-i,4)-(i,4)]$ is a generator of
$J(K)_{\mathrm{tors}}\cong \Z/3\Z$. The group $G$ is not saturated at~2; for instance, we have
\[
  16[\infty^+ - (2,1)] =P - 10Q -R\, .
\]
As discussed in the previous subsection, we nevertheless prefer to work with $G$ directly,
without first saturating at~2.

A detailed account of the computation of the sets $X(K_{\mathfrak{p}_i})_U$ for $i=1,2,3$,  where $\mathfrak{p}_i$ is a prime of $\cO_K$ lying above $p_i$ and $p_1 = 41$, $p_2 =73$ and $p_3=101$, is given in $\S\ref{ex2}$. After taking out the elements corresponding to the known rational points, we get a set of tuples $(\tilde{a}_1,\tilde{a}_2) \in (\Z/M\Z)^2$, where
$M=9\cdot 41^3\cdot 73^2\cdot 101^3$, and a corresponding set $C_M \subset G/MG$
containing ~2099520 residue classes.

To this end, we run the Mordell-Weil sieve (modified as above) with $S$ containing 
primes above $7,13,17,29,101, 109, 199, 239,313,373,677, 757$. 
We finally show that no odd prime divides both $\mathrm{lcm}\left(\{\#J(k_v)\,:\, v \in
S\}\right)$ and $(J(K):G)$; this proves that we indeed have $X(K)= \{(\pm 2: \pm 1)
,(\pm i,  \pm 4), \infty^{\pm} \}$, thus finishing the proof of
Theorem~\ref{x037pts}.


\bibliographystyle{amsplain} 

\begin{thebibliography}{10}

\bibitem{corrx037}
\emph{Personal communication with {H.} {D}aniels and {{\'A}}.
  {L}ozano-{R}obledo, 2015}.

\bibitem{balakrishnan2012non}
J.~S. Balakrishnan, I.~Dan-Cohen, M.~Kim, and S.~Wewers, \emph{A non-abelian
  conjecture of {T}ate-{S}hafarevich type for hyperbolic curves}, Math Ann., to
  appear (2017).

\bibitem{bb:crelle}
J.~\thinspace{}S. Balakrishnan and A.~Besser, \emph{Coleman-{G}ross height
  pairings and the {$p$}-adic sigma function}, J. Reine Angew. Math.
  \textbf{698} (2015), 89--104.

\bibitem{bb:heights}
J.\thinspace{}S. Balakrishnan and A.~Besser, \emph{Computing local $p$-adic
  height pairings on hyperelliptic curves}, IMRN \textbf{2012} (2012), no.~11,
  2405--2444.

\bibitem{balakrishnan2013p}
J.\thinspace{}S. Balakrishnan, A.~Besser, and J.\thinspace{}S. M\"{u}ller,
  \emph{Quadratic {C}habauty: $p$-adic height pairings and integral points on
  hyperelliptic curves}, J. Reine Angew. Math. \textbf{720} (2016), 51--79.

\bibitem{BBM:Computing}
\bysame, \emph{Computing integral points on hyperelliptic curves using
  quadratic {C}habauty}, Math. Comp. \textbf{86} (2017), no.~305, 1403--1434.

\bibitem{qc1code}
J.\thinspace{}S. Balakrishnan and N.~Dogra, \emph{Sage code and data},
  \url{https://github.com/jbalakrishnan/QCI}.

\bibitem{balakrishnan2017quadratic}
\bysame, \emph{Quadratic {C}habauty and rational points {I}{I}: {G}eneralised
  height functions on {S}elmer varieties}, arXiv preprint
  \href{https://arxiv.org/abs/1705.00401}{arXiv:1705.00401} (2017).

\bibitem{BDMTV}
J.\thinspace{}S. Balakrishnan, N.~Dogra, J.\thinspace{}S. Muller, J.~Tuitman,
  and J.~Vonk, \emph{Explicit {C}habauty-{K}im for the split {C}artan modular
  curve of level 13}, arXiv preprint
  \href{https://arxiv.org/abs/1711.05846}{arXiv:1711.05846} (2017).

\bibitem{Bes99a}
A.~Besser, \emph{The $p$-adic height pairings of {C}oleman-{G}ross and of
  {N}ekov{\'a}{\v r}}, Number Theory, CRM Proceedings \& Lecture Notes,
  vol.~36, American Mathematical Society, 2004, pp.~13--25.

\bibitem{bilu-parent}
Y.~Bilu and P.~Parent, \emph{Serre's uniformity problem in the split {Cartan}
  case}, Ann. of Math. (2) \textbf{173} (2011), no.~1, 569--584.

\bibitem{bilu-parent-rebolledo}
Y.~Bilu, P.~Parent, and M.~Rebolledo, \emph{Rational points on $x_0^+(p^r)$},
  Ann. Inst. Fourier \textbf{63} (2013), no.~3, 957--984.

\bibitem{bloch-kato}
S.~Bloch and K.~Kato, \emph{\protect{${L}$}-functions and \protect{T}amagawa
  numbers of motives}, The Grothendieck Festschrift, Vol. \protect{I},
  Birkh\"auser Boston, Boston, MA, 1990, pp.~333--400.

\bibitem{Bruin-Stoll:MWSieve}
N.~Bruin and M.~Stoll, \emph{The {M}ordell-{W}eil sieve: proving non-existence
  of rational points on curves}, LMS J. Comput. Math. \textbf{13} (2010),
  272--306.

\bibitem{chabauty}
C.~Chabauty, \emph{Sur les points rationnels des courbes alg\'ebriques de genre
  sup\'erieur \`a l'unit\'e}, C. R. Acad. Sci. Paris \textbf{212} (1941),
  882--885.

\bibitem{kim:coates}
J.~Coates and M.~Kim, \emph{Selmer varieties for curves with {CM} {J}acobians},
  Kyoto J. Math. \textbf{50} (2010), no.~4, 827--852.

\bibitem{coleman:chabauty}
R.\thinspace{}F. Coleman, \emph{Effective {C}habauty}, Duke Math. J.
  \textbf{52} (1985), no.~3, 765--770.

\bibitem{coleman-gross}
R.\thinspace{}F. Coleman and B.\thinspace{}H. Gross, \emph{$p$-adic heights on
  curves}, Algebraic Number Theory -- in honor of K. Iwasawa, Advanced Studies
  in Pure Mathematics, vol.~17, 1989, pp.~73--81.

\bibitem{darmon2012iterated}
H.~Darmon, V.~Rotger, and I.~Sols, \emph{Iterated integrals, diagonal cycles
  and rational points on elliptic curves}, Publications math\'ematiques de
  {B}esan\c{c}on. {A}lg\`ebre et th\'eorie des nombres, 2012/2, Publ. Math.
  Besan\c{c}on Alg\`ebre Th\'eorie Nr., vol. 2012, Presses Univ.
  Franche-Comt\'e, Besan\c{c}on, 2012, pp.~19--46.

\bibitem{deligne1989groupe}
P.~Deligne, \emph{Le groupe fondamental de la droite projective moins trois
  points}, Galois groups over $\mathbb{Q}$, Publ. MRSI, no.~16, 1989,
  pp.~79--297.

\bibitem{deligne2005groupes}
P.~Deligne and A.~B. Goncharov, \emph{Groupes fondamentaux motiviques de {T}ate
  mixte}, Ann. Sci. \'Ecole Norm. Sup. (4) \textbf{38} (2005), no.~1, 1--56.

\bibitem{dogra:thesis}
N.~Dogra, \emph{{T}opics in the theory of {S}elmer varieties}, Oxford Ph.D.
  thesis (2015).

\bibitem{ellenberg2017rational}
J.\thinspace{}S. Ellenberg and D.\thinspace{}R. Hast, \emph{Rational points on
  solvable curves over $\mathbf{Q}$ via non-abelian {C}habauty}, arXiv preprint
  arXiv:1706.00525 (2017).

\bibitem{faltings-finiteness}
G.~Faltings, \emph{\protect{Endlichkeitss\"atze f\"ur} abelsche
  \protect{V}ariet\"aten \"uber \protect{Z}ahlk\"orpern}, Invent. Math.
  \textbf{73} (1983), no.~3, 349--366.

\bibitem{flynn1999finding}
E.\thinspace{}V. Flynn and J.\thinspace{}L. Wetherell, \emph{Finding rational
  points on bielliptic genus 2 curves}, Manuscripta Math. \textbf{100} (1999),
  no.~4, 519--533.

\bibitem{fontaine1994perrin}
J.-M. Fontaine and B.~Perrin-Riou, \emph{Autour des {C}onjectures de {B}loch et
  {K}ato: {C}ohomologie {G}aloisienne et valeurs de fonctions {L} in
  {M}otives}, Proc. Sympos. Pure Math, vol.~55, 1994, pp.~599--706.

\bibitem{fulton2013intersection}
W.~Fulton, \emph{Intersection theory}, vol.~2, Springer Science \& Business
  Media, 2013.

\bibitem{grothendieck288groupes}
A.~Grothendieck, P.~Deligne, N.~Katz, et~al., \emph{Groupes de monodromie en
  g{\'e}om{\'e}trie alg{\'e}brique, s{\'e}minaire de g{\'e}om{\'e}trie
  alg{\'e}brique du {B}ois {M}arie 1967-1969 ({SGA 7 I, II})}, Lecture Notes in
  Mathematics \textbf{288}, 340.

\bibitem{kaenders2001mixed}
R.~H. Kaenders, \emph{The mixed {H}odge structure on the fundamental group of a
  punctured {R}iemann surface}, Proc. Amer. Math. Soc. \textbf{129} (2001),
  no.~5, 1271--1281.

\bibitem{kim:siegel}
M.~Kim, \emph{The motivic fundamental group of
  {$\mathbf{P}^1\setminus\{0,1,\infty\}$} and the theorem of {S}iegel}, Invent.
  Math. \textbf{161} (2005), no.~3, 629--656.

\bibitem{kim:chabauty}
\bysame, \emph{The unipotent {A}lbanese map and {S}elmer varieties for curves},
  Publ. Res. Inst. Math. Sci. \textbf{45} (2009), no.~1, 89--133.

\bibitem{kim2012tangential}
\bysame, \emph{Tangential localization for {S}elmer varieties}, Duke Math. J.
  \textbf{161} (2012), no.~2, 173--199.

\bibitem{kim2008component}
M.~Kim and A.~Tamagawa, \emph{The {$l$}-component of the unipotent {A}lbanese
  map}, Math. Ann. \textbf{340} (2008), no.~1, 223--235.

\bibitem{liu2002algebraic}
Q.~Liu, \emph{Algebraic geometry and arithmetic curves}, Oxford Graduate Texts in Mathematics \textbf{6} (2002).

\bibitem{mazur-stein-tate}
B.~Mazur, W.~Stein, and J.~Tate, \emph{Computation of {$p$}-adic heights and
  log convergence}, Doc. Math. (2006), no.~Extra Vol., 577--614 (electronic).

\bibitem{Nek93}
J.~Nekov{\'a}{\v{r}}, \emph{On $p$-adic height pairings}, S\'eminaire de
  Th\'eorie des Nombres, Paris, 1990--91, Birkh\"auser Boston, Boston, MA,
  1993, pp.~127--202.

\bibitem{olsson2011towards}
M.\thinspace{}C. Olsson, \emph{Towards non-abelian {$p$}-adic {H}odge theory in
  the good reduction case}, Mem. Amer. Math. Soc. \textbf{210} (2011), no.~990,
  vi+157.

\bibitem{PSS:Twists}
B.~Poonen, E.~F. Schaefer, and M.~Stoll, \emph{Twists of {$X(7)$} and primitive
  solutions to {$x^2+y^3=z^7$}}, Duke Math. J. \textbf{137} (2007), no.~1,
  103--158.

\bibitem{raynaud19941}
M.~Raynaud, \emph{1-motifs et monodromie g\'eom\'etrique}, Ast\'erisque (1994),
  no.~223, 295--319, P{\'e}riodes $p$-adiques (Bures-sur-Yvette, 1988).

\bibitem{Scharaschkin:Thesis}
V.~Scharaschkin, \emph{Local-global problems and the {B}rauer-{M}anin
  obstruction}, ProQuest LLC, Ann Arbor, MI, 1999, Thesis (Ph.D.)--University
  of Michigan.

\bibitem{scholl1991height}
A.~J. Scholl, \emph{Height pairings and special values of {$L$}-functions},
  Motives ({S}eattle, {WA}, 1991), Proc. Sympos. Pure Math., vol.~55, Amer.
  Math. Soc., Providence, RI, 1994, pp.~571--598.

\bibitem{serre:gc}
J.-P. Serre, \emph{Galois cohomology}, Springer-Verlag, Berlin, 1997,
  Translated from the French by Patrick Ion and revised by the author.

\bibitem{Siksek:ECNF}
S.~Siksek, \emph{Explicit {C}habauty over number fields}, Algebra Number Theory
  \textbf{7} (2013), no.~4, 765--793.

\bibitem{silverman:aec2}
H.~H Silverman, \emph{Advanced topics in the arithmetic of elliptic curves}, Springer-Verlag, (1994).

\bibitem{silverman1988computing}
J.~H. Silverman, \emph{Computing heights on elliptic curves}, Math. Comp.
  \textbf{51} (1988), no.~183, 339--358.

\bibitem{sage}
W.\thinspace{}A. Stein et~al., \emph{{S}age {M}athematics {S}oftware ({V}ersion
  8.0)}, The Sage Development Team, 2017, {\tt http://www.sagemath.org}.

\bibitem{waldschmidt2011p}
M.~Waldschmidt, \emph{On the {$p$}-adic closure of a subgroup of rational
  points on an abelian variety}, Afrika Matematika \textbf{22} (2011), no.~1,
  79--89.

\end{thebibliography}
\providecommand{\bysame}{\leavevmode\hbox to3em{\hrulefill}\thinspace}
\providecommand{\MR}{\relax\ifhmode\unskip\space\fi MR }
\providecommand{\MRhref}[2]{%
  \href{http://www.ams.org/mathscinet-getitem?mr=#1}{#2}
}
\providecommand{\href}[2]{#2}

\end{document}